\documentclass[reqno]{amsart}

\usepackage[dvipsnames]{xcolor}

\usepackage{iftex}
\ifPDFTeX 
  \usepackage[utf8]{inputenc}
  \usepackage[T1]{fontenc}
\else 
  \usepackage{fontspec}
\fi

\usepackage{geometry}
\usepackage{lmodern}
\usepackage{dsfont}
\usepackage{enumitem}
\usepackage{amssymb}

\usepackage{pgf,tikz,pgfplots}
\pgfplotsset{compat=1.18}
\usetikzlibrary{arrows.meta, arrows, decorations.pathreplacing}

\usepackage[pdfencoding=auto]{hyperref}
\hypersetup{
    colorlinks,
    citecolor=green,
    filecolor=black,
    linkcolor=blue,
    urlcolor=black
}

\usepackage{empheq}

\usepackage{ifthen}
\newboolean{details}
\newcommand{\details}[1]{\ifthenelse{\boolean{details}}{\textcolor{blue}{#1}}{}}

\numberwithin{equation}{section}

\newtheorem{theorem}{Theorem}[section]
\newtheorem{lemma}[theorem]{Lemma}
\newtheorem{proposition}[theorem]{Proposition}
\newtheorem{corollary}[theorem]{Corollary}
\newtheorem{Assumption}{Assumption}

\theoremstyle{definition}
\newtheorem{definition}[theorem]{Definition}

\newtheorem{Remark}{Remark}

\newcommand{\myParagraph}[1]{\medskip\noindent \textbf{#1.}}
\newcommand{\lla}{\left\langle}
\newcommand{\rra}{\right\rangle}

\newcommand{\lbr}{\left\lbrace}
\newcommand{\rbr}{\right\rbrace}

\newcommand{\ldb}{[\![}
\newcommand{\rdb}{]\!]}
\newcommand{\argmin}{\mathrm{argmin}\,}

\newcommand{\overbar}[1]{\mkern 1.5mu\overline{\mkern-1.5mu#1\mkern-1.5mu}\mkern 1.5mu}

\newcommand{\dt}{\,\partial_t\, }
\newcommand{\Dt}{{\Delta t}}
\newcommand{\dx}{\,\partial_x\, }
\newcommand{\Dx}{{\Delta x}}

\newcommand{\Dv}{{\Delta v}}

\newcommand{\dd}{\,\mathrm{d}}

\newcommand{\ep}{\varepsilon}
\newcommand{\e}{\,\mathrm{e}}

\newcommand{\N}{\mathbb{N}}
\newcommand{\Z}{\mathbb{Z}}
\newcommand{\R}{\mathbb{R}}
\newcommand{\TT}{\mathbb{T}}

\newcommand{\M}{\mathcal{M}}

\newcommand{\wt}{\widetilde{w}}
\newcommand{\wb}{\overbar{w}}

\newcommand{\gamb}{\overbar{\gamma}}
\newcommand{\gamt}{\widetilde{\gamma}}

\newcommand{\Ab}{\overbar{A}}
\newcommand{\At}{\widetilde{A}}

\newcommand{\tin}{{\mathrm{in}}}

\title[Numerical analysis of a kinetic equation with a nonlocal H-J limit]{Numerical analysis of the large deviation regime of a kinetic equation with a nonlocal Hamilton-Jacobi limit}

\author{Hélène Hivert}
\address[Hélène Hivert]{Univ. Rennes, Inria, Géosciences Rennes - UMR 6118, IRMAR - UMR 6625, F-35000 Rennes, France}

\author{Tino Laidin}
\address[Tino Laidin]{Univ Brest, CNRS UMR 6205, Laboratoire de Mathématiques de Bretagne Atlantique, F-29200 Brest, France}

\email{tino.laidin@univ-brest.fr}
\email{helene.hivert@inria.fr}

\begin{document}

\setboolean{details}{false}

\begin{abstract}
    We develop and study an asymptotic-preserving (AP) numerical scheme for a linear kinetic equation in a large deviation regime. After applying a Hopf-Cole transform to the distribution function, the system exhibits the behavior of rare events, which in the limit is governed by a non-standard, nonlocal Hamilton-Jacobi equation, as identified in [E. Bouin et al., J. Lond. Math. Soc., II. Ser., 2023].

    The proposed scheme efficiently handles the stiffness introduced by scaling, with a computational cost that remains uniform with respect to the small parameter. It takes advantage of the conservation properties of the original kinetic model to overcome the numerical challenges posed by stiffness. The scheme satisfies a discrete maximum principle, preserves equilibrium states, and correctly captures the asymptotic limit, recovering the viscosity solution of the limit nonlocal Hamilton-Jacobi equation.

    As the limit problem is non-standard, convergence results from the literature are not directly applicable. We introduce new analytical tools based on a discrete representation formula that links the numerical scheme with the continuous setting. This allows us to prove the convergence and establish key structural properties of the method. Numerical tests support the analysis and illustrate the robustness of the scheme and the original behavior of the limit system. \\[1em]

\end{abstract}

\maketitle

\section{Introduction}
In this paper, we introduce and analyze a new Asymptotic Preserving (AP) scheme for a linear kinetic equation under a scaling tailored to capture rare events in the dynamics. Specifically, we consider the linear BGK equation for the unknown distribution function $f(t,x,v)$:
\begin{equation}\label{eq:KinRelax}
    \dt f(t,x,v) + v\cdot\nabla_x f(t,x,v) = \rho(t,x)\M(v) - f(t,x,v),\quad (t,x,v)\in\R^+\times\R^{d_x}\times\R^{d_v},
\end{equation}
where $d_x, d_v \in \N$ denote the spatial and velocity dimensions. The macroscopic density is given by $\rho(t,x) = \langle f(t,x,v)\rangle \coloneq \int_{\R^{d_v}} f(t,x,v) \, \dd v$, and the equilibrium distribution $\M$ is defined as
\begin{equation*}
    \M(v) = \frac{1}{(2\pi)^{d_v/2}} \exp\left(-\frac{|v|^2}{2}\right).
\end{equation*}
The equation is supplemented with an initial data $f_\tin$, which will be specified later. The function $f(t,x,v)$ represents the probability density of finding an individual at time $t$, position $x$, and moving at velocity $v$. The relaxation operator $\rho\M - f$ models interactions that cause individuals to change velocity: they are removed according to the distribution $f$ and reintroduced with a Maxwellian profile $\rho\M$ while preserving mass. This corresponds to a stochastic process in which individuals randomly change velocity following a velocity jump process given by $\rho\M - f$.

In this work, we consider the rescaling, introduced in \cite{BouinCalvezGrenierNadin2023}: 
\begin{equation}\label{eq:scaling}
    \left(\frac{t}{\ep}, \frac{x}{\ep^{3/2}},\frac{v}{\ep^{1/2}}\right).
\end{equation}
Under this scaling, equation \eqref{eq:KinRelax} then reads
\begin{equation}\label{eq:kineticScaled}\tag{$\mathcal{P}^\ep$}
    \dt f^\ep(t,x,v) + v\cdot\nabla_x f^\ep(t,x,v) = \frac{1}{\ep}\left(\rho^\ep(t,x)\M^\ep(v) - f^\ep(t,x,v)\right),\quad (t,x,v)\in\R^+\times\R^{d_x}\times\R^{d_v},
\end{equation}
where $\M^\ep$ is the centered Gaussian of variance $\ep$: 
\begin{equation*}
    \M^\ep(v)=\frac{1}{(2\ep \pi)^{d_v/2}}\exp\left(-\frac{|v|^2}{2\ep }\right).
\end{equation*}
Equation \eqref{eq:KinRelax} has been extensively studied in both the hydrodynamic \cite{BardosGolseLevermore1991, SaintRaymond2009} and diffusive \cite{BensoussanLionsPapanicolaou1979, GoudonPoupaudDegond2000} regimes, but its analysis under the scaling \eqref{eq:scaling} is a recent research subject. One motivation for considering this scaling comes from the following kinetic reaction-transport equation, where individuals also undergo a reproduction process:
\begin{equation}\label{eq:KineticKPP}
    \dt f + v\cdot\nabla_x f = \rho\M - f + r\rho(\M-f),
\end{equation}
where $r>0$ is a reproduction rate. Equation \eqref{eq:KineticKPP} can be viewed as a kinetic analogue of the well-known Fisher-KPP equation. In \cite{BouinCalvezNadin2015} it was shown that \eqref{eq:KineticKPP} exhibits accelerated front propagation, with a velocity that grows like $t^{3/2}$. Due to this acceleration, the classical diffusive scaling is not appropriate to capture the motion of the fronts, which motivates the use of scaling \eqref{eq:scaling}.

The main challenge in analyzing the limit $\ep \to 0$ in \eqref{eq:kineticScaled} lies in the interplay between the stiffness of the right-hand side and the concentration in velocity induced by the resampling through the distribution $\rho^\ep \M^\ep$. To accurately capture the asymptotic behavior as $\ep \to 0$, a typical strategy is to study the limit of the Hopf-Cole transform of the unknown. It is defined as 
\begin{equation}\label{eq:HopfCole}
    \varphi^\ep(t,x,v)=-\ep\ln(f^\ep(t,x,v)).
\end{equation}
This logarithmic transform has, for example, been used in \cite{DiekmannJabinMischlerPerthame2005, BarlesPerthame2007, BarlesMirrahimiPerthame2009, LorzMirrahimiPerthame2011,BouinCalvez2012,CalvezHendersonMirrahimiTuranovaDumont2022} for the study of adaptive dynamics where concentration phenomena typically occur. Another way to interpret \eqref{eq:HopfCole} is that it captures rare events occurring in the tails of the distribution $f^\ep$, which becomes increasingly concentrated as $\ep \to 0$. These rare events are precisely what drive the accelerating propagation fronts in \eqref{eq:KineticKPP} \cite{BouinCalvezNadin2015}. 

The analysis of \eqref{eq:kineticScaled} follows a line of reasoning similar to that used for the classical Fisher-KPP equation \cite{EvansSouganidis1989, BarlesEvansSouganidis1990, CrandallIshiiLions1992}, based on the so-called geometric optics approximation \cite{Freidlin1986, EvansSouganidis1989}. In \cite{BouinCalvezGrenierNadin2023}, a parallel is drawn with the heat equation under hyperbolic scaling, where the Hopf-Cole transform is shown to converge locally uniformly to the viscosity solution of a Hamilton-Jacobi equation (see, e.g., \cite{CrandallIshiiLions1992, Barles2013}). In the case of \eqref{eq:kineticScaled} the authors showed that $\varphi^\ep$ converges, as $\ep\to0$, towards a function $\varphi(t,x,v)$ solution to:
\begin{equation}\label{eq:Vari}\tag{$\mathcal{P}^0$}
    \left\lbrace\begin{aligned}
        &\max\left(\dt \varphi(t,x,v)+v\cdot\nabla_x \varphi(t,x,v) -1,\, \varphi(t,x,v)-\underset{v\in\R^{d_v}}{\min}\lbr \varphi(t,x,v)\rbr - \frac{|v|^2}{2}\right)=0,\\
        &\dt\left( \underset{v\in\R^{d_v}}{\min}\lbr \varphi\rbr(t,x)\right)\leq 0,\\
        &\dt\left( \underset{v\in\R^{d_v}}{\min}\lbr \varphi\rbr(t,x)\right)= 0\quad \text{if}\quad \underset{v\in\R^{d_v}}{\argmin}\lbr\varphi(t,x,v)\rbr = \lbr0\rbr,\quad (t,x,v)\in\R^+\times\R^{d_x}\times\R^{d_v}.
    \end{aligned}\right.
\end{equation}
For simplicity, the initial data of \eqref{eq:kineticScaled} is now assumed to be of the form
\begin{equation*}
    f_\tin^\ep(x,v)=\exp\left(-\frac{\varphi_\tin(x,v)}{\ep}\right)
\end{equation*}
so that \eqref{eq:Vari} is supplemented with an initial data $\varphi_\tin$ such that $\varphi_\tin-|v|^2/2$ is bounded and satisfies Lipschitz regularity. The precise assumptions are stated later. In addition, we assume that $\varphi_\tin$ does not depend on $\ep$, but this assumption can be relaxed by assuming local uniform convergence in $\ep$.

They are several key features to \eqref{eq:Vari}. First, contrary to standard small-scale asymptotics of kinetic equations, the limit still depends on the velocity variable $v$. Secondly, the limit remains nonlocal in velocity because of the term $\min_{v \in \R^{d_v}} \lbr \varphi \rbr(t,x)$. Thirdly, although the system at $\ep > 0$ consists of a single equation for the distribution, the limiting system becomes coupled through a condition involving the minimum in velocity of the unknown. Lastly, a remarkable feature of \eqref{eq:kineticScaled} is that the limits $\ep \to 0$ and $t \to \infty$ do not commute: for fixed $\ep > 0$, the system relaxes to a constant equilibrium as $t \to \infty$ \cite{DolbeaultMouhotSchmeiser2015}, whereas in the limit $\ep = 0$, it retains a memory of the initial data even at large times. A schematic representation of this behavior is shown in Figure~\ref{fig:LTBSchematic}, where the spatial amplitude of the solution is plotted as a function of time. This quantity effectively illustrates the phenomenon: for $\ep$ large, the amplitude converges to $0$, whereas in our setting, as $\ep \to 0$, it converges toward a no-zero constant.
\begin{figure}
    \begin{tikzpicture}[>=Latex,line cap=round,line join=round, scale=0.8]

        \begin{semilogyaxis}[
            width=10cm,
            height=7cm,
            axis lines=left,          
            xlabel={Time},
            ylabel={$(\max_x- \min_x)\lbr\varphi^\ep(t,x,\cdot)\rbr$},
            ymin=1e-10, ymax=1,
            xmin=0, xmax=500,
            ytick={1e0,1e-2,1e-4,1e-6,1e-8,1e-10},
            every axis plot/.append style={very thick,no markers},
            enlargelimits=false,
        ]

        \addplot+[smooth,very thick,teal] coordinates {
            (0,1) (100,0.3) (200,0.2) (500,0.2)
        };

        \addplot+[smooth,very thick,orange,dashed] coordinates {
            (0,1) (200,0.1) (500,0.01)
        };

        \addplot+[smooth,very thick,violet,dashdotted] coordinates {
            (0,1) (100,0.01) (300,1e-5) (500,1e-8)
        };

        \addplot+[smooth,very thick,red,dotted] coordinates {
            (0,1) (50,1e-3) (150,1e-7) (250,1e-10)
        };

        \end{semilogyaxis}

        \draw[->,very thick,dashed]
        (1.5,1.5) .. controls (3.4,2.) and (3.5,2.9) .. (4,5)
        node[midway,right]{\Large $\ep \to 0$};

    \end{tikzpicture}
    \caption{Schematic depiction of the long-time spatial amplitude at a fixed velocity of the solution to \eqref{eq:kineticScaled}  as $\ep\to0$.}
    \label{fig:LTBSchematic}
\end{figure}
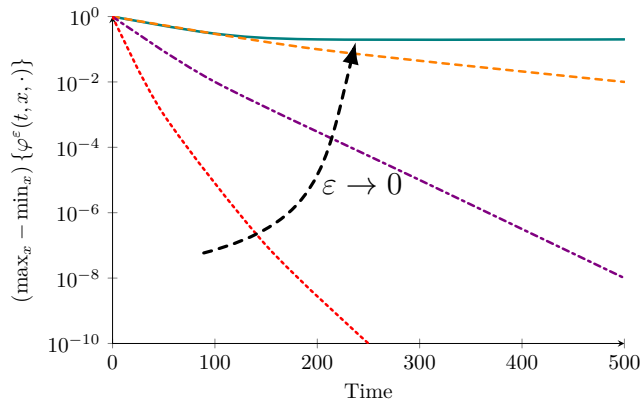

The limit system \eqref{eq:Vari} is referred to as a nonlocal Hamilton–Jacobi equation, both by analogy with the classical use of the Hopf–Cole transform for parabolic equations and because it admits a unique viscosity solution, characterized via the method of sub- and super-solutions. It is well known that solutions to Hamilton-Jacobi equations are, at best, expected to have Lipschitz regularity. This low regularity presents additional challenges both for the analysis and from a numerical perspective.

It was shown in \cite{BouinCalvezGrenierNadin2023} that \eqref{eq:Vari} admits a representation formula that characterizes the unique viscosity solution. The rigorous notion of viscosity solutions for \eqref{eq:Vari} is also developed in that work, based on appropriate definitions of sub- and super-solutions. We refer the reader to \cite{BouinCalvezGrenierNadin2023} for a detailed discussion; the explicit form of the representation formula will be recalled later. For standard Hamilton-Jacobi equations, there is a duality between the Eulerian viewpoint \cite{EvansSouganidis1984, EvansSouganidis1989, Fathi2012}, where the dynamics are described by a PDE as in \eqref{eq:Vari}, and the Lagrangian perspective, which focuses on trajectories, as in a representation formula. In the latter case, minimizing trajectories are typically smooth. In contrast, the trajectories associated with the system \eqref{eq:Vari} are piecewise linear in position and piecewise constant in velocity \cite{BouinCalvezGrenierNadin2023}. The construction of AP schemes offers a promising perspective for extending the analysis of \cite{BouinCalvezGrenierNadin2023} to other systems, such as neutron transport \cite{DuderstadtMartin1979}, where the accurate treatment of rare events is equally critical.

The design of AP schemes for kinetic equations dates back to \cite{Jin1999, Klar1999}. The key idea is to build a scheme for the $\ep > 0$ problem that, in the limit $\ep \to 0$, converges towards a consistent discretization of the asymptotic model. Another fundamental feature of AP schemes is their ability to take time steps independent of the scaling parameter, avoiding a restrictive condition of the type $\Dt \leq \ep \Dx$, commonly encountered in classical schemes. For an overview of such methods in the kinetic context, we refer to \cite{DimarcoPareschi2014ActaNumerica, JinReview2022}. However, these approaches typically work directly on the distribution function and do not involve the Hopf-Cole transform \eqref{eq:HopfCole}. 

In the context of biological systems, AP schemes based on logarithmic transformations have been proposed in \cite{CalvezHivertYoldas2023, GaudeulHivert2024, AlmeidaPertRuan2022}. To the best of our knowledge, the only instance of such a strategy applied to kinetic equations is found in \cite{Hivert2018}, where \eqref{eq:KineticKPP} is treated in the case of bounded velocities to study front propagation. In this work the author considered a multiplicative micro-macro decomposition that allows to treat a macroscopic part, that does not depend on velocity, and a perturbative part separately. Micro-macro schemes are typically additive \cite{BennouneLemouMieussens2008,LemouMieussens2008, LiuMieusens2010}, but the logarithmic transform prompted the multiplicative one allowing the construction of an AP scheme for \eqref{eq:KineticKPP}. 

A major difference between the study in \cite{Hivert2018} and ours is the fact that they considered bounded velocities. Under this modeling assumption, it is sufficient to rescale \eqref{eq:KineticKPP} under a hyperbolic scaling $\left(\frac{t}{\ep},\frac{x}{\ep},v\right)$. The asymptotic limit $\ep\to0$ of this problem was studied in \cite{BouinCalvez2012,Bouin2015,Caillerie2017} and its front propagation property showed in \cite{BouinCalvezNadin2015}. Because of the bounded velocities, the asymptotic limit does not depend on the velocity variable which is the typical behavior that prompted the use of micro-macro decomposition. For unbounded velocities, which is the novelty of our study compared to \cite{Hivert2018}, it is not the case and such decomposition is not relevant anymore.

The standard approach to discretize Hamilton-Jacobi equations is based on finite differences, for which robust analytical tools have been developped to study convergence towards viscosity solutions \cite{CrandallLions1984, Souganidis1985, BarlesSouganidis1991}. These tools were adapted in \cite{Hivert2018}, where the limiting system admits a classical Hamilton-Jacobi formulation, albeit with a Hamiltonian defined implicitly. However, in our setting, recasting \eqref{eq:Vari} in Hamiltonian form remains an open problem, necessitating an alternative approach.

The goal of this paper is therefore twofold. First, we construct a numerical scheme capable of capturing the asymptotic regime $\ep \to 0$ in \eqref{eq:kineticScaled} as well as the key features mentioned above. Unlike in \cite{Hivert2018}, the nature of the limit system leads us to develop a new class of AP schemes, based on the Duhamel representation of \eqref{eq:kineticScaled} and a kind of ``meso-macro'' framework. In this approach, a macroscopic component evolves along the function $\varphi^\ep$. Second, beyond the AP property, we aim to ensure that the limit scheme converges, when the discretization parameters tend to $0$, towards the viscosity solution of \eqref{eq:Vari}. To this end, we make use of the representation formula of \eqref{eq:Vari}, along with a discrete analogue, to establish convergence to the correct solution.

At a deeper level, AP schemes are not just computational tools. They can also guide analytical understanding. While they were initially introduced to efficiently handle stiffness and capture asymptotic limits, their structure can serve as a bridge to the limiting problem itself. In our case, directly discretizing the variational formulation \eqref{eq:Vari} is far from straightforward. By instead designing a scheme that preserves the key features of the scaled kinetic equation \eqref{eq:kineticScaled}, we obtain, in the limit $\ep \to 0$, a discretization that naturally approximates the viscosity solution of \eqref{eq:Vari}. In addition, the derivation of a discrete representation formula for \eqref{eq:Vari} not only reinforces the consistency of our scheme but also opens a path for extending this approach to other models with similar variational structures.

\myParagraph{Outline of the paper}

In Section~\ref{sec:ConstructionAndMR}, we begin by introducing a discretization \eqref{eq:schemeMuPhiFull} of \eqref{eq:kineticScaled} and presenting the main results of the paper. These are the AP property of \eqref{eq:schemeMuPhiFull} (Theorem~\ref{thm:AP}) and the fact that the asymptotic limit \eqref{eq:VariDiscFull} captures the viscosity solution of \eqref{eq:Vari} (Theorem~\ref{thm:ConvVisc}). The theorems are proven in the following sections. Theorem~\ref{thm:AP} is proven in Section~\ref{sec:AP}. It relies on several stability properties of \eqref{eq:schemeMuPhiFull}. Regarding Theorem~\ref{thm:ConvVisc}, the main difficulty lies in the fact that classical techniques are not applicable to establish the convergence of \eqref{eq:VariDiscFull}. As in \cite{BouinCalvezGrenierNadin2023}, where non-standard features of \eqref{eq:Vari} prompted the introduction of new notions of viscosity solutions, we reformulate \eqref{eq:VariDiscFull} in Section~\ref{sec:AsymptProp} to connect the discrete setting with its continuous counterpart. Specifically, we analyze stability properties of \eqref{eq:VariDiscFull} in Section~\ref{subsec:Variprop}, and derive a discrete representation formula in Section~\ref{subsec:DerivRep}. The full convergence result of Theorem~\ref{thm:ConvVisc} is proven in Section~\ref{sec:ConvVisc}. Finally, Section~\ref{sec:NumRes} presents numerical experiments that both confirm the theoretical results and investigate the long time behavior of \eqref{eq:Vari}.

\section{Construction of the numerical scheme and main results}\label{sec:ConstructionAndMR}

\subsection{Construction of the numerical scheme}
We now propose the derivation of an asymptotic preserving scheme for \eqref{eq:kineticScaled}. The main ideas of the discretization are as follows. First, the Duhamel formulation of \eqref{eq:kineticScaled} is considered. Then, we adopt a meso-macro approach by introducing the Hopf-Cole transform of the distribution $f^\ep$. This method allows for an explicit treatment of the stiff term. We proceed in three steps, namely the time, velocity, and space discretizations, which we detail below.

\myParagraph{Time discretization} 
Let $N_t\in\N^*$. We set $(0,T]=\cup^{N_t}_{k= 0} (t^k,t^{k+1}]$ with $t^n=n\Dt$ and $t^{N_t}=T$. Our goal is to find a reformulation of \eqref{eq:kineticScaled} that allows for the derivation of a numerical scheme that is stable in the limit $\ep\to0$. To this achieve this, we start by considering its Duhamel formulation at time $t^{n+1}=t^n+\Dt$:
\begin{equation}\label{eq:Duhamel}
    \begin{aligned}
        f^\ep(t^n+\Dt,x,v) &= f^\ep(t^n,x-\Dt\,v,v)\e^{-\Dt/\ep}\\
        &\qquad - \frac{1}{\ep c^\ep}\e^{-\Dt/\ep}\int_0^{\Dt}\e^{s/\ep}\e^{-|v|^2/2\ep}\rho^\ep(t^n+s, x+(s-\Dt)v)\dd s,
    \end{aligned}
\end{equation}
with
\begin{equation}\label{eq:cepsCont}
    c^\ep=\lla\exp\left(-\frac{|v|^2}{2\ep }\right)\rra.
\end{equation}
In the spirit of \cite{Hivert2018}, we introduce $\varphi^\ep$ and $\mu^\ep$ defined respectively as the Hopf-Cole transforms of $f^\ep$ and $\rho^\ep$. Specifically,
\begin{equation}\label{eq:rhomu}
    f^\ep(t,x,v)=\e^{-\frac{\varphi^\ep(t,x,v)}{\ep}},\quad\rho^\ep(t,x)=\e^{-\frac{\mu(t,x)^\ep}{\ep}}\quad\text{and}\quad \lla \e^{-\frac{\varphi^\ep(t,x,v)}{\ep}} \rra =\e^{-\frac{\mu^\ep(t,x)}{\ep}}.
\end{equation}
The semi-discrete scheme in time arises from selecting a quadrature rule to approximate the integral in \eqref{eq:Duhamel}. Due to the stiffness in $\mu^\ep$, we adopt a right-point quadrature approximation. Let $\varphi^{\ep,n}(x,v)$ denote the approximation of $\varphi^\ep(t^n,x,v)$, and similarly, $\mu^{\ep,n}(x)$ approximates $\mu^{\ep}(t^n,x)$. The term $\e^{-\mu(t^n+s, x+ (s-\Dt v))/\ep}$ is approximated by $\e^{-\mu^{n+1}(x)/\ep}$ allowing what remains in the integral to be computed explicitly. One then obtains
\begin{equation}\label{eq:KineticDuhamelPhi}
    \e^{-\varphi^{\ep,n+1}(x,v)/\ep} = \e^{-\varphi^{\ep,n}(x-\Dt v,v)/\ep}\e^{-\Dt/\ep} + (1-\e^{-\Dt/\ep})\frac{1}{c^\ep}\e^{-\frac{|v|^2}{2\ep}}\e^{-\mu^{\ep,n+1}(x)/\ep}.
\end{equation}
Remark that \eqref{eq:KineticDuhamelPhi} is well-posed, despite its implicit expression. Indeed, because of the conservation of mass of the original kinetic equation \eqref{eq:kineticScaled}, $\mu^{\ep,n+1}$ can be obtained explicitly by integrating \eqref{eq:KineticDuhamelPhi} in velocity and using \eqref{eq:rhomu}, along with the definition \eqref{eq:cepsCont} of $c_\ep$. The update of $\mu^{\ep,n+1}$ then reads:
\begin{equation}\label{eq:updateMu}
    \e^{-\mu^{\ep,n+1}(x)/\ep} = \lla  \e^{-\varphi^{\ep,n}(x-\Dt v,v)/\ep}\rra .
\end{equation}

Recalling that our aim is to build an AP scheme for \eqref{eq:kineticScaled}, the asymptotic behavior of \eqref{eq:updateMu} can be understood thanks to the following reformulation:
\begin{equation}\label{eq:updateMuStable}
    \mu^{\ep,n+1}(x) = m^{\ep,n}_\varphi(x) - \ep\ln\left(\lla  \e^{-[\varphi^{\ep,n}(x-\Dt v,v)-m_\varphi(x)]/\ep}\rra \right),
\end{equation}
where
\begin{equation}\label{eq:petitmSemiT}
    m_\varphi^{\ep,n}(x)=\underset{v\in\R^{d_v}}{\min} \lbr\varphi^{\ep,n}(x-\Dt v, v)\rbr.
\end{equation}
Indeed, assuming that one can formally pass to the limit $\ep\to0$ in \eqref{eq:updateMuStable}, the log term vanishes and the limit of $\mu^{\ep,n+1}$ is obtained as the minimum \eqref{eq:petitmSemiT} of the limit of $\varphi^{\ep,n}$. Similarly, we introduce the quantity
\begin{equation}\label{eq:grandMSemiT}
    M^{\ep,n}(x,v)=\min\lbr \varphi^{\ep,n}(x-\Dt v,v)+\Dt,\, \frac{|v|^2}{2}+\mu^{\ep,n+1}(x)\rbr ,
\end{equation}
allowing us to rewrite \eqref{eq:KineticDuhamelPhi} as a semi-discrete in time numerical scheme. For all $x\in\R$, $v\in\R$ and $n\in\ldb1,N_t\rdb$, it reads:
\begin{equation}\label{eq:schemeMuPhiSemiT}\tag{$\mathcal{S}^\ep_{\Dt}$}
    \lbr\begin{aligned}
        \mu^{\ep,n+1}(x) &= m_\varphi^{\ep,n}(x) - \ep \ln\left(
            \lla  
                \e^{-\left[\varphi^{\ep,n}(x-\Dt v,v) - m_\varphi^{\ep,n}(x)\right]/\ep} 
            \rra  \right), \\
        \varphi^{\ep,n+1}(x, v) &= M^{\ep,n}(x,v) 
        - \ep \ln\bigg(
            c^\ep \e^{-\left[\varphi^{\ep,n}(x-\Dt v,v) + \Dt - M^{\ep,n}(x,v)\right]/\ep}\\
        &\quad + \left(1 - \e^{-\Dt/\ep}\right)
            \e^{-\left[\frac{|v|^2}{2} + \mu^{\ep,n+1}(x) - M^{\ep,n}(x,v)\right]/\ep}
        \bigg) + \ep \ln(c^\ep),
    \end{aligned}\right.
\end{equation}
initialized for all $x\in\R$ and all $v\in\R$, by
\begin{equation*}
    \mu^{\ep,0}(x) = \underset{v\in\R^{d_v}}{\min}\lbr\varphi_\tin(x-\Dt-v,v)\rbr \quad\text{and}\quad \varphi^{\ep,0}(x,v) = \varphi_\tin(x,v).
\end{equation*} 

\myParagraph{The Mesh} 
From now on, we restrict the presentation to the $d_x=d_v=1$ setting. In order to fit with the continuous setting, we consider unbounded discrete velocities in $\Dv\Z$ for some $\Dv>0$ and set $v_j=j\Dv$ for $j\in\Z$. We also consider unbounded positions in $\Dx\Z$ for some $\Dx>0$ and set $x_i=i\Dx$ for $i\in\Z$. The low dimensional setting is chosen to simplify the notations and the extension to higher dimensions will be discussed at the end of Section~\ref{sec:ConvVisc}. 

\myParagraph{Velocity discretization} 
In what follows, we denote by $\varphi^{\ep,n}_j(x)$ an approximation of $\varphi^{\ep,n}(x,v_j)$. The semi-discretization of the first line of \eqref{eq:schemeMuPhiSemiT} is based on the following velocity quadrature. For a sequence $u = \left(u_j\right)_{j\in\Z}$, it is defined by
\begin{equation*}
\langle u \rangle_{\Dv} = \sum_{j\in\Z} u_j\Dv,
\end{equation*}
and is used to approximate integrals in velocity. It also relies on an approximation of \eqref{eq:petitmSemiT}. For each position $x\in\R$, this is given by
\begin{equation}\label{eq:petitmSemiTV}
m^{\ep,n}(x) = \underset{j\in\Z}{\min}\lbr \varphi_{j}^{\ep,n}(x - \Dt v_j) \rbr,
\end{equation}
where the minimum is taken over the discrete velocity set $\left(v_j\right)_{j\in\Z}$. Similarly, to discretize the second line of \eqref{eq:schemeMuPhiSemiT}, we introduce $M_{j}^{\ep,n}(x)$ as an approximation of \eqref{eq:grandMSemiT}:
\begin{equation}\label{eq:grandMSemiTV}
    M_{j}^{\ep,n}(x) = \min\lbr \varphi_{j}^{\ep,n}(x - \Dt v_j) + \Dt,\, \frac{v_j^2}{2} + \mu^{\ep,n+1}(x) \rbr.
\end{equation}
The numerical scheme, semi-discrete in time and velocity for a fixed $\ep > 0$, is then given, for all $x \in \R$, $j \in \Z$, and $n \in \ldb1,N_t\rdb$, by:
\begin{equation}\label{eq:schemeMuPhiSemiTV}\tag{$\mathcal{S}^\ep_{\Dt,\Dv}$}
    \lbr\begin{aligned}
        \mu^{\ep,n+1}(x) &= m^{\ep,n}(x) - \ep \ln\left(
        \lla  
            \e^{-\left[\varphi_{j}^{\ep,n}(x-\Dt v_j) - m^{\ep,n}(x)\right]/\ep} 
        \rra_\Dv  \right),\\
        \varphi_{j}^{\ep,n+1}(x) &= M_{j}^{\ep,n}(x)
        - \ep \ln\bigg(
        c_\Dv^\ep \e^{-\left[\varphi_{j}^{\ep,n}(x-\Dt v_j) + \Dt - M_{j}^{\ep,n}(x)\right]/\ep} \\
        &\qquad+ \left(1 - \e^{-\Dt/\ep}\right) \e^{-\left[\frac{v_j^2}{2} + \mu^{\ep,n+1}(x) - M_{j}^{\ep,n}(x)\right]/\ep} \bigg) + \ep \ln(c_\Dv^\ep).
    \end{aligned}\right.
\end{equation}
For all $x\in\R$ and all $j\in\Z$, it is initialized with,
\begin{equation}\label{eq:initSemiTV}
    \mu^{\ep,0}(x) \coloneq \mu_{\tin}(x) = \underset{j\in\Z}{\min}\lbr\varphi_{\tin}(x-\Dt v_j,v_j)\rbr \quad\text{and}\quad \varphi_j^{\ep,0}(x) = \varphi_\tin(x,v_j),
\end{equation} 
and the constant $c_\Dv^\ep$ in the second line of \eqref{eq:schemeMuPhiSemiTV} is defined by
\begin{equation*}
    c_\Dv^\ep=\sum_j \e^{-\frac{v_j^2}{2\ep}} \Dv.
\end{equation*}
Note that, since it is the approximation of the integral of a Gaussian, the use of a first-order discrete integration introduces an error of order $\Dv$, even though the integral is not on a bounded domain:
\begin{equation}\label{eq:quadCeps}
    c^\ep_\Dv \underset{\Dv\to0,\, \ep\text{ fixed}}{=} \sqrt{2\pi\ep} + \mathcal{O}(\Dv).
\end{equation}

\myParagraph{Position discretization} 
Let $i\in\Z$. Denote $\varphi^{\ep,n}_{ij}$ and approximation of $\varphi^{\ep,n}_j(x_i)$. Remark that in \eqref{eq:schemeMuPhiSemiTV}, all terms are evaluated at grid points, except for $\varphi_{j}^{\ep,n}(x-\Dt v_j)$. This term is approximated using a linear interpolation between the two grid points surrounding its argument. Denoting by $\overbar{\varphi}_{i_j,j}^{\ep,n}$ this approximation, it is defined as:
\begin{equation}\label{eq:interpolation}
    \overbar{\varphi}_{i_j,j}^{\ep,n} = \alpha_j\varphi_{i_j,j}^{\ep,n} + (1-\alpha_j)\varphi_{i_j-1,j}^{\ep,n},
\end{equation}
where
\begin{equation*}
    i_j = i-\beta_j,\quad \beta_j = \left\lfloor v_j\frac{\Dt}{\Dx}\right\rfloor,\quad \alpha_j = v_j\frac{\Dt}{\Dx} - \beta_j,
\end{equation*}
and $\lfloor\cdot\rfloor$ denotes the floor function. In particular, this interpolation method guarantees that the transport preserves the $ L^\infty $-norm and inequalities, just as it does in the continuous setting. Indeed, $\overbar{\varphi}_{i_j,j}^{\ep,n}$ is a convex combination of the two pre-transport values, ensuring that it is bounded by these initial quantities. 

With the above notations, we define $\mu_i^{\ep,n}$ the approximation of $\mu^{\ep,n}(x_i)$, and $m_i^{\ep,n}$ the approximation of \eqref{eq:petitmSemiTV} given by
\begin{equation}\label{eq:petitmFull}
    m_i^{\ep,n} = \underset{j\in\Z}{\min}\lbr  \overbar{\varphi}_{i_j,j}^{\ep,n}\rbr.
\end{equation}
The quantity \eqref{eq:grandMSemiTV} is approximated by $M_{ij}^{\ep,n}$ defined as
\begin{equation}\label{eq:grandMFull}
    M_{ij}^{\ep,n}=\min\lbr \overbar{\varphi}_{i_j,j}^{\ep,n}+\Dt,\, \frac{v_j^2}{2}+\mu_i^{\ep,n+1}\rbr.
\end{equation}
Finally, for a fixed $\ep>0$, the fully discrete numerical scheme reads for all $(i,j)\in\Z^2$, and for all $n\in\ldb1,N_t\rdb$:
\begin{equation}\label{eq:schemeMuPhiFull}\tag{$\mathcal{S}^\ep$}
    \lbr\begin{aligned}
        \mu_i^{\ep,n+1} &= m_i^{\ep,n} - \ep \ln\left( \lla \e^{-\left[\overbar{\varphi}_{i_j,j}^{\ep,n} - m_i^{\ep,n}\right]/\ep} \rra_\Dv \right),\\
        \varphi_{ij}^{\ep,n+1} &= M_{ij}^{\ep,n} - \ep \ln\left( c_\Dv^\ep \e^{-\left[\overbar{\varphi}_{i_j,j}^{\ep,n} + \Dt - M_{ij}^{\ep,n}\right]/\ep} + \left( 1 - \e^{-\Dt/\ep} \right) \e^{-\left[\frac{v_j^2}{2} + \mu_i^{\ep,n+1} - M_{ij}^{\ep,n}\right]/\ep} \right)\\
        &\quad + \ep \ln(c_\Dv^\ep).
    \end{aligned}\right.
\end{equation}
This fully discrete scheme is initialized with
\begin{equation}\label{eq:initFull}
    \mu_i^{\ep,0} \coloneq \mu_{\tin,i} =\underset{j\in\Z}{\min}\lbr\overbar{\varphi}_{\tin,i_j,j}\rbr \quad\text{and}\quad \varphi_{ij}^{\ep,0}(x) = \varphi_\tin(x_i,v_j),
\end{equation}
and the notation $\overbar{\varphi}_{\tin,i_j,j}$ stand for the linear interpolation \eqref{eq:interpolation} of $\left(\varphi_\tin(x_i,v_j)\right)_{i\in\Z}$ at points $x_i-\Dt v_j$. Note that the introduction of \eqref{eq:petitmFull} and \eqref{eq:grandMFull} in the construction of the scheme not only provides intuition for the asymptotic behavior of the method, but also plays a crucial role in ensuring numerical stability with respect to $\ep$. Indeed, in this formulation, the exponential always has a non-positive argument, which prevents blow-up in the limit $\ep \to 0$.

\subsection{Main results}
The scheme \eqref{eq:schemeMuPhiFull} presented above is a structure-preserving discretization of \eqref{eq:kineticScaled}, for which the asymptotic limit can be rigorously justified. Moreover, its asymptotics naturally lead to a consistent and effective numerical scheme for the limit model \eqref{eq:Vari}. We emphasize that directly constructing a reliable numerical method for \eqref{eq:Vari} is not straightforward, due to the non-classical nature of the problem. In addition, the convergence of the asymptotic scheme toward the viscosity solution of \eqref{eq:Vari} as the discretization parameters vanish can also be established. These two results constitute the main contributions of this paper and are detailed below. Before stating them, we make the following assumption on the regularity of the initial data :

\begin{Assumption}\label{Ass:LinfBounds}
    First, there exists $M>0$ such that for all $(x,v)\in\R^2$ and all $\ep\in(0,1]$
    \begin{subequations}
        \makeatletter
        \def\@currentlabel{H1}
        \makeatother
        \renewcommand{\theequation}{H1.\alph{equation}}
        \begin{align}
            \left|\varphi_{\tin}(x,v)-\frac{v^2}{2}\right|&\leq M\label{eq:boundInitPhi}\\
            \left|\mu_{\tin}+\ep\ln(c_\Dv^\ep)\right| &\leq M.\label{eq:boundInitMu}
        \end{align}
    \end{subequations}
    Second, for all $v\in\R$, the function $x\mapsto\varphi_\tin(x,v)-\frac{v^2}{2}$ is $L$-Lipschitz.
\end{Assumption}

\begin{theorem}[Asymptotic Preserving property]\label{thm:AP}
    Let $\Dt > 0$, $\Dx > 0$, and $\Dv > 0$ be fixed. For all $(i,j) \in \Z^2$, we define $\mu_i^{\ep,0}$ and $\varphi_{ij}^{\ep,0}$ according to \eqref{eq:initFull}. Assume moreover that $(\varphi_{\tin},\mu_{\tin})$ satisfies Assumption~\ref{Ass:LinfBounds}. Then, in the limit $\ep\to0$, the solution $(\mu_i^{\ep,n},\varphi_{ij}^{\ep,n})$ to \eqref{eq:schemeMuPhiFull} converges locally uniformly in $i\in\Z$ towards $(\mu_i^{n},\varphi_{ij}^{n})$ solution to
    \begin{equation}\label{eq:VariDiscFull}\tag{$\mathcal{S}^0$}
        \lbr \begin{aligned}
            &\mu_i^{n+1}=\underset{j\in\Z}{\min} \lbr \overbar{\varphi}_{i_j,j}^{n}\rbr ,\\
            &\varphi_{ij}^{n+1} = \min\lbr \overbar{\varphi}_{i_j,j}^{n} + \Dt ,\, \frac{v_j^2}{2}+\mu_i^{n+1}\rbr,
        \end{aligned}\right.
    \end{equation}
    with $\overbar{\varphi}_{i_j,j}^{n}$ defined as in \eqref{eq:interpolation}, and the initial data is defined for all $j\in\Z$ and $i\in\Z$ by
    \begin{equation}\label{eq:initFullLimit}
        \mu_i^{0} = \underset{j\in\Z}{\min}\lbr\overbar{\varphi}_{\tin,i_j,j}\rbr \quad\text{and}\quad \varphi_{ij}^{0}=\varphi_\tin(x_i,v_j).
    \end{equation}
    In addition, the numerical cost of \eqref{eq:schemeMuPhiFull} is independent of $\ep$.
\end{theorem}
As stated in the following theorem, the asymptotic limit \eqref{eq:VariDiscFull} of \eqref{eq:schemeMuPhiFull} is actually a good discretization of \eqref{eq:Vari} in the sense that it captures the viscosity solution of \eqref{eq:Vari} when the discretization parameters tend to $0$.
\begin{theorem}[Convergence towards the viscosity solution]\label{thm:ConvVisc}
    Let $(n,i,j)\in\ldb1,N_t\rdb\times\Z^2$, and suppose that
    \begin{equation}\label{eq:stabJumpCond}
        \Dt \leq \frac{\Dv^2}{2}.
    \end{equation}
    Assume moreover that the initial data $\varphi_\tin$ satisfies Assumption~\ref{Ass:LinfBounds} as well as $L-$Lipschitz regularity in velocity. Then, the scheme \eqref{eq:VariDiscFull}-\eqref{eq:initFullLimit} captures the viscosity solution of \eqref{eq:Vari} when $\Dt$, $\Dx$, and $\Dv$ tend to $0$ under condition \eqref{eq:stabJumpCond} and $\frac{\Dx}{\Dt}\to0$. More precisely, there exists a constant $C$ such that
    \begin{equation}\label{eq:errorFull}
        \Big|\varphi_{ij}^{N_t}-\varphi(T, x_i, v_j)\Big|\leq C\left(\Dt+\frac{\Dx}{\Dt}+\Dv\right).
    \end{equation}
    The constant $C$ depends on the final time $T$ and on the velocity $v_j$.
\end{theorem}
The proof of these results both rely on stability properties of \eqref{eq:schemeMuPhiFull} and \eqref{eq:VariDiscFull}, namely an equilibrium-preserving property and a maximum principle. They are established in Sections~\ref{sec:AP} and~\ref{subsec:Variprop} respectively.

Regarding the proof of Theorem~\ref{thm:ConvVisc}, it unfolds as follows. The convergence of a semi-discrete scheme in time and velocity \eqref{eq:VariDiscSemiTV} towards \eqref{eq:Vari} is first established. In particular, we show that the viscosity solution is indeed captured. This is obtained by introducing a discrete representation formula. This reformulation is equivalent to \eqref{eq:VariDiscSemiTV} under the condition \eqref{eq:stabJumpCond} whose nature is unique to this problem and will be detailed in Section~\ref{subsec:DerivRep}. The full convergence of the solution to \eqref{eq:VariDiscFull} towards that of \eqref{eq:Vari} is then established by proving the convergence of \eqref{eq:VariDiscFull} towards \eqref{eq:VariDiscSemiTV} through the propagation of interpolation errors. It is worth mentioning that, proving this result directly, without the intermediate step, would require a deep understanding of a fully discrete representation formula to connect with the continuous notion of viscosity solutions—a significantly more complex problem.

The various convergence results in this work are summarized in Figure~\ref{fig:recapConv}. Although we do not emphasize this aspect, the convergence of \eqref{eq:schemeMuPhiFull} to \eqref{eq:kineticScaled} is expected for smooth solutions. However, a detailed analysis of this convergence lies outside the scope of this article and will instead be illustrated numerically in Section~\ref{sec:NumRes}. 

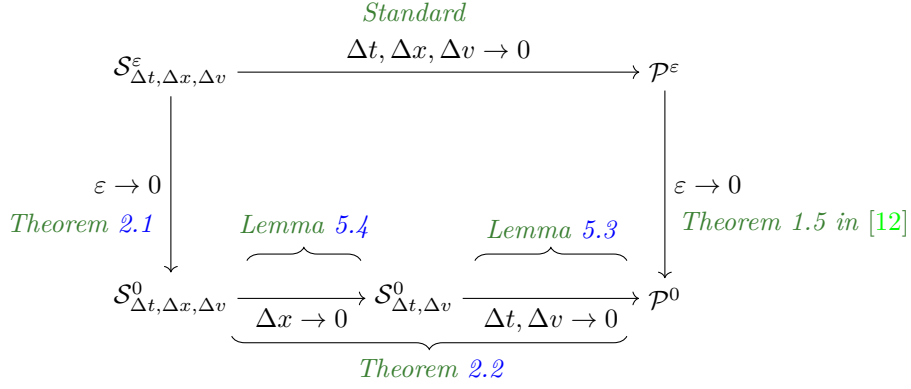
\begin{figure}
    \hspace{.8cm}
    \begin{tikzpicture}
        \node (P_e) at (0,2) {$\mathcal{S}^{\ep}_{\Delta t, \Delta x, \Delta v}$};
        \node (P_0) at (6.5,2) {$\mathcal{P}^{\ep}$};
        \node (P_disc) at (0,-1) {$\mathcal{S}^{0}_{\Delta t, \Delta x, \Delta v}$};
        \node (P_cont) at (6.5,-1) {$\mathcal{P}^{0}$};
        \node (P_mixed) at (3.2,-1) {$\mathcal{S}^{0}_{\Delta t, \Delta v}$};
        
        \draw[->] (P_e) -- (P_0) node[midway, above] {$\Delta t, \Delta x, \Delta v \to 0$};
        \draw[->] (P_e) -- (P_disc) node[midway, left] {$\ep \to 0$};
        \draw[->] (P_disc) -- (P_mixed) node[midway, below] {$\Delta x \to 0$};
        \draw[->] (P_mixed) -- (P_cont) node[midway, below] {$\Delta t, \Delta v \to 0$};
        \draw[->] (P_0) -- (P_cont) node[midway, right] {$\ep \to 0$};

        \draw [decorate, decoration={brace, amplitude=5pt}] (1.,-.5) -- (2.5,-.5) node[midway,above=5pt] {\textit{\textcolor{OliveGreen}{Lemma}~\ref{lem:ConvViscDx}}};
        \draw [decorate, decoration={brace, amplitude=5pt}] (4.,-.5) -- (6,-.5) node[midway,above=5pt] {\textit{\textcolor{OliveGreen}{Lemma}~\ref{lem:ConvViscDtDv}}};

        \draw [decorate, decoration={brace, amplitude=5pt, mirror}] (0.8,-1.5) -- (6,-1.5) node[midway,below=5pt] {\textit{\textcolor{OliveGreen}{Theorem}~\ref{thm:ConvVisc}}};

        \node[OliveGreen] at (8.2,0) {\textit{Theorem 1.5 in \cite{BouinCalvezGrenierNadin2023}}};
        \node[OliveGreen] at (-1.2,0) {\textit{Theorem~\ref{thm:AP}}};
        \node[OliveGreen] at (3.2,2.8) {\textit{Standard}};
    \end{tikzpicture}
    \caption{Structure of the convergence proofs}\label{fig:recapConv}
\end{figure}

\section{AP property}\label{sec:AP}

We begin by focusing on the semi-discrete scheme \eqref{eq:schemeMuPhiSemiTV}, where the position variable remains continuous. This choice allows to highlight the main analytical tools involved in the study of the numerical scheme. Both in this section and in Section~\ref{sec:ConvVisc}, the velocity and position variables play fundamentally different roles in the system’s dynamics. This distinction is already present in the original kinetic equation \eqref{eq:KinRelax}, where transport and relaxation affect the distribution $f^\ep$ in qualitatively different ways.

The extension of the analysis to the fully discrete case is conceptually straightforward, thanks to the interpolation procedure defined in \eqref{eq:interpolation}. However, the notations involved are considerably heavier. For this reason, we do not detail the fully discrete setting here, but some hints are provided at the end of the section.

\subsection{Stability properties of \texorpdfstring{\eqref{eq:schemeMuPhiFull}}{Seps}}
The proof of Theorem~\ref{thm:AP} follows the proof of the asymptotic behavior of \eqref{eq:kineticScaled} in the continuous setting. It involves several preliminary results, beginning with the following equilibrium-preserving property.

\begin{lemma}[Equilibrium Preserving]\label{lem:WBeps}
    The semi-discrete scheme \eqref{eq:schemeMuPhiSemiTV} preserves the equilibrium~$\left(v_j^2/2\right)_{j\in\Z}$: if $\varphi_\tin(x,v) = v^2/2$ for all $x\in\R$ and $v\in\R$, then 
    \begin{equation*}
        \forall n\in\ldb1,N_t\rdb,\,\forall j\in\Z,\,\forall x\in\R,\quad \varphi_j^{\ep,n}(x) = \frac{v_j^2}{2}.
    \end{equation*}
\end{lemma}
\begin{proof}
    Let $n\in\ldb1,N_t\rdb$ and assume that $\forall j\in\Z$ and $\forall x\in\R$, $\varphi_j^{\ep,n}(x) = \frac{v_j^2}{2}$. Then, $\forall x\in\R$, $\varphi_j^{\ep,n}(x-\Dt v_j)=\frac{v_j^2}{2}$. Moreover, since $v_0=0$, definition \eqref{eq:petitmSemiTV} yields $m^{\ep,n}(x)=0$, and
    \begin{equation}\label{eq:muEqui}
        \mu^{\ep,n+1}(x) \details{= m^{\ep,n}(x)-\ep\ln\left(\lla  \e^{-[\varphi_j^{\ep,n}(x-\Dt v_j)-m^{\ep,n}(x)]/\ep}\rra_\Dv \right)}
        =-\ep\ln\left(\lla  \e^{-\frac{v_j^2}{2}/\ep}\rra_\Dv \right)=-\ep\ln\left(c_\Dv^\ep\right).
    \end{equation}
    Equation \eqref{eq:muEqui} is then used to express $M_j^{\ep,n}(x)$, defined in \eqref{eq:grandMSemiTV}, as
    \begin{equation}\label{eq:mEqui}
        M_j^{\ep,n}(x) \details{=\min\lbr \frac{v_j^2}{2}+\Dt,\, \frac{v_j^2}{2}-\ep\ln\left(c_\Dv^\ep\right)\rbr } =\frac{v_j^2}{2} + \min\lbr \Dt,\,-\ep\ln\left(c_\Dv^\ep\right)\rbr .
    \end{equation}
    To determine $\varphi_j^{\ep,n+1}(x)$ and conclude, we distinguish two cases depending on the value of the minimum in \eqref{eq:mEqui}. In both cases, inserting \eqref{eq:muEqui} and \eqref{eq:mEqui} into the scheme on $\varphi_j^{\ep,n+1}$ in \eqref{eq:schemeMuPhiSemiTV} and simplifying the expressions yields
    \begin{equation*}
        \varphi_j^{\ep,n+1}(x) = \frac{v_j^2}{2},
    \end{equation*} 

    \details{We therefore distinguish between two cases, depending on the value of the minimum in \eqref{eq:mEqui}. If $\Dt\leq-\ep\ln(c_\Dv^\ep)$, pluging \eqref{eq:mEqui} into the scheme \eqref{eq:schemePhiSemiTV} on $\varphi_j^{\ep,n}(x)$:
    \begin{align*}
        \varphi_j^{\ep,n+1}(x) &= \frac{v_j^2}{2}+\Dt - \ep \ln\bigg(
        c_\Dv^\ep \e^{-\left[\frac{v_j^2}{2} + \Dt - (\frac{v_j^2}{2} + \Dt)\right]/\ep} \\
    &\quad + \left(1 - \e^{-\Dt/\ep}\right)
        \e^{-\left[\frac{v_j^2}{2} -\ep\ln\left(c_\Dv^\ep\right) - (\frac{v_j^2}{2}+\Dt)\right]/\ep}
    \bigg) + \ep \ln(c_\Dv^\ep)\\
    &=\frac{v_j^2}{2}+\Dt - \ep \ln\bigg( c_\Dv^\ep + c_\Dv^\ep \left(1 - \e^{-\Dt/\ep}\right)\e^{\Dt/\ep}\bigg) + \ep \ln(c_\Dv^\ep)\\
    &=\frac{v_j^2}{2}.
    \end{align*}
    Similarly, the same computations in the case $-\ep\ln(c_\Dv^\ep)\leq\Dt$ lead to
    \begin{align*}
        \varphi_j^{\ep,n+1}(x) &= \frac{v_j^2}{2}-\ep\ln(c_\Dv^\ep) - \ep \ln\bigg(c_\Dv^\ep \e^{-\left[\frac{v_j^2}{2} + \Dt - (\frac{v_j^2}{2} -\ep\ln(c_\Dv^\ep))\right]/\ep} \\
    &\quad + \left(1 - \e^{-\Dt/\ep}\right)
        \e^{-\left[\frac{v_j^2}{2} -\ep\ln\left(c_\Dv^\ep\right) - (\frac{v_j^2}{2}-\ep\ln(c_\Dv^\ep))\right]/\ep}
    \bigg) + \ep \ln(c_\Dv^\ep)\\
    &=\frac{v_j^2}{2} - \ep \ln\bigg(\e^{-\Dt/\ep} + \left(1 - \e^{-\Dt/\ep}\right)\bigg)\\
    &=\frac{v_j^2}{2},
    \end{align*}
    }
    which concludes the proof.
\end{proof}

\details{
The second stability property enjoyed by the scheme \eqref{eq:schemeMuPhiSemiTV} is a maximum principle. This property is closely related to the continuous $L^\infty$-bounds on $\varphi^\ep$ and is a crucial tool in the proof of Theorem~\ref{thm:AP}. We refer to \cite{BouinCalvezGrenierNadin2023} for more details on the continuous setting.

\begin{lemma}[Maximum Principle]\label{lem:MaxPrincEps}
    Let $\ep>0$, $\Dt>0,\Dx>0$ and $\Dv>0$. Let $\left(\varphi_{j}^{\ep,n}(x)\right)_{n\in\ldb1,N_t\rdb, j\in\Z}$ and $\left(\mu^{\ep,n}(x)\right)_{n\in\ldb1,N_t\rdb}$ be solution to \eqref{eq:schemeMuPhiSemiTV}. Assume that the initial condition $\varphi_{\tin,j}(x)$ satisfies for all $j\in\Z$ and $x\in\R$ the inequalities \eqref{eq:boundInitPhi}-\eqref{eq:boundInitMu}. Therefore, for all $(n,j)\in\ldb1,N_t\rdb\times\Z$ and $x\in\R$, $\varphi_{j}^{\ep,n}(x)$ and $\mu_i^{\ep,n}$ satisfy the same bounds:
    \begin{align}
        \left|\varphi_{j}^{\ep,n}(x)-\frac{v_j^2}{2}\right| &\leq M\label{eq:boundPhi}\\
        \left|\mu^{\ep,n}(x)+\ep\ln(c_\Dv^\ep)\right| &\leq M\label{eq:boundMu}
    \end{align}
\end{lemma}
\begin{proof}
    We proceed by induction, assuming that for some $n$ and $\forall j\in\Z$, $x\in\R$:
    \begin{equation}\label{eq:inducHypoMaxPrinc}
        -M \leq \varphi_j^{\ep,n}(x)-\frac{v_j^2}{2} \leq M.
    \end{equation}
    Remarking that the free transport preserves the $L^\infty$-norm, this implies the bounds:
    \begin{equation}\label{eq:boundTransportMinusPara}
        -M\leq\varphi_j^{\ep,n}(x-\Dt v_j)-\frac{v_j^2}{2}\leq M,
    \end{equation}
    and it yields 
    \begin{equation*}
        \e^{-M/\ep} \leq \e^{-\left(\varphi_j^{\ep,n}(x-\Dt v_j)-\frac{v_j^2}{2}\right)/\ep} \leq \e^{M/\ep}.
    \end{equation*}
    The proof of \eqref{eq:boundMu} is then a consequence definition \eqref{eq:updateMu}. Indeed,
    \begin{equation*}
        -\mu^{\ep,n+1}(x) \details{= \ep\ln\left(\lla  \e^{\varphi_j^{\ep,n}(x-\Dt v_j)/\ep} \rra_\Dv \right)} =\ep\ln\left(\lla  \e^{-\left(\varphi_j^{\ep,n}(x-\Dt v_j)-\frac{v_j^2}{2}\right)/\ep} \e^{-\frac{v_j^2}{2\ep}} \rra_\Dv \right),
    \end{equation*}
    and \eqref{eq:boundMu} is satisfied at time $t^{n+1}$.
    \details{
        \begin{equation*}
       M-\ep\ln(c_\Dv^\ep)\leq \mu^{\ep,n+1}(x) \leq M-\ep\ln(c_\Dv^\ep).
    \end{equation*}
    }
    Next we turn our attention to the proof of \eqref{eq:boundPhi}. We start from the following formulation of the scheme \eqref{eq:schemePhiSemiTV}, where we simplified the terms $M_j^{\ep,n}(x)$:
    \begin{equation}\label{eq:schemePhiSemiTVBis}
        \e^{-\varphi_j^{\ep,n+1}(x)/\ep} = \e^{-\varphi_j^{\ep,n}(x-\Dt v_j)/\ep}\e^{-\Dt/\ep} + (1-\e^{-\Dt/\ep})\frac{1}{c_\Dv^\ep}\e^{-\frac{v_j^2}{2\ep}}\e^{-\mu^{\ep,n+1}(x)/\ep}.
    \end{equation}
    It can further be reformulated as a convex combination of exponentials of bounded terms thanks to \eqref{eq:boundTransportMinusPara} at time $t^n$ and \eqref{eq:boundMu}, that is proven to be true at time $t^{n+1}$. It writes,
    \begin{equation*}
        \e^{-\varphi_j^{\ep,n+1}(x)/\ep}\e^{\frac{v_j^2}{2}} = \e^{-\Dt/\ep}\e^{-\varphi_j^{\ep,n}(x-\Dt v_j)/\ep}\e^{\frac{v_j^2}{2\ep}} + (1-\e^{-\Dt/\ep})\frac{1}{c_\Dv^\ep}\e^{-\mu^{\ep,n+1}(x)/\ep},
    \end{equation*}
    so that the following bounds hold:
    \begin{equation}\label{eq:boundConvCombPhi}
        \begin{aligned}
        \min\Bigg\{\e^{\left(\frac{v_j^2}{2}-\varphi_j^{\ep,n}(x-\Dt v_j)\right)/\ep}&,\, \frac{1}{c_\Dv^\ep}\e^{-\mu^{\ep,n+1}(x)/\ep}\Bigg\} 
        \leq \e^{\left(\frac{v_j^2}{2}-\varphi_j^{\ep,n+1}(x)\right)/\ep}\\
        & \leq \max\Bigg\{\e^{\left(\frac{v_j^2}{2}-\varphi_j^{\ep,n}(x-\Dt v_j)\right)/\ep},\, \frac{1}{c_\Dv^\ep}\e^{-\mu^{\ep,n+1}(x)/\ep}\Bigg\}.
    \end{aligned}
    \end{equation}
    Since the function $s \mapsto -\ep\ln(s)$ is decreasing, and in view of \eqref{eq:boundConvCombPhi}, we obtain—by applying the identities $-\ep \ln(\min\lbr a,b\rbr) = \max\lbr-\ep\ln(a), -\ep\ln(b)\rbr$ and $-\ep \ln(\max\lbr a,b\rbr) = \min\lbr-\ep\ln(a), -\ep\ln(b)\rbr$, that \details{(original: $-\min\lbr-a,-b\rbr=\max\lbr a,b\rbr$ and $-\max\lbr-a,-b\rbr=\min\lbr a,b\rbr$)}:
    \begin{equation}\label{eq:boundConcCombPhiBis}
        \begin{aligned}
        \min\Bigg\{\varphi_j^{\ep,n}(x-\Dt v_j)-\frac{v_j^2}{2}&,\, \mu^{\ep,n+1}(x)+ \ep\ln(c_\Dv^\ep)\Bigg\} \leq \varphi_j^{\ep,n+1}(x)-\frac{v_j^2}{2}\\
        & \leq \max\Bigg\{\varphi_j^{\ep,n}(x-\Dt v_j)-\frac{v_j^2}{2},\, \mu^{\ep,n+1}(x)+ \ep\ln(c_\Dv^\ep)\Bigg\}.
    \end{aligned}
    \end{equation}
    Finally, combining \eqref{eq:boundConcCombPhiBis} with \eqref{eq:boundTransportMinusPara} and \eqref{eq:boundMu} we obtain \eqref{eq:boundPhi}, which concludes the proof.
    \details{
        \begin{equation*}
                -M \leq \varphi_j^{\ep,n+1}(x)-\frac{v_j^2}{2} \leq M
        \end{equation*}
    }
\end{proof}
}

The next Lemma states the monotonicity of \eqref{eq:schemeMuPhiSemiTV} as well as invariance and commutative properties.
\begin{lemma}\label{lem:MonoTransConstEps}
    Let $(n,j)\in\ldb1,N_t\rdb\times\Z$ and $x\in\R$. Let $\left(\mu^{\ep,n}(x)\right)_{n\in \ldb1,N_t\rdb,j\in\Z}$ and $\left(\varphi_j^{\ep,n}(x)\right)_{n\in \ldb1,N_t\rdb,j\in\Z}$ be a solution to \eqref{eq:schemeMuPhiSemiTV} with initial data \eqref{eq:initSemiTV}. Denote $\left(\psi^{\ep,n}(x)\right)_{n\in \ldb1,N_t\rdb,j\in\Z}$ another solution to \eqref{eq:schemeMuPhiSemiTV} and
    \begin{equation}\label{eq:updateNu}
        \nu^{\ep,n}(x)=-\ep\ln\left(\lla\psi^{\ep,n}(x-\Dt v_j)\rra_\Dv\right).
    \end{equation}
    Then, the following properties hold:
    \begin{enumerate}[label=\roman*)]
        \item \textbf{Monotonicity.} Suppose that $\varphi_{\tin}\leq\psi_{\tin}$. Then, for all $j\in\Z$, $x\in\R$ and $n\in \ldb1,N_t\rdb$,
        \begin{align}
            &\mu^{\ep,n}(x)\leq\nu^{\ep,n}(x), \label{eq:monoEpsMu}\\
            &\varphi_j^{\ep,n}(x)\leq\psi_j^{\ep,n}(x).\label{eq:monoEpsPhi}
        \end{align}
        \item \textbf{Invariance by translation.} Let $y\in\R$ and suppose that $\psi_{\tin}=\varphi_{\tin}(\cdot+y,\cdot)$. Then, for all $j\in\Z$, $x\in\R$ and $n\in \ldb1,N_t\rdb$,
        \begin{align}
            &\nu^{\ep,n}(x)=\mu^{\ep,n}(x+y),\label{eq:TransEpsMu}\\
            &\psi_j^{\ep,n}(x)=\varphi_j^{\ep,n}(x+y).\label{eq:TransEpsPhi}
        \end{align}
        \item \textbf{Commutation with constants.} Let $K\in\R$ and suppose that $\psi_{\tin}~=~\varphi_{\tin}+K$. Then, for all $j\in\Z$, $x\in\R$ and $n\in \ldb1,N_t\rdb$,
        \begin{align}
            &\nu^{\ep,n}(x)=\mu^{\ep,n}(x)+K, \label{eq:ConstEpsMu}\\
            &\psi_j^{\ep,n}(x)=\varphi_j^{\ep,n}(x)+K.\label{eq:ConstEpsPhi}
        \end{align}
    \end{enumerate}
\end{lemma}
\begin{proof}
    For each property, we proceed by induction. 
    \begin{enumerate}[label=\roman*)]
        \item Let $n\in\ldb 0,N_t-1\rdb$ and suppose that $\varphi_j^{\ep,n}(x)\leq\psi_j^{\ep,n}(x)$ for all $j\in\Z$ and $x\in\R$. Then, one has in particular that $\varphi_j^{\ep,n}(x-\Dt v_j)\leq\psi_j^{\ep,n}(x-\Dt v_j)$. It implies
        \begin{equation*}
            \lla \e^{-\left(\varphi_j^{\ep,n}(x-\Dt v_j)\right)/\ep} \rra_\Dv \geq \lla \e^{-\left(\psi_j^{\ep,n}(x-\Dt v_j)\right)/\ep }\rra_\Dv.
        \end{equation*}
        Using definitions \eqref{eq:updateMu} and \eqref{eq:updateNu}, it yields \eqref{eq:monoEpsMu}.
        Similar computations lead to \eqref{eq:monoEpsPhi} by using that the second line of \eqref{eq:schemeMuPhiSemiTV} can be reformulated as a sum of positive terms as follows:
        \begin{equation*}
            \e^{-\varphi_j^{\ep,n+1}(x)/\ep} = \e^{-\varphi_j^{\ep,n}(x-\Dt v_j)/\ep}\e^{-\Dt/\ep} + (1-\e^{-\Dt/\ep})\frac{1}{c_\Dv^\ep}\e^{-\frac{v_j^2}{2\ep}}\e^{-\mu^{\ep,n+1}(x)/\ep}.
        \end{equation*}
        \item Let $y\in\R$ and define $\psi_j^{\ep,n}(x)=\varphi_j^{\ep,n}(x+y)$, for all $j\in\Z$ and $x\in\R$. Starting from definition \eqref{eq:updateNu}, direct computations yield:
        \begin{equation*}
            \e^{\nu^{\ep,n+1}(x)/\ep} =\lla \e^{-\left(\varphi_j^{\ep,n}(x+y-\Dt v_j)\right)/\ep} \rra_\Dv=\e^{\mu^{\ep,n+1}(x+y)/\ep},
        \end{equation*}
        from which \eqref{eq:TransEpsMu} is deduced. Using \eqref{eq:TransEpsMu}, similar computations then lead to \eqref{eq:TransEpsPhi}.
        \item The proof of \eqref{eq:ConstEpsMu}–\eqref{eq:ConstEpsPhi} is analogous to the previous one and uses the algebraic properties of the exponentials.
    \end{enumerate}
\end{proof}
From Lemma~\ref{lem:MonoTransConstEps} one can then deduce the exact propagation of Lipschitz bounds in position by the semi-discrete numerical scheme \eqref{eq:schemeMuPhiSemiTV} (see \cite{CrandallLions1984}), and:
\begin{corollary}\label{cor:LipBoundEps}
    For all $n\in\ldb1,N_t\rdb$ and $\ep>0$, the solution to \eqref{eq:schemeMuPhiSemiTV} satisfies Assumption~\ref{Ass:LinfBounds}. \details{Suppose that $\varphi_\tin-v^2/2$ is $L$-Lipschitz in both variables. Then,  $x\mapsto\varphi_j^{\ep,n}(x)-\frac{v_j^2}{2}$ is also $L$-Lipschitz.}
\end{corollary}
\details{
    \begin{proof}
        Let $y\in\R$ be fixed. We proceed by induction. Assuming that $\varphi_j^{\ep,n}(x)$ is $L$-Lipschitz for some $n\in\ldb1,N_t\rdb$, one has in particular
        \begin{equation}\label{eq:ineqLip}
            \varphi_j^{\ep,n}(x) \leq \varphi_j^{\ep,n}(x+y) + |y|L.
        \end{equation}
        Denoting $\psi_j^{\ep,n}(x)=\varphi_j^{\ep,n}(x+y) + |y|L$ it gives $\varphi_j^{\ep,n}(x) \leq \psi_j^{\ep,n}(x)$. Then, we denote by $S^\ep$ the pushforward operator corresponding to the numerical scheme \eqref{eq:schemePhiSemiTV} so that 
        \begin{equation*}
            S^\ep[\varphi_j^{\ep,n}](x) = \varphi_j^{\ep,n+1}(x).
        \end{equation*}
        We have shown in Lemma~\ref{lem:MonoTransConstEps} the scheme commute with constants and translations. Therefore, for $K\in\R$,
        \begin{equation}\label{eq:MonoTransConstEpsPF}
            S^\ep[\varphi_j^{\ep,n}(\cdot+y)+K](x) = S^\ep[\varphi_j^{\ep,n}](x+y)+K = \varphi_j^{\ep,n+1}(x+y)+K.
        \end{equation}
        Applying the relation \eqref{eq:MonoTransConstEpsPF} to $\psi_j^{\ep,n+1}(x)$ gives in particular
        \begin{equation*}
            \psi_j^{\ep,n+1}(x) = \varphi_j^{\ep,n+1}(x+y)+K.
        \end{equation*}
        Consequently, combining this identity with \eqref{eq:ineqLip}, one obtains
        \begin{equation*}
            \varphi_j^{\ep,n+1}(x) \leq \varphi_j^{\ep,n+1}(x+y) + |y|L,
        \end{equation*}
        and by following the same steps while exchanging the role of $\varphi_j^{\ep,n+1}(\cdot)$ and $\varphi_j^{\ep,n+1}(\cdot+y)$, the opposite inequality is obtained.
        Finally, 
        \begin{equation*}
            |\varphi_j^{\ep,n+1}(x) -\varphi_j^{\ep,n+1}(x+y)| \leq |y|L,
        \end{equation*}
        which concludes the proof.
    \end{proof}
}
Another key aspect of the analysis is the identification of the minimizer $\argmin_{j\in\Z} \{ \varphi_j^{\ep,n}(x - \Dt v_j) \}$. Since $\varphi_{\tin}(x,v)$ grows like $v^2/2$ we expect that, its minima belong to a bounded interval. In other term, the minima cannot be achieved for a too large velocity. This is the content of the following lemma.

\begin{lemma}\label{lem:LocMinimaEps}
    Let $\Dt,\Dx,\Dv>0$ be fixed. Assume that the initial data $\varphi_{\tin}$ satisfies Assumption~\ref{Ass:LinfBounds}. Then, there exists a bounded set $\overbar{V}\subset\Z$ such that for all $n\in\ldb1,N_t\rdb$, $\ep>0$ and $x\in\R$,
    \begin{equation}\label{eq:LocArgminEps}
        \underset{j\in\Z}{\argmin}\lbr\varphi_j^{\ep,n}(x-\Dt v_j)\rbr\subset \Dv\overbar{V}.
    \end{equation}
\end{lemma}
\begin{proof}
    We proceed in two steps. We first localize $\argmin_{j\in\Z}\lbr\varphi_j^{\ep,n}(x)\rbr$ and then consider the effect of the transport. Let $(n,j)\in\ldb1,N_t\rdb\times\Z$, $\ep>0$ and $x\in\R$. Thanks to Corollary~\ref{cor:LipBoundEps},
    \begin{equation}\label{eq:quadGrowthEps}
        -M + \frac{v_j^2}{2} \leq\,\varphi_{j}^{\ep,n}(x) \leq\, M + \frac{v_j^2}{2},
    \end{equation}
    which, for $j=0$, gives in particular $|\varphi_{0}^{\ep,n}(x)|\leq M$ and thus
    \begin{equation}\label{eq:quadGrowthEpsBis}
        \underset{j\in\Z}{\min}\lbr\varphi_{j}^{\ep,n}(x)\rbr \leq M,
    \end{equation}
    where we recall that $M$ depends neither on $n$, $\ep$ nor $x$. However, for $\Dv>0$ fixed there exists $J_1\in\N$, such that for all $|k|>J_1$ 
    \begin{equation}\label{eq:quadGrowthEpsBisBis}
        -M + \frac{v_k^2}{2} \geq M+1.
    \end{equation}
    Combining \eqref{eq:quadGrowthEpsBis} and \eqref{eq:quadGrowthEpsBisBis} then gives that for all $|k|>J_1$
    \details{
        \begin{equation*}
        \underset{j\in\Z}{\min}\lbr\varphi_{j}^{\ep,n}(x)\rbr+1 \leq M+1 \leq -M + \frac{v_k^2}{2} \leq \varphi_{k}^{\ep,n}(x).
    \end{equation*}
    This relation in particular states that for all $|k|>J_1$
    }
    \begin{equation*}
        \underset{j\in\Z}{\min}\lbr\varphi_{j}^{\ep,n}(x)\rbr+1 < \varphi_{k}^{\ep,n}(x),
    \end{equation*}
    which means that the minimum is achieved for some $|j|<J_1$. The next step consists in dealing with the effect of the transport on the localization of the minima. Thanks to Corollary~\ref{cor:LipBoundEps}, $x\mapsto\varphi_{j}^{\ep,n}(x)-v^2/2$ is $L$-Lipschitz. Combined with \eqref{eq:quadGrowthEps}, it yields
    \details{
        \begin{equation*}
        -L\Dt|v_j| \leq \varphi_{j}^{\ep,n}(x-\Dt v_j)-\varphi_{j}^{\ep,n}(x) \leq L\Dt|v_j|,
    \end{equation*}
    }
    \begin{equation}\label{eq:quadGrowthEpsBisBisBis}
        -M+\frac{v_j^2}{2}-L\Dt|v_j| \leq \varphi_{j}^{\ep,n}(x-\Dt v_j) \leq L\Dt|v_j|+M+\frac{v_j^2}{2}.
    \end{equation}
    Starting from \eqref{eq:quadGrowthEpsBisBisBis}, the same arguments as before can be used. The upper bound is in particular true at $j=0$ giving $\underset{j\in\Z}{\min}\lbr\varphi_{j}^{\ep,n}(x-\Dt v_j)\rbr\leq M$. Considering the lower bound, there exists $J_2\in\N$ such that for all $|k|>J_2$, 
    \begin{equation*}
        -M+\frac{v_j^2}{2}-L\Dt|v_j| > M+1.
    \end{equation*}
    Concluding in same way as the transport-free minimization problem, we obtain \eqref{eq:LocArgminEps} by setting $\overbar{V}=\Big[-\max\lbr J_1,J_2\rbr,\, \max\lbr J_1,J_2\rbr\Big]$.
\end{proof}

\begin{Remark}\label{rem:convEpsFull}
    The extension of the stability results to the fully discrete scheme \eqref{eq:schemeMuPhiFull} are actually straightforward given our choice of interpolation. Regarding Lemma~\ref{lem:WBeps}, since the property is independent of the position variable, it remains unaffected by the spatial discretization. The maximum principle and the monotonicity of the scheme follow directly from the monotonicity of the linear interpolation \eqref{eq:interpolation}. Finally, the commutative properties of \eqref{eq:schemeMuPhiFull} with respect to constants and translations also carry over, as these properties are preserved by the linear interpolation. In the case of translations, it should be noted that they must correspond to shifts by grid points, i.e., of the form $y=i\Dx$ for some $i \in \Z$. This is a natural assumption, given that the unknowns are localized on the grid.
\end{Remark}

\subsection{Limit \texorpdfstring{$\ep\to0$ of \eqref{eq:schemeMuPhiFull}}{eps to 0 of Seps}}
This section is dedicated to the proof of the asymptotic preserving property of \eqref{eq:schemeMuPhiFull} stated in Theorem~\ref{thm:AP}. It is achieved in two steps. First, the logarithmic terms in \eqref{eq:schemeMuPhiSemiTV} are shown to converge towards $0$. Secondly, relying on the stability properties of Lemmas~\ref{lem:WBeps} and \ref{lem:MonoTransConstEps} and on Corollary~\ref{cor:LipBoundEps}, the limit is proven.

\begin{proof}[Proof of Theorem~\ref{thm:AP}]
    Let us first point out that the numerical cost of solving \eqref{eq:schemeMuPhiFull} is independent of $\ep$. The resolution is an explicit procedure thanks to \eqref{eq:updateMuStable} and the discretization parameters can be chosen independently of $\ep$ ensuring that \eqref{eq:schemeMuPhiFull} is an efficient numerical scheme in practice.
    
    \textbf{Step 1.} Let $x\in\R$, let us fix $\Dt$, $\Dv>0$, $j\in\Z$ and $n\in\ldb1,N_t\rdb$. We recall that the quantities $M^{\ep,n}$ and $m^{\ep,n}$ have been introduced in \eqref{eq:petitmSemiTV} and \eqref{eq:grandMSemiTV} to ensure non-positivity in the exponential terms of \eqref{eq:schemeMuPhiSemiTV}. Then, focusing on the first line of \eqref{eq:schemeMuPhiSemiTV}, we observe that, by definition \eqref{eq:petitmSemiTV}, at least one exponential in the discrete integral is equal to one. More precisely, there exists $j_0\in\Z$ such that $\varphi_{j_0}^{\ep,n}(x-\Dt v_{j_0}) - m^{\ep,n}(x)=0$. The following relation then holds:
    \begin{equation}\label{eq:integralMuAP}
        \lla \e^{-\left[\varphi_{j}^{\ep,n}(x-\Dt v_j) - m^{\ep,n}(x)\right]/\ep} \rra_\Dv = \Dv + \sum_{j\in\Z,\, j\neq j_0} \e^{-\left[\varphi_{j}^{\ep,n}(x-\Dt v_j) - m^{\ep,n}(x)\right]/\ep}\Dv,
    \end{equation}
    where the remaining terms in the exponentials are non-positive. On the one hand, one obtains from \eqref{eq:integralMuAP}
    \begin{equation}\label{eq:integralMuAP1}
        \ep\ln\left(\lla \e^{-\left[\varphi_{j}^{\ep,n}(x-\Dt v_j) - m^{\ep,n}(x)\right]/\ep} \rra_\Dv\right) \geq \ep\ln\left(\Dv\right)\underset{\ep\to0}{\to} 0.
    \end{equation}
    On the other hand, we establish an upper bound for \eqref{eq:integralMuAP} by exploiting the quadratic growth of $\varphi_{j}^{\ep,n}$. From Corollary~\ref{cor:LipBoundEps}, we obtain in particular
    \begin{equation}\label{eq:boundAP1}
    -M-m^{\ep,n}(x)+\frac{v_j^2}{2}-|v_j|
    \leq \varphi_{j}^{\ep,n}(x-\Dt v_j)-m^{\ep,n}(x)-|v_j|
    \leq M-m^{\ep,n}(x)+\frac{v_j^2}{2}-|v_j|,
    \end{equation}
    where the second inequality and \eqref{eq:petitmSemiTV} and with $j=0$ yields
    \begin{equation}\label{eq:boundAP2}
    m^{\ep,n}(x)\leq M.
    \end{equation}
    Combining \eqref{eq:boundAP1} and \eqref{eq:boundAP2} gives
    \begin{equation}
    -2M+\frac{v_j^2}{2}-|v_j|
    \leq \varphi_{j}^{\ep,n}(x-\Dt v_j)-m^{\ep,n}(x)-|v_j|.
    \end{equation}
    Remark now that the set
    \begin{equation}\label{eq:defSetA}
    A \coloneq \lbr j \in \Z \,:\, -2M+\tfrac{v_j^2}{2}-|v_j| < 0 \rbr
    \end{equation}
    is finite, independently of $\ep$. Moreover, since $0<\ep\leq1$, for any $j\in\bar{A}$, the complement set of $A$, one has
    \begin{equation}\label{eq:boundAP3}
    -\frac{1}{\ep}\left(\varphi_{j}^{\ep,n}(x-\Dt v_j)-m^{\ep,n}(x)\right) \details{\leq -\frac{|v_j|}{\ep}} \leq -|v_j|.
    \end{equation}
    Thus, splitting the sum in \eqref{eq:integralMuAP} according to $A$ and $\bar A$ gives
    \begin{equation*}
        \begin{aligned}
        \lla \e^{-\left[\varphi_{j}^{\ep,n}(x-\Dt v_j) - m^{\ep,n}(x)\right]/\ep} \rra_\Dv
        &= \Dv \sum_{j\in A} \e^{-\left[\varphi_{j}^{\ep,n}(x-\Dt v_j)-m^{\ep,n}(x)\right]/\ep} \\
        &\quad+ \Dv \sum_{j\in\bar A} \e^{-\left[\varphi_{j}^{\ep,n}(x-\Dt v_j)-m^{\ep,n}(x)\right]/\ep}.
        \end{aligned}
    \end{equation*}
    The first sum is finite and bounded independently of $\ep$ by $\Dv\# A$. Thanks to \eqref{eq:boundAP3}, the second sum is bounded from above by a convergent series that is also independent of $\ep$. Consequently, the quantity in \eqref{eq:integralMuAP} is bounded by a constant $C>0$ for all $\ep$, and
    \begin{equation}\label{eq:integralMuAP2}
    \ep \ln\left(\lla \e^{-\left[\varphi_{j}^{\ep,n}(x-\Dt v_j)-m^{\ep,n}(x)\right]/\ep} \rra_\Dv\right)
    \leq \ep \ln(C) \underset{\ep\to0}{\longrightarrow} 0.
    \end{equation}
    Finally, combining \eqref{eq:integralMuAP2} with \eqref{eq:integralMuAP1}, we conclude that
    \begin{equation}\label{eq:convExpmu}
    \forall j\in\Z,~ \forall \Dt,\Dv>0, \qquad
    \ep \ln\left( \lla \e^{-\left[\varphi_{j}^{\ep,n}(x-\Dt v_j)-m^{\ep,n}(x)\right]/\ep} \rra_\Dv \right)
    \underset{\ep\to0}{\longrightarrow} 0.
    \end{equation}

    Using a similar argument, we deal with the logarithmic term in the second line of \eqref{eq:schemeMuPhiSemiTV}. By definition \eqref{eq:grandMSemiTV}, at least one of the two exponentials is necessarily equal to one. Distinguishing two cases depending on the value of \eqref{eq:grandMSemiTV}, the argument of the logarithm reduces to
    \begin{equation}
        \label{eq:negExponentialsAP}
        \begin{cases}
        c_\Dv^\ep + \left(1 - \e^{-\Dt/\ep}\right)  
        \e^{-\left[\frac{v_j^2}{2} + \mu^{\ep,n+1}(x) - M_{j}^{\ep,n}(x)\right]/\ep}, 
        & \text{if } \varphi_{j}^{\ep,n}(x-\Dt v_j) + \Dt \leq \frac{v_j^2}{2} + \mu^{\ep,n+1}(x), \\[1ex]
        c_\Dv^\ep \e^{-\left[\varphi_{j}^{\ep,n}(x-\Dt v_j) + \Dt - M_{j}^{\ep,n}(x)\right]/\ep}
        + \left(1 - \e^{-\Dt/\ep}\right), 
        & \text{else.}
        \end{cases}
    \end{equation}
    In both cases, the remaining exponential terms have negative arguments and in the limit $\ep\to0$, the dominating term is at most $c_\Dv^\ep + 1$. Then, recalling that $\Dv$ is fixed, we remark that $\ep\ln(c_\Dv^\ep)$ converges towards $0$ thanks to the relation $\ep\ln(c_\Dv^\ep)=\ep\ln(\sqrt{2\pi\ep}+\mathcal{O}(\Dv))$, where we used the error estimate on first order quadratures \eqref{eq:quadCeps}. Combining \eqref{eq:negExponentialsAP} with this observation, we obtain, for all $j\in\Z$ and $\Dt,\,\Dv>0$, that the logarithmic terms in the second line of \eqref{eq:schemeMuPhiSemiTV} also vanishes in the limit $\ep\to0$:
    \begin{equation}\label{eq:convExpPhi}
        \begin{aligned}
            \ep \ln\bigg(&c_\Dv^\ep \e^{-\big[\varphi_{j}^{\ep,n}(x-\Dt v_j) + \Dt - M_{j}^{\ep,n}(x)\big]/\ep} \\
            &\qquad+ \left(1 - \e^{-\Dt/\ep}\right) \e^{-\left[\frac{v_j^2}{2} + \mu^{\ep,n+1}(x) - M_{j}^{\ep,n}(x)\right]/\ep} \bigg) + \ep \ln(c_\Dv^\ep)\underset{\ep\to0}{\to} 0.
        \end{aligned}
    \end{equation}

    \details{
        Again, if a difference is of size $C\ep$ one has, depending on the value of $M_{j}^{\ep,n}(x),$
        \begin{equation*}
            c_\Dv^\ep + \left(1 - \e^{-\Dt/\ep}\right)  
            \e^{-C} \underset{\ep\to0}{\to} \e^{-C}
        \end{equation*}
        or, 
        \begin{equation*}
            c_\Dv^\ep \e^{-C} + \left(1 - \e^{-\Dt/\ep}\right)\underset{\ep\to0}{\to} 1
        \end{equation*}
        and the logarithm term still converges towards $0$.
    }

    \textbf{Step 2.} Let $\Delta t$, $\Delta v$, and $j \in \Z$ be fixed. The next step in the proof is to pass to the limit in the terms $m^{\ep,n}(x)$ and $M_j^{\ep,n}(x)$. We begin with $m^{\ep,n}(x)$. Let $I\subset\mathbb{R}$ an interval and let $x\in I$.

    First, by applying Corollary~\ref{cor:LipBoundEps}, which provides a Lipschitz bound propagation, one may apply Ascoli's theorem on the family of functions $\left( y \mapsto \varphi_j^{\ep,n}(y) - v_j^2/2 \right)_{\ep > 0}$ for any $j\in\Z$. Therefore, for all $j\in\Z$, this yields convergence, up to the extraction of a subsequence, uniformly on $I$ towards a function $y \mapsto \varphi_j^n(y)- v_j^2/2$.
    
    As a consequence, for all $j\in\Z$, the limit function $\varphi_j^n$ satisfies Assumption~\ref{Ass:LinfBounds}. Using the same reasoning as in the proof of Lemma~\ref{lem:LocMinimaEps}, the minimum $\underset{j \in \Z}{\min} \lbr \varphi_j^{n}(x - \Delta t v_j) \rbr$ is well-defined and attained within a bounded set $\ldb -J, J \rdb$ for some $J \in \N$.

    Secondly, let $A>0$ and define $I = [-A - \Dt v_{-J}, A + \Dt v_J]$ and fix $j \in \ldb-J, J\rdb$. The family $\left(y \mapsto \varphi_j^{\ep,n}(y) - v_j^2/2\right)_{\ep > 0, n \in \ldb1, N_t\rdb}$ converges uniformly on $I$ thanks to the above discussion. In particular, for $x \in [-A, A]$, one has $\varphi_j^{\ep,n}(x - \Dt v_j) \to \varphi_j^n(x - \Dt v_j)$, up to a subsequence. In what follows, the notation $\ep\to0$ will refer to a subsequence for which these convergences are true for all $j\in\ldb -J,J\rdb$.

    Since the minimum of Lipschitz functions is Lipschitz and the index set $\ldb-J, J\rdb$ is finite, the minimum is taken over a set of uniformly bounded and Lipschitz functions. Consequently, there exists a function $x\mapsto m^n(x)$ such that
    \begin{equation*}
        m^{\ep,n}(x)\underset{\ep\to0}{\to}m^{n}(x),
    \end{equation*}
    uniformly on $[-A,A]$. It remains to identify the limit $m^n(x)$. As, for all $j\in\ldb-J,J\rdb$, the convergence of $(\varphi^{\ep,n}(\cdot-\Dt v_j))_\ep$ towards $\varphi^n(\cdot-\Dt v_j)$ is uniform on $[-A,A]$, for any $\eta > 0$, there exists $\ep_0 > 0$ such that for all $\ep < \ep_0$, $j \in \ldb-J, J\rdb$, $n \in \ldb1, N_t\rdb$, and $x \in [-A, A]$,
    \begin{equation*}
        \left|\varphi_j^{\ep,n}(x-\Dt v_j)-\varphi_j^{n}(x-\Dt v_j)\right|\leq\eta.
    \end{equation*}
    Taking the minimum over $j\in\ldb-J,J\rdb$ then gives
    \begin{equation*}
        \left|m^{\ep,n}(x)-\underset{j\in\ldb-J,J\rdb}{\min}\lbr\varphi_j^{n}(x-\Dt v_j)\rbr\right|\leq\eta.
    \end{equation*}
    Therefore, $m^{\ep,n}(x)=\min_{j \in \Dv \Z} \lbr \varphi_j^n(x - \Dt v_j) \rbr$. This convergence, together with \eqref{eq:convExpmu}, implies
    \begin{equation}\label{eq:convPetitmFinal}
        \mu^{\ep,n+1}(x)\underset{\ep\to 0}{\to}\underset{j \in \Dv \Z}{\min} \lbr \varphi_j^n(x - \Dt v_j) \rbr,
    \end{equation}
    locally uniformly towards $\mu^{n+1}(x)$.

    \details{
        We have a minimization over $\ldb-J,J\rdb\subset\Z$ a finite set, and $\varphi_{j}^{\ep, n}(x-\Dt v_{j})\underset{\ep\to0}{\to}\varphi_{j}^{n}(x-\Dt v_{j})$ uniformly on $I$ for all $j\in\ldb-J,J\rdb$. Let us define:
        \begin{equation*}
        l^\ep(x) := \min_{j\in\ldb-J,J\rdb} \varphi_{j}^{\ep, n}(x-\Dt v_{j}), \quad l(x) := \min_{j\in\ldb-J,J\rdb} \varphi_{j}^{n}(x-\Dt v_{j})
        \end{equation*}
        We want to show that $l^\ep(x) \to l(x)$. The uniform convergence on $\overbar{X}$ implies, $\forall\eta>0$, that $\forall j\in\ldb-J,J\rdb$ fixed, $\exists\ep_j $ such that for $ \ep \leq \ep_j $,  
        \begin{equation*}
        |\varphi_{j}^{\ep, n}(x-\Dt v_{j}) - \varphi_{j}^{n}(x-\Dt v_{j})| < \eta.
        \end{equation*}
        Since the minimization set is finite (here, finiteness is mandatory), one can take
        \begin{equation*}
        \ep_0 := \min_{j\in\ldb-J,J\rdb} \ep_j,
        \end{equation*}
        and the uniform convergence yields, for any $\eta > 0$, $\exists \ep_0$, $\ep \leq \ep_0$ and all $j\in\ldb-J,J\rdb$:
        \begin{equation*}
        |\varphi_{j}^{\ep, n}(x-\Dt v_{j}) - \varphi_{j}^{n}(x-\Dt v_{j})| < \eta
        \end{equation*}
        Then, for $\ep \leq \ep_0$,
        \begin{equation*}
        \varphi_{j}^{n}(x-\Dt v_{j}) - \eta < \varphi_{j}^{\ep, n}(x-\Dt v_{j}) < \varphi_{j}^{n}(x-\Dt v_{j}) + \eta \quad \text{for all } j\in\ldb-J,J\rdb
        \end{equation*}
        Taking the minimum over $j$ on both sides gives:
        \begin{equation*}
        \min_{j\in\ldb-J,J\rdb} \varphi_{j}^{n}(x-\Dt v_{j}) - \eta < \min_{j\in\ldb-J,J\rdb} \varphi_{j}^{\ep, n}(x-\Dt v_{j})< \min_{j\in\ldb-J,J\rdb} \varphi_{j}^{n}(x-\Dt v_{j}) + \eta
        \end{equation*}
        so
        \begin{equation*}
        |l^\ep(x) - l(x)| < \eta
        \end{equation*}
        Hence, $l^\ep \to l$. Moreover, this convergence is also locally uniform on $\overbar{X}$. Indeed, local uniform convergence of $\left(\varphi_j^{\ep,n}(x)\right)_{\underset{j\in\Z}{\ep>0}}$ implies that $\forall\eta>0$, $\exists \ep_0>0$, $\forall\ep<\ep_0$, $\forall x\in\overbar{X}$, $\forall j\in\ldb-J,J\rdb$, $\forall n\in\ldb1,N_t\rdb$:
        \begin{equation*}
            \left|\varphi_j^{\ep,n}(x-\Dt v_{j})-\varphi_j^{n}(x-\Dt v_{j})\right|<\eta.
        \end{equation*}
        Therefore, $\forall\eta>0$, $\exists \ep_0>0$, $\forall\ep<\ep_0$, $\forall x\in\overbar{X}$, $\forall j\in\ldb-J,J\rdb$, $\forall n\in\ldb1,N_t\rdb$ one has
        \begin{equation*}
            \begin{aligned}
                \bigg|\min_{j\in\ldb-J,J\rdb}\lbr \varphi_j^{\ep,n}((x-\Dt v_{j}))\rbr&-\min_{j\in\ldb-J,J\rdb}\lbr \varphi_j^{n}((x-\Dt v_{j}))\rbr\bigg| \\ &\leq\max_{j\in\ldb-J,J\rdb}\lbr\left|\varphi_j^{\ep,n}((x-\Dt v_{j}))-\varphi_j^{n}((x-\Dt v_{j}))\right|\rbr<\eta.
            \end{aligned}
        \end{equation*}
    }
    
    We now turn our attention to the limit of $M_j^{\ep,n}(x)$. Using the same argument as before, uniform continuity of $(x,y)\mapsto\min\lbr x,y\rbr$ preserves the uniform convergence on $ x\in[-A,A]$, and we can pass to the limit in \eqref{eq:grandMSemiTV} giving for all $x\in[-A,A]$, $j\in\ldb-J,J\rdb$ and $n\in\ldb1,N_t\rdb$:
    \begin{equation}\label{eq:convGrandmFinal}
        M_j^{\ep,n}(x)\details{=\min\lbr \varphi_{j}^{\ep,n}(x-\Dt v_j) + \Dt ,\, \frac{v_j^2}{2}+\mu^{\ep,n+1}(x)\rbr}\underset{\ep\to0}{\to} \min\lbr \varphi_{j}^{n}(x-\Dt v_j) + \Dt ,\, \frac{v_j^2}{2}+\mu^{n+1}(x)\rbr,
    \end{equation}
    which is the definition of $\varphi^{n+1}_j(x)$. Finally, combining the convergence of \eqref{eq:convExpPhi}, \eqref{eq:convPetitmFinal} and \eqref{eq:convGrandmFinal} we can conclude on the convergence of \eqref{eq:schemeMuPhiSemiTV}, as $\ep\to0$, towards the semi-discrete limit scheme:
    \begin{equation}\label{eq:VariDiscSemiTV}\tag{$\mathcal{S}^0_{\Dt,\Dv}$}
        \lbr\begin{aligned}
            \mu^{n+1}(x)&=\underset{j\in\Z}{\min} \lbr \varphi_{j}^{n}(x-\Dt v_j)\rbr,\\
            \varphi_{j}^{n+1}(x) &= \min\lbr \varphi_{j}^{n}(x-\Dt v_j) + \Dt ,\, \frac{v_j^2}{2}+\mu^{n+1}(x)\rbr.
        \end{aligned}\right.
    \end{equation}
    and this convergence is locally uniform in $x\in\R$. It remains to remove the restriction “up to a subsequence” in the convergence as $\ep \to 0$. Since \eqref{eq:VariDiscSemiTV} is an explicit scheme, and since, for all $j \in \Z$ and $x \in \R$, the initial data $\varphi_{\tin,j}(x)$ is assumed to be independent of $\ep$, this restriction can be lifted. If, instead, the initial data does depend on $\ep$ — that is, if one further assume that $\varphi_{\tin,j}^\ep(x) \to \varphi_{\tin,j}(x)$ locally uniformly in $x$ as $\ep \to 0$ — then all subsequences satisfy the same explicit induction relation. Consequently, the sequences $\varphi_{j}^{\ep,n+1}$ and $\mu_{j}^{\ep,n+1}$ converge locally uniformly when $\ep\to0$ and their limits satisfy \eqref{eq:VariDiscSemiTV}. 
\end{proof}
The proof of Theorem~\ref{thm:AP} can be extended to the fully discrete setting following Remark~\ref{rem:convEpsFull} by using the properties of the linear interpolation \eqref{eq:interpolation}.

\section{Properties of the asymptotic scheme}\label{sec:AsymptProp}
In this section, we again focus on the semi discrete scheme \eqref{eq:VariDiscSemiTV}. The fully discrete scheme \eqref{eq:VariDiscFull} mainly adds technicalities. We start by showing stability properties satisfied by \eqref{eq:VariDiscSemiTV} in Section~\ref{subsec:Variprop}. Then, we present several equivalent reformulations of \eqref{eq:VariDiscSemiTV}. In particular, following the continuous setting, a discrete representation formula is introduced in Section~\ref{subsec:DerivRep}.

\subsection{Stability properties of the limit scheme}\label{subsec:Variprop}
In this section, we prove stability properties of the limit scheme \eqref{eq:VariDiscSemiTV}, such as the preservation of the equilibrium and maximum principle that are classical in the study of numerical schemes for Hamilton-Jacobi equations \cite{CrandallLions1984}. Note that the same procedures can be adapted to the fully discrete case thanks to the monotonicity property of the first order interpolation \eqref{eq:interpolation}, and we refer back to Remark~\ref{rem:convEpsFull} for more details. The section is concluded with two Lemmas that emphasize the link between \eqref{eq:VariDiscSemiTV} and \eqref{eq:Vari}.

\begin{lemma}\label{lem:WBLim}
    The semi-discrete scheme \eqref{eq:VariDiscSemiTV} preserves the equilibrium~$\left(v_j^2/2\right)_{j\in\Z}$: if~$\varphi_\tin(x,v) = v^2/2$ for all $x\in\R$ and $v\in\R$, then 
    \begin{equation*}
        \forall n\in\ldb1,N_t\rdb,\,\forall j\in\Z,\,\forall x\in\R,\quad \varphi_j^{n}(x) = \frac{v_j^2}{2}.
    \end{equation*}
\end{lemma}
\details{
    \begin{proof}
        Let $\varphi_{j}^n(x)=\frac{v_j^2}{2}$. Since $\varphi_{j}^n(x)\geq0$ $\forall j\in\Z$ and $v_{0}=0$, $\mu^{n+1}(x)\details{=\underset{j\in\Z}{\min}\lbr\frac{v_j^2}{2}\rbr}= 0$. Plugging this relation into the second equation of \eqref{eq:VariDiscSemiTV} then yields
        \begin{equation*}
            \varphi_{j}^{n+1}(x) = \min\lbr \frac{v_j^2}{2} + \Dt ,\, \frac{v_j^2}{2}\rbr .
        \end{equation*}
        Since $\Dt\geq 0$, we obtain that $\varphi_{j}^{n+1}(x)=\frac{v_j^2}{2}$.
    \end{proof} 
}
\begin{lemma}\label{lem:MonoTransConstLim}
    Let $(n,j)\in\ldb1,N_t\rdb\times\Z$ and $x\in\R$. Let $\left(\mu^{n}(x),\varphi_j^{n}(x)\right)_{n\in \ldb1,N_t\rdb,j\in\Z}$ be a solution to \eqref{eq:VariDiscSemiTV} with initial data \eqref{eq:initSemiTV} and denote $\left(\psi^{n}(x)\right)_{n\in \ldb1,N_t\rdb,j\in\Z}$ another solution to \eqref{eq:VariDiscSemiTV}. For any function $\psi_j^{n}(x)$, we use the notation
    \begin{equation*}
            \nu^{n}(x)= \underset{j\in\Z}{\min} \lbr \psi_{j}^{n-1}(x-\Dt v_j)\rbr.
    \end{equation*}
    Then, the following properties hold:
    \begin{enumerate}[label=\roman*)]
        \item \textbf{Monotonicity.} Suppose that $\varphi_{\tin}\leq\psi_{\tin}$. Then, for all $j\in\Z$, $x\in\R$ and for all $n\in\ldb1,N_t\rdb$,
        \begin{align*}
            &\mu^{n}(x)\leq\nu^{n}(x),\\ 
            &\varphi_j^{n}(x)\leq\psi_j^{n}(x).
        \end{align*}
        \item \textbf{Invariance by translation.} Let $y\in\R$ and suppose that $\psi_{\tin}=\varphi_{\tin}(\cdot+y,\cdot)$. Then, for all $j\in\Z$, $x\in\R$ and for all $n\in\ldb1,N_t\rdb$,
        \begin{align*}
            &\nu^{n}(x)=\mu^{n}(x+y),\\ 
            &\psi_j^{n}(x)=\varphi_j^{n}(x+y).
        \end{align*}
        \item \textbf{Commutation with constants.} Let $K\in\R$ and suppose that $\psi_{\tin}=\varphi_{\tin}+K$. Then, for all $j\in\Z$, $x\in\R$ and for all $n\in\ldb1,N_t\rdb$,
        \begin{align*}
            &\nu^{n}(x)=\mu^{n}(x)+K,\\ 
            &\psi_j^{n}(x)=\varphi_j^{n}(x)+K.
        \end{align*}
    \end{enumerate}
\end{lemma}
\details{
    \begin{proof}  
        \begin{enumerate}[label=\roman*)]
            \item We proceed by induction assume that $\psi_{\tin,j}(x)\leq\varphi_{\tin,j}(x)$ for all $j\in\Z$ and $x\in\R$. Let $\psi_{j}^{n}(x)$ be a solution to \eqref{eq:VariDiscSemiTV}. Then, \eqref{eq:monoLimMu} is a direct consequence of the definitions:
            \begin{equation*}
                \nu^{n+1}(x) = \underset{j\in\Z}{\min}\lbr \psi_{j}^{n}(x-\Dt v_j)\rbr \leq\underset{j\in\Z}{\min}\lbr \varphi_{j}^{n}(x-\Dt v_j)\rbr  =\mu^{n+1}(x).
            \end{equation*}
            Inequality \eqref{eq:monoLimPhi} is obtained similarly:
            \begin{align*}
                \psi_{j}^{n+1}(x) &= \min\lbr \psi_{j}^{n}(x-\Dt v_j) + \Dt ,\, \frac{v_j^2}{2}+\nu^{n+1}(x)\rbr \\
                &\leq \min\lbr \varphi_{j}^{n}(x-\Dt v_j) + \Dt ,\, \frac{v_j^2}{2}+\mu^{n+1}(x)\rbr  = \varphi_{j}^{n+1}(x),
            \end{align*}
            which concludes the induction.
            \item Let $y\in\R$ and define $\psi_{j}^{n}(x)=\varphi_{j}^{n}(x+y)$. We look for the relation satisfied by $\psi_{j}^{n+1}(x)$. First, \eqref{eq:TransLimMu} follows from the definition of scheme \eqref{eq:VariDiscSemiTV}:
            \begin{equation}\label{eq:TransLimMuBis}
                \nu^{n+1}(x) = \underset{j\in\Z}{\min}\lbr \psi_{j}^{n}(x-\Dt v_j)\rbr =\underset{j\in\Z}{\min}\lbr \varphi_{j}^{n}(x+y-\Dt v_j)\rbr =\mu^{n+1}(x+y).
            \end{equation}
            Secondly, using \eqref{eq:TransLimMuBis}, the same argument yields \eqref{eq:TransLimPhi}:
            \begin{align*}
                \psi_{j}^{n+1}(x) &= \min\lbr \varphi_{j}^{n}(x+y-\Dt v_j) + \Dt ,\, \frac{v_j^2}{2}+\mu^{n+1}(x+y)\rbr\\
                \psi_{j}^{n+1}(x) &= \min\lbr \psi_{j}^{n}(x-\Dt v_j) + \Dt ,\, \frac{v_j^2}{2}+\nu^{n+1}(x)\rbr .
            \end{align*}
            \item Let $K\in\R$. We set $\psi_{ij}^n=\varphi_{j}^n(x) + K$ Then,
            \begin{equation}\label{eq:ConstLimMuBis}
                \nu^{n+1}(x) = \underset{j\in\Z}{\min}\lbr \psi_{j}^{n}(x-\Dt v_j)\rbr =K+\underset{j\in\Z}{\min}\lbr \varphi_{j}^{n}(x-\Dt v_j)\rbr =K+\mu^{n+1}(x).
            \end{equation}
            Since $K$ is a constant, one can in particular add it to the second equation of \eqref{eq:VariDiscSemiTV}. In addition, using \eqref{eq:ConstLimMuBis} one obtains \eqref{eq:ConstLimPhi}:
            \begin{align*}
                \varphi_{j}^{n+1}(x) + K&= \min\lbr \varphi_{j}^{n}(x-\Dt v_j) + \Dt ,\, \frac{v_j^2}{2}+\mu^{n+1}(x)\rbr  + K\\
                \psi_{j}^{n+1}(x) &= \min\lbr \psi_{j}^{n}(x-\Dt v_j) + \Dt ,\, \frac{v_j^2}{2}+\nu^{n+1}(x)\rbr .
            \end{align*}
        \end{enumerate}
    \end{proof}
}

From Lemma~\ref{lem:MonoTransConstLim}, one can then deduce the exact propagation of Lipschitz bounds by the semi-discrete numerical scheme \eqref{eq:VariDiscSemiTV} (see \cite{CrandallLions1984}). In addition, one has:
\begin{corollary}\label{cor:LipBoundLim}
    For all $n\in\ldb1,N_t\rdb$, the solution to \eqref{eq:VariDiscSemiTV} satisfies Assumption~\ref{Ass:LinfBounds}.
\end{corollary}
The proofs of Lemmas~\ref{lem:MonoTransConstLim} and~\ref{lem:WBLim} are done by induction, using the definition of \eqref{eq:VariDiscSemiTV} and are therefore not detailed here.

We conclude this section with two lemmas linking \eqref{eq:VariDiscSemiTV} and \eqref{eq:Vari}. First, we prove that \eqref{eq:VariDiscSemiTV} also satisfies a discrete version of the condition $\dt \min_{v\in\R^{d_v}} \lbr\varphi\rbr \leq 0$ from \eqref{eq:Vari}. To show this result, we start by introducing the following property of \eqref{eq:VariDiscSemiTV}.
\begin{lemma}\label{lem:VariLink}
    Let $\Dt,\Dx,\Dv>0$. Let $\left(\varphi_{j}^{n}(x)\right)_{n\in\ldb 1,N_t\rdb, j\in\Z}$ and $\left(\mu^{n}(x)\right)_{n\in\ldb 1,N_t\rdb}$ be solution to \eqref{eq:VariDiscSemiTV}. Then,
    \begin{equation}\label{eq:LinkMu}
        \forall n\in\ldb 1,N_t\rdb,\quad\mu^{n}(x)=\underset{j\in\Z}{\min} \lbr \varphi_{j}^{n}(x)\rbr.
    \end{equation}
\end{lemma}
\begin{proof}
    Starting from the second line of \eqref{eq:VariDiscSemiTV}, the minimum in $j$ is taken, and since $\Dt\geq0$, one obtains
    \begin{equation*}
        \begin{aligned}
            \underset{j\in\Z}{\min} \lbr \varphi_{j}^{n+1}(x)\rbr &=\min\lbr \underset{j\in\Z}{\min} \lbr\varphi_{j}^{n}(x-\Dt v_j)\rbr + \Dt,\,\underset{j\in\Z}{\min}\lbr \frac{v_j^2}{2}\rbr +\mu^{n+1}(x) \rbr \\
            &=\mu^{n+1}(x),
        \end{aligned}
    \end{equation*}
    \details{
        \begin{equation*}
            \begin{aligned}
                \underset{j\in\Z}{\min} \lbr \varphi_{j}^{n+1}(x)\rbr  &= \underset{j\in\Z}{\min} \min\lbr \varphi_{j}^{n}(x-\Dt v_j) + \Dt ,\, \frac{v_j^2}{2}+\mu^{n+1}(x)\rbr .\\
                &=\min\lbr \underset{j\in\Z}{\min} \lbr\varphi_{j}^{n}(x-\Dt v_j)\rbr + \Dt,\,\underset{j\in\Z}{\min}\lbr \frac{v_j^2}{2}\rbr +\mu^{n+1}(x) \rbr \\
                &=\min\lbr \mu^{n+1}(x) + \Dt,\,\mu^{n+1}(x) \rbr \\
                &=\mu^{n+1}(x).
            \end{aligned}
        \end{equation*}
    }
\end{proof}

Using Lemma~\ref{lem:VariLink}, the following property, can now be stated.
\begin{lemma}\label{lem:DiscDTmuneg}
    Let $x\in\R$ and $\left(\mu^{n}(x), \varphi_{j}^{n}(x)\right)_{j\in\Z,n\in\ldb0,N_t\rdb}$ be a solution to \eqref{eq:VariDiscSemiTV}. Then, under condition \eqref{eq:stabJumpCond},
    \begin{equation}\label{eq:DiscDTmuneg}
        \forall x\in\R,\,n\in\ldb0, N_t-1\rdb,\quad \mu^{n+1}(x)-\mu^{n}(x)\leq0.
    \end{equation}
    \details{
        Moreover, assuming that $0\in\underset{j\in\Z}{\argmin}\lbr\varphi_{\tin,j}(x)\rbr$, $j=0$ is the only minimum: for all $\,x\in\R,\,n\in\ldb1,N_t\rdb$, $\underset{j\in\Z}{\argmin}\lbr\varphi_{j}^{n+1}(x)\rbr=\{0\}$.
    }
\end{lemma}
\details{
    The relation \eqref{eq:DiscDTmuneg} is precisely the discrete equivalent of $\dt \underset{v\in\R^{d_v}}{\min} u \leq 0$. However, a key difference is that, in the discrete setting, there is always only one minimum that is always achieved at $j=0$.
}
\begin{proof}
    Let $x\in\R$. Starting from \eqref{eq:LinkMu}, it gives in particular that,
    \begin{equation*}
        \forall j\in\Z\quad\mu^{n+1}(x) \leq \varphi_{j}^{n+1}(x).
    \end{equation*}
    Then, using the second line of \eqref{eq:VariDiscSemiTV}, we deduce that for all $j\in\Z$,
    \begin{equation}\label{eq:linkPhiBound1}
        \varphi_{j}^{n+1}(x) \leq \mu^{n+1}(x) + \frac{v_j^2}{2}.
    \end{equation}
    In particular, combining the first line of \eqref{eq:VariDiscSemiTV} and \eqref{eq:linkPhiBound1} at time $n$, one obtains, for $j=0$
    \begin{equation*}
        \varphi_{0}^{n+1}(x) \leq \mu^{n+1}(x) \leq \varphi_{0}^{n}(x) \leq \mu^{n}(x),
    \end{equation*}
    which yields \eqref{eq:DiscDTmuneg}.
    
    \details{
        The relation \eqref{eq:dtminDiscDecreases} is precisely the discrete equivalent of $\dt \underset{v\in\R^{d_v}}{\min} u \leq 0$. It remains to show that this discrete derivative is $0$ if the minimum in velocity, for a fixed position $x\in\R$, is only achieved at $v_j=0$. We start by showing that $j=0$ is always the only minimum. Let $j^*\in\underset{j\in\Z}{\argmin}\lbr\varphi^{n+1}\rbr$. The relation \eqref{eq:LinkMu} then rewrites
            \begin{equation}\label{eq:jstarinargmin}
                \mu^{n+1}(x)=\varphi_{j^*}^{n+1}(x)
            \end{equation}
            Subtracting $\mu^{n+1}(x)$ from \eqref{eq:VariPhiSemiTV} evaluated at $j^*$, one obtains 
            \begin{equation}\label{eq:0isminimumtmp}
                \min\lbr \varphi_{j^*}^{n}(x-\Dt v_{j^*})-\mu^{n+1}(x) + \Dt ,\, \frac{v_{j^*}^2}{2}\rbr=0
            \end{equation}
            Since $j^*\in \underset{j\in\Z}{\argmin} \lbr\varphi^{n+1}\rbr$, we can use the first line of \eqref{eq:VariDiscSemiTV}. Together with equation \eqref{eq:0isminimumtmp}, this yields $0 \geq \min\lbr \Dt,\frac{v_{j^*}}{2} \rbr$. As $\Dt > 0$, this inequality can only be satisfied if $v_{j^*} = 0$, giving
            \begin{equation}\label{eq:0istheminimum}
                \forall x\in\R,\forall n\in\ldb0,N_t\rdb,\quad\underset{j\in\Z}{\argmin}\lbr\varphi^n(x)\rbr=\lbr0\rbr.
            \end{equation}    

            Suppose now that $\mu^{n+1}(x) < \mu^n(x)$. Then, at a given position $x$, as $\mu^n(x)= \varphi^n_0(x)$ thanks to \eqref{eq:0istheminimum}, and $\mu^{n+1}(x)=\underset{j\in\Z}{\min}\lbr \varphi^n_j(x-\Dt v_j)\rbr$, there exists $j^*\neq 0$ such that $\mu^{n+1}(x)= \varphi^n_{j^*}(x-\Dt v_{j^*})$. Thanks to the second line of \eqref{eq:VariDiscSemiTV} and \eqref{eq:stabJumpCond} one obtains
            \begin{equation*}
                \varphi^{n+1}_{j ^*}(x)\details{= \min\{ \Dt + \varphi^n_{j^*}(x-\Dt v_{j^*}), v_{j^*}^2/2 + \mu^{n+1}(x) \}} = \mu^{n+1}(x)+\min\lbr\Dt,(j ^*\Dv)^2/2 \rbr = \mu^{n+1}(x)+\Dt.
            \end{equation*}
    }
    
\end{proof}

To conclude this section, we present an equivalent reformulation of \eqref{eq:VariDiscSemiTV} that more closely mirrors the structure of the continuous equation \eqref{eq:Vari}.
\begin{lemma}
    Let $\left(\mu^{n}(x),\varphi_{j}^{n}(x)\right)_{j\in\Z, x\in\R, n\in\ldb 0,N_t\rdb}$ be solution to \eqref{eq:VariDiscSemiTV}. Then, it is solution to:
    \begin{equation*}
    \lbr \begin{aligned}
        &\mu^{n+1}(x)=\underset{j\in\Z}{\min} \lbr \varphi_{j}^{n}(x-\Dt v_j)\rbr ,\\
        &\max\lbr \frac{\varphi_{j}^{n+1}(x)-\varphi_{j}^{n}(x-\Dt v_j)}{\Dt} + 1 ,\, \varphi_{j}^{n+1}(x)-\frac{v_j^2}{2}-\mu^{n+1}(x)\rbr =0.
    \end{aligned}\right.
\end{equation*}
\end{lemma}
\details{
    \begin{proof}
        The formula is obtained from the second equation of \eqref{eq:VariDiscSemiTV} by subtracting $\varphi_j^{n+1}(x)$, multiplying by $-1$ and using that $\Dt>0$.
    \end{proof}  
}

\subsection{Towards a discrete representation formula}
To derive a semi-discrete representation formula for \eqref{eq:Vari}, we begin by introducing the necessary notations and definitions. These are based on the discussion in the introduction concerning the duality between Eulerian and Lagrangian perspectives. We first recall the notion of minimizing curves for \eqref{eq:Vari} defined in \cite{BouinCalvezGrenierNadin2023}, along with their associated continuous action. We then introduce corresponding concepts in the semi-discrete setting and recall the representation formula for \eqref{eq:Vari}.

\begin{definition}[Continuous minimizing curves]\label{def:pathCont}
    Let $\Sigma^{0,T}$ be the space of piecewise constant, c\`adl\`ag functions defined over the time interval $(0,T]$ taking values in $\mathbb{R}$. For any $\sigma \in \Sigma^{0,T}$, denote $s^0=0$ and let $\mathbf{s}=(s^k)_{0\leq k \leq N_{\ast}}$ be the $N_{\ast}\geq1$ times, of discontinuity of $\sigma$ in $(0,T]$, such that
    \begin{equation*}
        \sigma = w^0 \mathbf{1}_{(0,s_1)} + \sum_{k=1}^{N_{\ast}-1} w^k \mathbf{1}_{[s^k,s^{k+1})} + w^{N_{\ast}} \mathbf{1}_{[s^{N_{\ast}},T]}.
    \end{equation*}
    Then, for any starting point $y\in\R$, the piecewise linear curve $\gamma\in\R\times\Sigma^{0,T}$ is defined for all $\tau\in[0,T]$ as
    \begin{equation*}
        \gamma(\tau) = y + \int_0^\tau \sigma(\tau') d\tau'.
    \end{equation*}
    In addition, we denote by $\textbf{w}=(w^0,\dots,w^{N_{\ast}})$ the associated tuple of velocity jumps. It allows to characterize the path $\gamma$ by its starting point and the sequences $\mathbf{s}$ and $\mathbf{w}$ : $\gamma=(y,\mathbf{s},\mathbf{w})\in\R\times\Sigma^{0,T}$.
\end{definition}
From this definition, the action $\Ab_0^T[\gamma]$ of a path $\gamma$ is defined as follows.
\begin{definition}[Action of a piecewise linear curve (continuous case)]\label{def:contAction}
    Let $\gamma=(y,\mathbf{s},\mathbf{w})\in\R\times\Sigma^{0,T}$ be defined as in Definition~\ref{def:pathCont}. The continuous action of $\gamma$ on $(0,T]$ is defined as follows:
    \begin{equation}\label{eq:contAction}
        \Ab_0^T[\gamma] = \frac{1}{2} \sum_{k=1}^{N_{\ast}} |w^k|^2 + \sum_{k=0}^{N_{\ast}} (s^{k+1} - s^k) \mathbf{1}_{w^k \neq 0},
    \end{equation}
    with the convention $s^0 = 0$ and $s^{N_{\ast}+1} = T$.
\end{definition}
The action $\Ab$ describes the cost of a moving individual. The first part corresponds to a change in velocity that is instantaneous and induces a cost $|w|^2/2$, where $w$ is the new velocity. In addition, when the individual is subject to free transport, there is a running cost that depends only on the time spent traveling and not on the velocity at which it occurs.

With these definitions, we recall the following result from \cite[Theorem 1.8]{BouinCalvezGrenierNadin2023}, which introduces a representation formula and establishes the uniqueness of the viscosity solution to \eqref{eq:Vari}.
\begin{theorem}[(1.8 in \cite{BouinCalvezGrenierNadin2023}), Kinetic Hopf-Lax formula]\label{thm:kinHopfLax}
    Let $\varphi_\tin$ be a continuous function such that $\varphi_\tin-v^2/2\in L^\infty(\R\times\R)$. Then, the following representation formula
    \begin{equation}\label{eq:repFormContHL}
        \varphi(t,x,v) = \underset{\underset{\gamb(t) = x, \dot{\gamb}(t) = v}{\bar{\gamma}=(y,\mathbf{s},\mathbf{w})\in\R\times\Sigma^{0,T}}}{\inf} \lbr \varphi_\tin(y,w^0) + \Ab_0^t[\gamb] \rbr 
    \end{equation}
    is the viscosity solution of \eqref{eq:Vari} with initial data $\min\lbr\varphi_\tin, \min_{v\in\R} \lbr\varphi_\tin\rbr + \frac{v^2}{2}\rbr$.
\end{theorem}
In the discrete setting, analogous definitions can be introduced. A semi-discrete analog of \eqref{eq:repFormContHL} will be derived in the next section. We begin with time-discrete curves, where the velocity jump times are restricted to the grid $\left(n\Dt\right)_{n \in \ldb0, N_t\rdb}$. The following definition can be viewed as a discrete counterpart — or a special case — of Definition~\ref{def:pathCont}. It serves as a technical bridge between fully continuous minimizing curves and their time-velocity semi-discrete counterparts.
\begin{definition}[Time-discrete minimizing curves]\label{def:pathSemiT}
    Let $\Sigma^{0,T}_\Dt$ be the subset of $\Sigma^{0,T}$ such that for any $\sigma \in\Sigma^{0,T}_\Dt$, the $N_{\ast}$ times $\mathbf{s}=(s^k)_{1\le k\le N_{\ast}}$ of discontinuity of $\sigma \in (0,T]$ are located on the time grid: for all $k\in\ldb 1,N_{\ast}\rdb$, $s^k\in\Dt\N$. 
    Then, for any starting point $y\in\R$ and for any tuple $\mathbf{w}=(w^0, \dots,w^{N_{\ast}})$ of velocity jumps, the piecewise linear curve $\gamma=(y,\mathbf{s},\mathbf{w})\in\R\times\Sigma^{0,T}_\Dt$ is defined on the time grid $(t^n)_{0\le n\le N_t}$ as in Definition~\ref{def:pathCont}.
    
    As the discontinuities of the elements of $\Sigma^{0,T}_\Dt$ are located on the grid, the element $\sigma\in\Sigma^{0,T}_\Dt$ can also be described with all the values it takes on time steps $[k \Dt, (k+1)\Dt)$, for all $k\in\ldb0,N_t-1\rdb$, and at time $T=t^{N_t}$. Introducing $\mathtt{w}=(w^0,\dots,w^{N_t}) \in\R^{\ldb0,N_t\rdb}$, 
    $\sigma$ is defined as 
    \begin{equation*}
        \sigma = w^0 \mathds{1}_{(0,\Dt)} + \sum\limits_{k=1}^{N_t-1} w^k \mathds{1}_{[k\Dt,(k+1)\Dt)} + w^{N_t} \mathds{1}_{\{T\}},
    \end{equation*}
    and for any initial point $y\in\R$, the associated piecewise linear curves $\gamma =(y,\mathtt{w})\in\R\times \R^{\ldb1,N_t\rdb}$ is  defined on the time grid for all $n\in\ldb0,N_t\rdb$ by
    \begin{equation*}
        \gamma(t^n) = y +\Dt \sum\limits_{k=1}^n w^{k-1}.
    \end{equation*}
\end{definition}
In the definition above, $\mathtt{w}$ may contain repetitions of velocities on adjacent intervals, while $\mathbf{w}$ does not. The first convention will be used when the focus is on the times of discontinuities of the paths, whereas the second will be used when the focus will be on the sequence of velocities arising in the path. In what follows, the statement $\gamma=(y,\mathbf{s},\mathbf{w})\in\R\times\Sigma^{0,T}_\Dt$ refers to the first convention, while $\gamma=(y,\mathtt{w})\in\R\times\R^{\ldb 0,N_t\rdb}$ refers to the second.

Further discretizing the velocity variable on the grid $\Dv \Z$ leads to a semi-discrete in time and velocity setting. In this framework, one can define a specific class of discrete paths that form another special case of Definition~\ref{def:pathCont}. These paths are particularly suited to build a discrete analog of the representation formula \eqref{eq:repFormContHL}.
\begin{definition}[Time-velocity discrete minimizing curves]\label{def:pathSemiTV}
    Let $\Sigma^{0,T}_{\Dt,\Dv}$ be the subset of $\Sigma^{0,T}_\Dt$ such that for any $\sigma\in\Sigma^{0,T}_{\Dt,\Dv}$, the $N_{\ast}+1$ velocities $\mathbf{w}=(w^k)_{0\le k\le N_{\ast}}$ are located on the velocity grid : for all $k\in\ldb 0,N_{\ast}\rdb$, $w^k\in\Dv\Z$.
    Then for any starting point $y\in\R$, the piecewise linear curve $\gamma=(y,\mathbf{s},\mathbf{w})\in\R\times\Sigma^{0,T}_{\Dt,\Dv}$ is defined on the time grid as in Definition~\ref{def:pathCont}.
\end{definition}

As in Definition~\ref{def:pathSemiT}, the path $\gamma$ can also be defined by its initial point $y\in\R$ and the sequence of velocities $\mathtt{w}=(w^k)_{0\le k\le N_t}\in\Dv\Z^{\ldb 0,N_t\rdb}$. We will denote $\gamma=(y,\mathtt{w})\in\R\times\Dv\Z^{\ldb 0,N_t\rdb}$ when we use this convention. With both conventions, the velocities $w^k$ arising in the time-velocitie discrete minimizing curves of Definition~\ref{def:pathSemiTV} lie on the velocity grid. Therefore, for all $k$, we introduce the index $j(k)\in\Z$, such that $w^k=v_{j(k)}$.

Since Definitions~\ref{def:pathSemiT} and \ref{def:pathSemiTV} are particular cases of Definition~\ref{def:pathCont}, the continuous action $\Ab$ can still be evaluated along discrete paths. However, the goal of the next section is to construct a fully discrete counterpart to the representation formula \eqref{eq:repFormContHL}, which is equivalent to the semi-discrete system \eqref{eq:VariDiscSemiTV}. To this end, we introduce a discrete analog of Definition~\ref{def:contAction}, tailored specifically to time-discrete paths and adapted to the structure of \eqref{eq:VariDiscSemiTV}.
\begin{definition}[Action of a piecewise linear curve (discrete case)]\label{def:discAction}
    Let $\gamma=(y,\mathtt{w})\in\R\times\R^{\ldb0,N_t\rdb}$ be defined as in Definition~\ref{def:pathSemiT}. The discrete action of $\gamma$ on $(0,T]$ is defined as follows:
    \begin{equation*}
        \At_0^T[\gamma] = \sum_{k=0}^{N_t-1}\,A_{w^k}^{w^{k+1}},
    \end{equation*}
    where for all $(v,w)\in\R^2$, $A_{w}^{v}$ denotes the cost of a jump from velocity $w$ to $v$, and is defined as
    \begin{equation}\label{eq:discAction}
        A_{w}^{v} =\lbr \begin{aligned}
        & \Dt \mathds{1}_{v\neq0}\quad&\text{if}\,v=w,\\
        & \frac{|v|^2}{2} \quad&\text{if}\,v\neq w.
        \end{aligned}\right.
    \end{equation}
\end{definition}
This discrete action represents the cost of transitioning from velocity $w$ to velocity $v$ on a discrete time interval of size $\Dt$. The first case indicates that if there is no change in velocity, the cost is simply the time spent in free transport, except when $v=0$. The second case shows that the cost of changing velocity depends only on the new velocity and not on the initial one.

While the continuous action \eqref{eq:contAction} and the discrete one \eqref{eq:discAction} share similarities, such as the cost of a velocity jump and free transport, there is a key difference between the two descriptions. In the continuous setting, a jump occurs instantaneously, whereas in the discrete case, it takes an entire time interval $\Dt$ to get to a new velocity. This discrepancy introduces an error of order $\Dt$ per jump (if it is not at the end of the interval) when comparing the cost of the same path under the continuous and discrete actions. This is the content of the following lemma.
\begin{lemma}\label{lem:discrepancyDt}
    Let $\gamma=(y,\mathbf{s},\mathbf{w})\in\R\times\Sigma^{0,T}_\Dt$, and let $N_{\ast}$ the number of velocity jumps in $\gamma$. Therefore, with the convention 
    $s^0=0$, $\mathbf{s}=(s^0, \dots,s^{N_{\ast}})$ and $\mathbf{w}=(w^0,\dots,w^{N_{\ast}})$ are such that 
    \begin{equation*}
        \forall k\in\ldb 1,N_{\ast}\rdb, \quad s^k-s^{k-1} \ge \Dt, \,\,\text{and}\,\, w^k\neq w^{k-1}.
    \end{equation*}
    The action $\tilde{A}[\gamma]$ of Definition~\ref{def:discAction} can be rewritten 
    \begin{equation}\label{eq:tildeA_Convention1}
        \tilde{A}[\gamma] = \sum\limits_{k=1}^{N_{\ast}} \frac{|w^k|^2}{2}+  \sum\limits_{k=0}^{N_{\ast}-2} (s^{k+1}-s^{k}-\Dt) \mathds{1}_{w^{k+1}\neq 0} + \max\lbr T-s^{N_{\ast}}-\Dt,0\rbr \mathds{1}_{w^{N_{\ast}}\neq 0}.
    \end{equation}
    Moreover, 
    \begin{equation*}
        \left|\bar{A}[\gamma]-\tilde{A}[\gamma]\right|\le N_{\ast}\Dt.
    \end{equation*}
\end{lemma}
\begin{proof}
    Equation \eqref{eq:tildeA_Convention1} is a consequence of the definition of $\tilde{A}[\gamma]$, obtained by rewriting the velocities $\textbf{w}$ of $\gamma$ as a sequence $\mathtt{w}\in\R^{\ldb 0,N_t\rdb}$ defined on the time grid. The last term is taken apart to distinguish the cases $s^{N_{\ast}}=T$ and $s^{N_{\ast}}\neq T$. The difference between $\tilde{A}$ and $\bar{A}$ follows immediately.
\end{proof}

\subsection{Derivation of the semi-discrete representation formula}\label{subsec:DerivRep}
We now introduce a semi-discrete representation formula for \eqref{eq:VariDiscSemiTV}. In line with the continuous setting, this formulation involves the discrete action \eqref{eq:discAction} along with the initial datum evaluated at the foot of a specific characteristic curve. We have the following proposition:
\begin{proposition}[Semi-discrete representation formula]\label{prop:repFormDiscSemi}
    Assume that the time step $\Dt$ and the velocity step $\Dv$ satisfy the relation \eqref{eq:stabJumpCond}, that is $\Dt \leq \Dv^2/2$, and assume morever that the initial data satisfies Assumption~\ref{Ass:LinfBounds}. Then, the solution $\left(\varphi_{j}^{N_t}(x)\right)_{j\in\Z}$ to \eqref{eq:VariDiscSemiTV}, whose existence is known, is also solution to the following discrete representation formula:
    \begin{equation}\label{eq:repFormDiscSemi}
        \varphi_{j}^{N_t}(x) = \underset{\underset{\gamma(T)=x,\,w^{N_t}=v_j}{\gamma=(y,\mathtt{w})\in\R\times\Dv\Z^{\ldb 0,N_t\rdb}}}{\mathrm{inf}}\lbr  \varphi_{\tin,j(0)}(y)+\At_0^T[\gamma]\rbr,
    \end{equation}
    where $\gamma$ is defined in Definition~\ref{def:pathSemiTV}, the discrete action $\At_0^T$ is given in Definition~\ref{def:contAction}, and we use the notation $v_{j(0)} = w^0$.
\end{proposition}
The assumption on the time step can be understood as the requirement that, at the discrete level, free transport and a velocity jump cannot occur simultaneously. This condition is purely discrete and specific to both the model and this reformulation. It highlights a key distinction from the continuous setting, where jumps are instantaneous and independent of transport. 

Moreover, Assumption~\ref{Ass:LinfBounds} on the initial data, together with the definition~\eqref{eq:discAction} of the discrete action, ensure coercivity of the problem. In particular, the infimum in~\eqref{eq:repFormDiscSemi} is well defined.

The proof of Proposition~\ref{prop:repFormDiscSemi} proceeds by induction, using the following lemma.
\begin{lemma}\label{lem:repFormDisc0}
    Assume that the time step $\Dt$ and the velocity step $\Dv$ satisfy the relation \eqref{eq:stabJumpCond}. Assume moreover that the initial data satisfies Assumption~\ref{Ass:LinfBounds}. Then, the solution $\left(\varphi_{j}^{n+1}(x)\right)_{j\in\Z}$ to \eqref{eq:VariDiscSemiTV} is solution to the following discrete representation formula:
    \begin{equation}\label{eq:repFormDisc0}
        \varphi_{j}^{n+1}(x) = \underset{w^n\in\Dv\Z}{\mathrm{inf}}\lbr  \varphi_{j(n)}^n(x-\Dt w^n)+ A_{w^n}^{v_j}\rbr,
    \end{equation}
    where $A_{w^n}^{v_j}$ is defined in \eqref{eq:discAction} and $j(n)$ in Definition~\ref{def:pathSemiTV}.
\end{lemma}
\begin{proof}
    The proof consists in showing that under condition \eqref{eq:stabJumpCond}, $\varphi^{n+1}_j(x)$ defined in \eqref{eq:repFormDisc0} coincides with $\varphi^{n+1}_j(x)$ defined with \eqref{eq:VariDiscSemiTV}. Start by considering $\left(U_{j}^n(x)\right)_{j\in\Z, x\in\R}$ satisfying
    \begin{equation}\label{eq:repFormDisc00}
        U_{j}^{n+1}(x) = \underset{w^n\in\Dv\Z}{\mathrm{inf}}\lbr  U_{j(n)}^n(x-\Dt w^n) + A_{w^n}^{v_j}\rbr.
    \end{equation}
    Let us now assume that, for all $j \in \Z$ and $x \in \R$, we have $U_j^n(x) = \varphi_j^n(x)$, and that these functions satisfy Assumption~\ref{Ass:LinfBounds}. Our goal is to show that $U_j^{n+1}(x) = \varphi_j^{n+1}(x)$. Let $x\in\R$ be fixed. By separating the cases $v_j=w^n$ ($j(n)=j$) and $v_j\neq w^n$ ($j(n)\neq j$), equation \eqref{eq:repFormDisc00} can be reformulated as follows:
    \begin{equation}\label{eq:repFormDisc1}
        U_{j}^{n+1}(x) = \min\lbr \varphi^n_{j}(x-\Dt v_j) + \Dt \mathds{1}_{v_j\neq0},\,\underset{w^n\neq v_j}{\mathrm{inf}}\lbr  \varphi_{j(n)}^n(x-\Dt w^n) + \frac{v_j^2}{2}\rbr \rbr .
    \end{equation}
    We now consider separately the cases $v_j=0$ and $v_j\neq0$.
    \begin{itemize}
        \item \textbf{Case $v_j\neq0$.} In that cas, one has:
        \begin{equation}\label{eq:schemeUtmp}
            U_{j}^{n+1}(x) = \min\lbr \varphi_{j}^n(x-\Dt v_j) + \Dt ,\,\underset{w^n\neq v_j}{\mathrm{inf}}\lbr  \varphi^n_{j(n)}(x-\Dt w^n) + \frac{v_j^2}{2}\rbr \rbr .
        \end{equation}
        Using condition \eqref{eq:stabJumpCond} and since $v_j\neq0$, the following inequality holds
        \begin{equation}\label{eq:stabJumpCondBis}
            \Dt \leq \frac{v_j^2}{2}.
        \end{equation}
        Moreover, the case $w^n = v_j$ can be included in the infimum in \eqref{eq:schemeUtmp}, as it does not attain the minimum: the first term in the expression is the minimizer, as ensured by \eqref{eq:stabJumpCondBis}. This yields:
        \begin{equation}\label{eq:schemeUSemiTV}
            U_{j}^{n+1}(x) = \min\lbr \varphi_{j}^n(x-\Dt v_j) + \Dt ,\,\underset{w^n\in\Dv\Z}{\mathrm{inf}}\lbr  \varphi^n_{j(n)}(x-\Dt w^n) \rbr + \frac{v_j^2}{2}\rbr .
        \end{equation}
        Since $\varphi^n$ satisfies Assumption~\ref{Ass:LinfBounds}, the infimum in \eqref{eq:schemeUSemiTV} is in fact a minimum. Consequently, as $v_j \neq 0$, we observe that both $\varphi_j^{n+1}(x)$ and $U_j^{n+1}(x)$ satisfy the same induction relation. It follows that $\varphi_j^{n+1}(x) = U_j^{n+1}(x)$ for all $j \in \Z^*$ and $x \in \R$.

        \item \textbf{Case $v_j=0$ ($j=0$).} On the one hand, using the notations from Definition~\ref{def:pathSemiTV} for the discrete velocites, the first line of \eqref{eq:VariDiscSemiTV} rewrites
        \begin{equation}\label{eq:DerivRepPhi0Bis}
            \mu^{n+1}(x) = \underset{w^n\in\Dv\Z}{\min} \lbr \varphi_{j(n)}^{n}(x-\Dt w^n)\rbr.
        \end{equation}
        On the other hand, evaluating \eqref{eq:repFormDisc1} at $j=0$ gives
        \begin{equation}\label{eq:DerivRepU0}
            U_{0}^{n+1}(x) = \min\lbr U_{0}^n(x) ,\,\underset{w^n\neq 0}{\mathrm{inf}}\lbr  U_{j(n)}^{n}(x-\Dt w^n) \rbr \rbr.
        \end{equation}
        As $U^n$ satisfies Assumption~\ref{Ass:LinfBounds}, the infimum is also attained meaning it is actually a minimum. Then, comparing \eqref{eq:DerivRepPhi0Bis} and \eqref{eq:DerivRepU0}, one concludes that $U_{0}^{n+1}(x)=\varphi^{n+1}_0(x)$ for all $x\in\R$.
    \end{itemize}
\end{proof}

\begin{proof}[Proof of Proposition~\ref{prop:repFormDiscSemi}]
    Thanks to Lemma~\ref{lem:repFormDisc0}, the following relation is obtained by induction:
    \begin{equation*}
        \varphi_{j}^{N_t}(x) = \underset{\textbf{w}\in\Dv\Z^{\ldb 0,N_t\rdb}}{\mathrm{inf}}\lbr  \varphi_{j(0)}^0\left(x-\Dt\sum_{k=0}^{N_t-1} w^k\right)+\sum_{k=0}^{N_t-1}A_{w^{k}}^{w^{k+1}}\rbr ,
    \end{equation*}
    with the notation $w^{N_t}=v_j$. The proof is completed by introducing the piecewise linear curve $\gamma$ as in Definition~\ref{def:pathSemiTV} to recover the reformulation \eqref{eq:repFormDiscSemi}.
\end{proof}

\section{Convergence towards the viscosity solution}\label{sec:ConvVisc}
We conclude our analysis by proving Theorem~\ref{thm:ConvVisc}, namely the convergence of the solution to \eqref{eq:VariDiscFull} towards the viscosity solution of \eqref{eq:Vari} and the quantification of the error. As mentioned at the beginning of this work, we first address the semi-discrete case, using the representation formula \eqref{eq:repFormDiscSemi}, that is equivalent to the semi-discrete scheme \eqref{eq:VariDiscSemiTV} thanks to Proposition~\ref{prop:repFormDiscSemi}. We establish that the limit of \eqref{eq:repFormDiscSemi}, when $\Dt, \Dv \to 0$ is the unique viscosity solution to \eqref{eq:Vari} by directly comparing the continuous and semi-discrete representation formulas, given by Proposition~\ref{prop:repFormDiscSemi} and Theorem~\ref{thm:kinHopfLax}. Finally, we analyze the convergence of the solution to \eqref{eq:VariDiscFull} towards the one of the semi discrete scheme \eqref{eq:VariDiscSemiTV}. The lower part of Figure~\ref{fig:recapConv} summarizes these proofs.

\subsection{Reduction of the discrete action}\label{subsec:reducedActions}

We start by showing that the minimal path $\gamma\in\R\times\R^{\ldb 0,N_t\rdb}$, defined in Proposition~\ref{prop:repFormDiscSemi} admits at most two non-zero velocity jumps. This property is a discrete analog of the following proposition showed in \cite{BouinCalvezGrenierNadin2023}.
\begin{proposition}{(5.1 in \cite{BouinCalvezGrenierNadin2023})}\label{prop:ReducedActionCont}
        Consider the representation formula \eqref{eq:repFormContHL} with the action $\Ab_0^T[\gamma]$ of a minimal path $\gamma\in\R\times\Sigma^{0,T}$, respectively defined in \eqref{eq:contAction} and Definition~\ref{def:pathCont}. In arbitrary dimension, there is at most three states with one intermediate non-trivial one $ (s^1, w^1) $ with non-zero velocity $ w_1 \neq 0 $ and at most one flat state with zero velocity. Hence, the action reduces to the following problem:
    \begin{equation*}
    \lbr 
    \begin{aligned}
        &\text{either } w \equiv w^0 = \frac{x}{T},\,\text{then } \Ab_0^T[\gamma] = T \mathbf{1}_{w^0 \neq 0} \quad (\text{free transport}) \\
        &\text{or: } \Ab_0^T[\gamma] = \frac{|w^2|^2}{2} + \underset{\substack{s^0 w^0 + s^1 w^1 + s^2 w^2 = x \\ s^0, s^1, s^2 \geq 0 \\ s^0 + s^1 + s^2 \leq T}}{\min}
        \lbr \frac{|w^1|^2}{2} + s^0 + s^1 + s^2 \rbr.
    \end{aligned}\right.
    \end{equation*}
\end{proposition}

In the discrete setting, it reads
\begin{lemma}[Reduction of the semi-discrete in time action]\label{lem:ActionReducDisc}
    Suppose that $\varphi_\tin$ satisfies Assumption~\ref{Ass:LinfBounds}. Assume moreover that $\varphi_\tin-v^2/2$ is $L$-Lipschitz in $v$. There exists a minimal path $\gamma\in\R\times\R^{\ldb0,N_t\rdb}$, as defined in Definition~\ref{def:pathSemiT}, to the minimization problem
    \begin{equation*}
        \varphi_j^{N_t}(x) = \underset{\underset{\gamma(T)=x,\,w^{N_t}=v_j}{\gamma=(y,\mathtt{w})\in\R\times\R^{\ldb0,N_t\rdb}}}{\mathrm{inf}}\lbr  \varphi_\tin(y,w^0)+ \At_0^{T}[\gamma]\rbr.
    \end{equation*}
    Moreover, $\gamma$ admits at most one non-trivial velocity state. Denoting $\widetilde{s}^k=\widetilde{N}^k\Dt$ the free transport duration at velocity $\wt^k$, $k=0,1,2$, $\At_0^{T}[\gamma]$ reduces to
    \begin{equation}\label{eq:ReducedDisc}
        \lbr \begin{aligned}
            &\text{either } v_j \equiv w^0 = \frac{x}{T}, \text{then } \At_0^{T}[\gamma] = T \mathds{1}_{w^0 \neq 0} \quad (\text{free transport}) \\[1em]
        &\text{or: }  \At_0^{T}[\gamma] = \frac{|w^2|^2}{2} + \underset{\substack{(\tilde{s}^0, \tilde{s}^1, \tilde{s}^2) \geq 0 \\ \tilde{s}^0 w^0 + \tilde{s}^1 w^1 + \tilde{s}^2 w^2 = x \\ \tilde{s}^0 + \tilde{s}^1 + \tilde{s}^2 \leq T}}{\min} \lbr \frac{|w_1|^2}{2} + \tilde{s}^0 + \tilde{s}^1 + \tilde{s}^2 \rbr.
        \end{aligned}\right.
    \end{equation}
    The initial velocity $w^0$ and the intermediate one $w^1$ are bounded by $v_j$.
\end{lemma}
\begin{proof}
    The proof follows closely the proof of the continuous case recalled in Proposition~\ref{prop:ReducedActionCont}. The only difference is that one has to deal with time intervals of fixed size $\Dt$ instead of variable ones. 

    Let us now consider a path $\gamma=\left(y,\mathtt{w}=(w^0,\dots,w^{N_t})\right)\in\R\times\R^{\ldb 0,N_t \rdb}$. Without changing the cost of $\gamma$ under the discrete action $\At$, one can merge all $0$ velocity sections ($w^k=0$) and put them at the end of the time interval yielding a path $\gamma_0$. Then, instead of making redundant jumps, segments with the same velocity are merged, as free transport is generally more cost-effective. More precisely, this holds when $w^k$ satisfies $(w^k)^2 \geq 2N^k\Dt$ where $N^k$ is the number of time interval spent in free transport. The resulting path is then denoted $\gamma_1$. The case $(w^k)^2 \leq 2N^k\Dt$, which corresponds to relatively small velocities, is considered at the end of the proof, but it does not affect the final results.
    
    Finally, consider a sub-path $\gamma_2 \subset \gamma_1$ consisting of two consecutive segments, $(N^\alpha\Dt, w^\alpha)$ and $(N^\beta\Dt, w^\beta)$, with $w^\alpha \neq w^\beta$. We denote by $s^\alpha$ and $s^\beta$ the starting times of these two segments. The sub-path $\gamma_2$ is simplified to a single segment $\gamma_3$ by averaging the velocities over the total duration $(N_\alpha + N_\beta)\Dt$ into 
    \begin{equation*}
        \wb^\alpha = \frac{N_\alpha\Dt w^\alpha+(N_\beta-1)\Dt w^\beta}{(N_\alpha+N_\beta-1)\Dt}.
    \end{equation*}
    The action of $\gamma_2$ is given by
    \begin{equation}\label{eq:reducStep1}
        \At_{s^\alpha-\Dt}^{s^\beta+N^\beta\Dt}[\gamma_2] =(N^\alpha+N^\beta-1)\Dt + \frac{(w^\alpha)^2}{2} + \frac{(w^\beta)^2}{2} - \Dt.\\
    \end{equation}
    Note that the starting time is $t = s^\alpha - \Dt$ to account for the fact that $\alpha > 0$: while there is no cost associated with the initial velocity, each subsequent jump does incur a cost. By convexity, one has
    \begin{equation}\label{eq:reducConvexity}
            |\wb^\alpha|^2\leq \frac{N_\alpha\Dt}{(N_\alpha+N_\beta-1)\Dt}|w^\alpha|^2 + \frac{N_\beta\Dt}{(N_\alpha+N_\beta-1)\Dt}|w^\beta|^2\leq \left|w^\alpha\right|^2 + \left|w^\beta\right|^2.
    \end{equation}
    Combining \eqref{eq:reducStep1} and \eqref{eq:reducConvexity} one obtains
    \begin{equation*}
        \At_{s^\alpha-\Dt}^{s^\beta+N^\beta\Dt}[\gamma_2] \geq \details{(N^\alpha+N^\beta-1)\Dt + \frac{(\wb^\alpha)^2}{2} - \Dt=}\At_{s^\alpha-\Dt}^{s^\beta+N^\beta\Dt}[\gamma_3]-\Dt.
    \end{equation*}

    An illustration of this procedure is provided in Figure~\ref{fig:ActionReduc} for clarity. The solid blue path represents an initial trajectory whose cost we aim to reduce. The dotted red line depicts a rearranged path where all segments with zero velocity are shifted to the end, without altering the cost, while the remaining velocity segments are ordered in descending magnitude. This reordering decreases the number of jumps by merging the segments with the same velocities. Finally, the dashed green line shows a further reduction in the number of jumps by introducing a single one at an average velocity, up to introducing an additional $-\Dt$ term in the cost of the action.

    Note that this approach ensures the cost is reduced to \textit{at most} two jumps: an intermediate one and a final one. However, depending on the values of $x$ and $v$, it may be more advantageous to perform a single large jump, relatively to $x$ and $v$, and remain stationary rather than accumulating the running cost of free transport. This could for example occur for small target velocities $v$.

    The existence of a minimal path follows from the reduced form~\eqref{eq:ReducedDisc}, which admits a solution. By Assumption~\ref{Ass:LinfBounds} and the continuity of $\varphi_\tin$, the minimization problem involves a continuous coercive function over a closed domain. In particular, this ensures coercivity with respect to $w^0$, while $x$ and $w^2$ act merely as parameters in~\eqref{eq:ReducedDisc}.

    Furthermore, the inclusion of the initial datum $\varphi_\tin$ allows one to determine the initial velocity $w^0$. The boundedness of both $w^0$ and $w^1$ by $v_j$ then follows from the definition of the cost of a velocity jump in~\eqref{eq:discAction}: overshooting the target velocity and subsequently dropping down to it would incur a strictly higher cost.
\end{proof}

\begin{Remark}\label{rem:ActionReduc}
    \begin{enumerate}
        \item Using the same argument as in the discrete case, the initial and intermediate velocities in Proposition~\ref{prop:ReducedActionCont} are also bounded by the final one.
        \item In the proof of Lemma~\ref{lem:ActionReducDisc}, the velocities $\mathtt{w}$ are defined in $\R^{\ldb 0,N_t\rdb}$. This is crucial, as it ensures that the average of two velocities remains within the admissible velocities. We emphasize that this property does not hold when velocities are restricted to a discrete grid. As a result, the reduction of the cost $\At[\gamma]$ for a path $\gamma \in \R\times\Dv\Z^{\ldb0, N_t\rdb}$ cannot be achieved using this method.
        \item Remark that the cost $v^2/2$ of the final jump is always paid except if $x = y + w^0T$ and $w^0=v_j$.
        \item The minimal path is not necessarily unique. Indeed, one can shift the location of the zero-velocity segment and any possible non-trivial portion without affecting the cost.
    \end{enumerate}
\end{Remark}

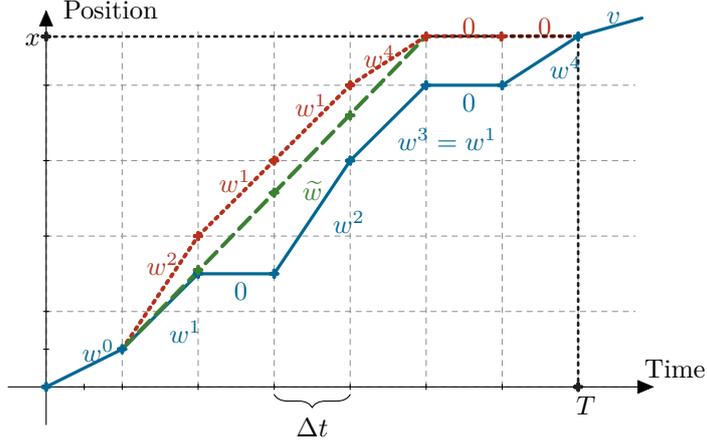
\begin{figure}
    \begin{tikzpicture}[line cap=round,line join=round,>=triangle 45,x=1.0cm,y=1.0cm, scale=0.5]
        \draw [color=Gray,dash pattern=on 2pt off 2pt, xstep=2.0cm,ystep=2.0cm] (0,0) grid (15.5,9.5);
        \draw[->,color=black] (-1,0) -- (16,0);
        \foreach \x in {,2,4,6,8,10,12}
        \draw[shift={(\x,0)},color=black] (0pt,2pt) -- (0pt,-2pt);
        \draw[color=black] (15.5,0) node [anchor=south west] {Time};
        \draw[->,color=black] (0,-1) -- (0,10);
        \foreach \y in {,2,4,6,8}
        \draw[shift={(0,\y)},color=black] (2pt,0pt) -- (-2pt,0pt);
        \draw[color=black] (0.2,10) node [anchor=west] { Position};
        \draw [line width=1.2pt,color=MidnightBlue] (0,0)-- (2,1);
        \draw [line width=1.2pt,color=MidnightBlue] (2,1)-- (4,3);
        \draw [line width=1.2pt,color=MidnightBlue] (4,3)-- (6,3);
        \draw [line width=1.2pt,color=MidnightBlue] (6,3)-- (8,6);
        \draw [line width=1.2pt,color=MidnightBlue] (8,6)-- (10,8);
        \draw [line width=1.2pt,color=MidnightBlue] (10,8)-- (12,8);
        \draw [line width=1.2pt,color=MidnightBlue] (12,8)-- (14,9.3);
        \draw [line width=1.2pt,color=MidnightBlue] (14,9.3)-- (15.68,9.78);
        \draw (-0.8,9.6) node[anchor=north west] {$ x $};
        \draw (13.7,0) node[anchor=north west] {$ T $};
        \draw (0.7,1.5) node[anchor=north west] {\textcolor{MidnightBlue}{$ w^0 $}};
        \draw (4.3,6.) node[anchor=north west] {\textcolor{BrickRed}{$ w^1 $}};
        \draw (4.7,3.) node[anchor=north west] {\textcolor{MidnightBlue}{$0$}};
        \draw (7.3,4.9) node[anchor=north west] {\textcolor{MidnightBlue}{$ w^2 $}};
        \draw (9,7.1) node[anchor=north west] {\textcolor{MidnightBlue}{$ w^3=w^1 $}};
        \draw (10.7,8) node[anchor=north west] {\textcolor{MidnightBlue}{$0$}};
        \draw (13,9.) node[anchor=north west] {\textcolor{MidnightBlue}{$w^4$}};
        \draw (14.5,10.2) node[anchor=north west] {\textcolor{MidnightBlue}{$ v $}};
        \draw (10.7,10) node[anchor=north west] {\textcolor{BrickRed}{$0$}};
        \draw (12.7,10) node[anchor=north west] {\textcolor{BrickRed}{$0$}};
        \draw (6.5,5.7) node[anchor=north west] {\textcolor{OliveGreen}{$ \widetilde{w} $}};
        \draw (3,2) node[anchor=north west] {\textcolor{MidnightBlue}{$w^1$}};
        \draw (2.4,3.8) node[anchor=north west] {\textcolor{BrickRed}{$w^2$}};
        \draw (6.3,8) node[anchor=north west] {\textcolor{BrickRed}{$ w^1 $}};
        \draw (8.1,9.3) node[anchor=north west] {\textcolor{BrickRed}{$ w^4 $}};
        \draw [line width=1.5pt,dash pattern=on 1pt off 2pt,color=BrickRed] (14,9.3)-- (12,9.3);
        \draw [line width=1.5pt,dash pattern=on 1pt off 2pt,color=BrickRed] (12,9.3)-- (10,9.3);
        \draw [line width=1.5pt,dash pattern=on 1pt off 2pt,color=BrickRed] (10,9.3)-- (8,8);
        \draw [line width=1.5pt,dash pattern=on 1pt off 2pt,color=BrickRed] (8,8)-- (6,6);
        \draw [line width=1.5pt,dash pattern=on 1pt off 2pt,color=BrickRed] (6,6)-- (4,4);
        \draw [line width=1.5pt,dash pattern=on 1pt off 2pt,color=BrickRed] (2,1)-- (4,4);
        \draw [line width=1.5pt,dash pattern=on 6pt off 3pt,color=OliveGreen] (2,1)-- (10,9.3);
        \draw [dotted, line width=1pt, color=Black] (0,9.29)-- (14,9.3);
        \draw [dotted, line width=1pt, color=Black] (14,9.29)-- (14,0);
        \draw [decorate, decoration={brace, amplitude=5pt, mirror}] (6,-0.2) -- (8,-0.2) node[midway,below=5pt] {$\Dt$};
        \begin{scriptsize}
        \draw [line width=1.5pt, color=MidnightBlue] (0,0)-- ++(-2.5pt,0 pt) -- ++(5.0pt,0 pt) ++(-2.5pt,-2.5pt) -- ++(0 pt,5.0pt);
        \draw [line width=1.5pt, color=MidnightBlue] (2,1)-- ++(-2.5pt,0 pt) -- ++(5.0pt,0 pt) ++(-2.5pt,-2.5pt) -- ++(0 pt,5.0pt);
        \draw [line width=1.5pt, color=MidnightBlue] (4,3)-- ++(-2.5pt,0 pt) -- ++(5.0pt,0 pt) ++(-2.5pt,-2.5pt) -- ++(0 pt,5.0pt);
        \draw [line width=1.5pt, color=MidnightBlue] (6,3)-- ++(-2.5pt,0 pt) -- ++(5.0pt,0 pt) ++(-2.5pt,-2.5pt) -- ++(0 pt,5.0pt);
        \draw [line width=1.5pt, color=MidnightBlue] (8,6)-- ++(-2.5pt,0 pt) -- ++(5.0pt,0 pt) ++(-2.5pt,-2.5pt) -- ++(0 pt,5.0pt);
        \draw [line width=1.5pt, color=MidnightBlue] (10,8)-- ++(-2.5pt,0 pt) -- ++(5.0pt,0 pt) ++(-2.5pt,-2.5pt) -- ++(0 pt,5.0pt);
        \draw [line width=1.5pt, color=MidnightBlue] (12,8)-- ++(-2.5pt,0 pt) -- ++(5.0pt,0 pt) ++(-2.5pt,-2.5pt) -- ++(0 pt,5.0pt);
        \draw [line width=1.5pt, color=MidnightBlue] (14,9.29)-- ++(-2.5pt,0 pt) -- ++(5.0pt,0 pt) ++(-2.5pt,-2.5pt) -- ++(0 pt,5.0pt);
        \draw [line width=1.5pt, color=Black] (0,9.29)-- ++(-2.5pt,0 pt) -- ++(5.0pt,0 pt) ++(-2.5pt,-2.5pt) -- ++(0 pt,5.0pt);
        \draw [line width=1.5pt, color=Black] (14,0)-- ++(-2.5pt,0 pt) -- ++(5.0pt,0 pt) ++(-2.5pt,-2.5pt) -- ++(0 pt,5.0pt);
        \draw [line width=1.5pt, color=BrickRed] (11.98,9.29)-- ++(-2.5pt,0 pt) -- ++(5.0pt,0 pt) ++(-2.5pt,-2.5pt) -- ++(0 pt,5.0pt);
        \draw [line width=1.5pt, color=BrickRed] (10,9.29)-- ++(-2.5pt,0 pt) -- ++(5.0pt,0 pt) ++(-2.5pt,-2.5pt) -- ++(0 pt,5.0pt);
        \draw [line width=1.5pt, color=BrickRed] (8,8)-- ++(-2.5pt,0 pt) -- ++(5.0pt,0 pt) ++(-2.5pt,-2.5pt) -- ++(0 pt,5.0pt);
        \draw [line width=1.5pt, color=BrickRed] (6,6)-- ++(-2.5pt,0 pt) -- ++(5.0pt,0 pt) ++(-2.5pt,-2.5pt) -- ++(0 pt,5.0pt);
        \draw [line width=1.5pt, color=BrickRed] (4,4)-- ++(-2.5pt,0 pt) -- ++(5.0pt,0 pt) ++(-2.5pt,-2.5pt) -- ++(0 pt,5.0pt);

        \draw [line width=1.5pt, color=OliveGreen] (4,3.1)-- ++(-2.5pt,0 pt) -- ++(5.0pt,0 pt) ++(-2.5pt,-2.5pt) -- ++(0 pt,5.0pt);
        \draw [line width=1.5pt, color=OliveGreen] (6,5.15)-- ++(-2.5pt,0 pt) -- ++(5.0pt,0 pt) ++(-2.5pt,-2.5pt) -- ++(0 pt,5.0pt);
        \draw [line width=1.5pt, color=OliveGreen] (8,7.2)-- ++(-2.5pt,0 pt) -- ++(5.0pt,0 pt) ++(-2.5pt,-2.5pt) -- ++(0 pt,5.0pt);
        \end{scriptsize}
    \end{tikzpicture}
    \caption{Diagram of the reduction of a time discrete path with 4 non-trivial intermediate velocities. Original path (solid blue line), reorganized path (dotted red line), reduced path (dashed green line).}\label{fig:ActionReduc}
\end{figure}

\subsection{Semi-discrete convergence}
As mentionned at the begining of the section, we first address the semi-discrete case and show that, as $\Dt, \Dv \to 0$, the solution given by the representation formula \eqref{eq:repFormDiscSemi} converges towards the viscosity solution of \eqref{eq:Vari}. 

\begin{lemma}\label{lem:ConvViscDtDv}
    Let $x\in\R$ and $(\varphi_{j}^n(x))_{j\in\Z,\,n\in\ldb1,N_t\rdb}$ be given by \eqref{eq:repFormDiscSemi}. Suppose that $\varphi_\tin$ satisfies Assumption~\ref{Ass:LinfBounds} and assume moreover that $\varphi_\tin-v^2/2$ is $L$-Lipschitz in $v$. Let $j\in\Z$ be fixed. Then, in the limit $\Dt,\,\Dv\to0$, $(\varphi_{j}^n(x))_{n\in\ldb1,N_t\rdb}$ converges towards the viscosity solution of \eqref{eq:Vari}. Moreover, there exists $C>0$, depending on $j$, such that
    \begin{equation}\label{eq:errorConvViscDtDv}
        \forall n\in\ldb1,N_t\rdb,\quad \left|\varphi(t^n,x,v_j) - \varphi_{j}^n(x) \right|\leq C(\Dt+\Dv).
    \end{equation}
\end{lemma}
\begin{proof}
    The proof of this result is a consequence of the identification of the minimal paths introduced in Section~\ref{subsec:reducedActions}. We start by recalling the following inclusions:
    \begin{equation}\label{eq:infInclusions}
        \R\times\Sigma^{0,T}_{\Dt,\Dv} \subset \R\times\Sigma^{0,T}_{\Dt} \subset\R\times\Sigma^{0,T}.
    \end{equation}
    Inequality \eqref{eq:errorConvViscDtDv} is then proven in two steps.

    \myParagraph{Step 1}
    We start from \eqref{eq:repFormDiscSemi} and use the first inclusion in \eqref{eq:infInclusions} to obtain
    \begin{equation*}
        \varphi_j^{N_t}(x) \details{= \underset{\underset{\gamma(T)=x,\,w^{N_t}=v_j}{\gamma=(y,\textbf{w})\in\R\times\Dv\Z^{\ldb1,N_t\rdb}}}{\mathrm{inf}}\lbr  \varphi_\tin(y,w^0)+\At_0^T[\gamma]\rbr}
        \geq \underset{\underset{\gamma(T)=x,\,w^{N_t}=v_j}{\gamma=(y,\textbf{w})\in\R\times\Sigma^{0,T}_{\Dt}}}{\mathrm{inf}}\lbr  \varphi_\tin(y,w^0)+\At_0^T[\gamma]\rbr,
    \end{equation*}
    which is a consequence of minimizing over a larger set. Using Lemma~\ref{lem:ActionReducDisc}, this infimum is in particular reached at a minimal path $\gamt\in\R\times\Sigma^{0,T}_{\Dt}$ with at most two jumps. We denote by $\wt^0$ the initial velocity and $\wt$ the possible non-trivial velocity. It yields:
    \begin{equation*}
        \varphi_j^{N_t}(x) \geq \varphi_\tin(y,\wt^0) + \At_0^{T}[\gamt],
    \end{equation*}
    and we introduce $\Ab$ defined in Definition~\ref{def:contAction} to obtain
    \begin{align*}
        \varphi_j^{N_t}(x) \details{&\geq \varphi_\tin(y,w^0)+ \Ab_0^{T}[\gamt]  + \left(\At_0^{T}[\gamt] - \Ab_0^{T}[\gamt]\right)\notag\\}
        &\geq \underset{\underset{\gamma(T)=x,\,\dot{\gamma}(T)=v_j}{\gamma=(y,\textbf{w})\in\R\times\Sigma^{0,T}}}{\mathrm{inf}}\lbr  \varphi_\tin(y,w^0)+\Ab_0^{T}[\gamma]\rbr  + \left(\At_0^{T}[\gamt] - \Ab_0^{T}[\gamt]\right)\notag\\
        &= \varphi(T,x,v_j) + \At_0^{T}[\gamt] - \Ab_0^{T}[\gamt],
    \end{align*}
    One can then apply Lemma~\ref{lem:discrepancyDt} to the difference $\At_0^{T}[\gamt] - \Ab_0^{T}[\gamt]$. Since $\gamt$ admits at most two velocity jumps, one obtains
    \begin{equation}\label{eq:InequalitiesConvergence1}
        \varphi_j^{N_t}(x)\geq\varphi(T,x,v_j) - 2\Dt.
    \end{equation}

    \myParagraph{Step 2}
    We start from \eqref{eq:repFormContHL} in which we choose a minimal path $\gamb\in\R\times\Sigma^{0,T}$ with $\gamb(0)=\overbar{y}$. Thanks to Proposition~\ref{prop:ReducedActionCont}, it has at most two jumps. In the following, we assume that $\gamb$ admits two jumps but the same analysis can be done for one or no jumps. We denote by $\wb^0$ and $\wb^1$ the associated velocities. One has:
    \begin{equation}\label{eq:convStep2start}
        \varphi(T,x,v_j) \details{= \underset{\underset{\gamma(T)=x,\,\dot{\gamma}(T)=v_j}{(y,\textbf{w})\in\R\times\Sigma^{0,T}}}{\mathrm{inf}}\lbr  \varphi_\tin(y,w^0)+\Ab_0^{T}[\gamma]\rbr} = \varphi_\tin(\overbar{y},\wb^0) + \Ab_0^{T}[\gamb].
    \end{equation}
    In what follows, we define $\gamt\in\R\times\Sigma^{0,T}_{\Dt,\Dv}$ such that it has a similar cost compared to $\gamb\in\R\times\Sigma^{0,T}$ using the action $\Ab$. It is chosen with the same number of velocity jumps as $\gamb$, that is at most 2 thanks to Proposition~\ref{prop:ReducedActionCont}. The velocities in $\gamt$ are the velocities of $\gamb$ shifted towards the nearest neighbor on the grid $\Dv\Z$. Similarly, the jump times in $\gamt$ are the jump times in $\gamb$ shifted to their nearest neighbor on the grid $\lbr n\Dt,\,0\leq n\leq N_t\rbr$. We denote $\wt^0=\dot{\gamt}(0)\in\Dv\Z$ and $\wt^1\in\Dv\Z$ the initial and intermediate velocity of this path. Its starting point is denoted $\widetilde{y}=\gamt(0)\in\R$. Without loss of generality, let us assume that the trivial state with zero velocity is located at the end of the time interval. Then, by construction of $\gamt$, one has the following estimates:
    \begin{equation}\label{eq:decalDt}
        |\widetilde{s}^1-\overbar{s}^1|\leq\Dt,\quad|\widetilde{s}^0-\overbar{s}^0|\leq\Dt,\quad |\wt^1-\wb^1|\leq\Dv,\text{ and}\,|\wt^0-\wb^0|\leq\Dv.
    \end{equation}
    The following observation can then be made: the definition of $\gamt$ introduces an error on $y^0$, $\wb^0$, $\wb^1$ and on the jump times $\overbar{s}^0$ and $\overbar{s}^1$. Figure~\ref{fig:choiceIntermediatePath} illustrates this remark, and the notations. To estimate these errors, one follows the characteristic curves backward at velocities $\wt^1$ and $\wb^1$ for the durations $\widetilde{s}_1=\widetilde{s}^1-\widetilde{s}^0$ and $\overbar{s}_1=\overbar{s}^1-\overbar{s}^0$ respectively. We denote $\widetilde{x}_1 = x - \wt^1 \widetilde{s}_1\quad\text{and}\quad\overbar{x}_1 = x - \wb^1 \overbar{s}_1$ the feet of the intermediate states of $\gamt$ and $\gamb$, so that
    \begin{equation*}
        |\overbar{x}_1-\widetilde{x}_1|\details{=|\overbar{s}_1\wb - \widetilde{s}_1\wt|
        \leq|\wb(\overbar{s}_1-\widetilde{s}_1)| + |\widetilde{s}_1(\wb-\wt)|
        \leq |v_j||\overbar{s}-\overbar{s}^0 -(\widetilde{s}-\widetilde{s}^0)| + T\Dv}
        \leq 2v_j\Dt + T\Dv,
    \end{equation*}
    thanks to \eqref{eq:decalDt} and $\wb^1 \leq v_j$ (Remark~\ref{rem:ActionReduc}). As $\gamt$ and $\gamb$ have at most two jumps, the same reasoning can be applied to the sections with velocities $\wt^0$ and $\wb^0$, so that
    \begin{equation*}
        |\widetilde{y}-\overbar{y}|=|\gamt(0)-\gamb(0)|\leq C(\Dt+\Dv),
    \end{equation*}
    where the constant $C$ depends on $T$ and $j\in\Z$.
    
    \begin{figure}
        \begin{tikzpicture}[line cap=round,line join=round,>=triangle 45,x=1.0cm,y=1.0cm]
            \draw [color=Gray,dash pattern=on 1pt off 1pt, xstep=1.0cm,ystep=1.0cm] (-0.0,-0.0) grid (7.86,4.36);
            \draw[->,color=black] (-0.5,0) -- (7.86,0);
            \foreach \x in {,1,2,3,4,5,6,7}
            \draw[shift={(\x,0)},color=black] (0pt,2pt) -- (0pt,-2pt);
            \draw[color=black] (7.2,0.04) node [anchor=south west] { Time};
            \draw[->,color=black] (0,-0.5) -- (0,4.36);
            \foreach \y in {,1,2,3,4}
            \draw[shift={(0,\y)},color=black] (2pt,0pt) -- (-2pt,0pt);
            \draw[color=black] (0.05,4.2) node [anchor=west] { Position};
            
            \draw [line width=1.3pt,color=BrickRed] (0,0)-- (2.3,2.2);
            \draw [line width=1.3pt,color=BrickRed] (2.3,2.2)--(3.3,3);
            \draw [line width=1.3pt,color=BrickRed] (3.3,3)--(6,3);
            \draw [line width=1.3pt,color=BrickRed] (6,3)-- (7,3.7);

            \draw [line width=0.9pt,dotted,color=BrickRed] (0,2.2)-- (2.3,2.2);
            \draw [line width=0.9pt,dotted,color=BrickRed] (3.3,3)-- (3.3,0);
            \draw [line width=0.9pt,dotted,color=BrickRed] (2.3,2.2)-- (2.3,0);
            \draw [line width=0.9pt,dotted,color=BrickRed] (6,3)-- (6,0);
            
            \draw (-0.5,0.1) node[anchor=north west] {\textcolor{BrickRed}{$\overbar{y}$}};
            \draw (-0.6,2.4) node[anchor=north west] {\textcolor{BrickRed}{$\overbar{x}_1 $}};
            \draw (1.1,1.1) node[anchor=north west] {\textcolor{BrickRed}{$\wb^0 $}};
            \draw (2.5,2.5) node[anchor=north west] {\textcolor{BrickRed}{$\wb^1$}};
            \draw (6.4,3.34) node[anchor=north west] {\textcolor{BrickRed}{$v_j $}};

            \draw (2.1,-0.03) node[anchor=north west] {\textcolor{BrickRed}{$\overbar{s}^0$}};
            \draw (3.1,-0.05) node[anchor=north west] {\textcolor{BrickRed}{$\overbar{s}^1$}};

            \draw [line width=1.3pt,dash pattern=on 4pt off 2.5pt,color=MidnightBlue] (0,0.3)-- (2,2.5);
            \draw [line width=1.3pt,dash pattern=on 4pt off 2.5pt,color=MidnightBlue] (2,2.5)-- (3,3);
            \draw [line width=1.3pt,dash pattern=on 4pt off 2.5pt,color=MidnightBlue] (3,3)-- (6,3);

            \draw (-0.6,2.8) node[anchor=north west] {\textcolor{MidnightBlue}{$\widetilde{x}_1 $}};
            \draw (-0.5,0.8) node[anchor=north west] {\textcolor{MidnightBlue}{$\widetilde{y}$}};
            \draw (0.6,2.0) node[anchor=north west] {\textcolor{MidnightBlue}{$\wt^0$}};
            \draw (2.,3.1) node[anchor=north west] {\textcolor{MidnightBlue}{$\wt^1$}};
            
            \draw [line width=0.9pt,dotted,color=MidnightBlue] (0,2.5)-- (2,2.5);
            \draw [line width=0.9pt,dotted,color=MidnightBlue] (2,0)-- (2,2.5);
            \draw [line width=0.9pt,dotted,color=MidnightBlue] (3,0)-- (3,3);

            \draw (2.8,-0.03) node[anchor=north west] {\textcolor{MidnightBlue}{$ \widetilde{s}^1 $}};
            \draw (1.7,-0.03) node[anchor=north west] {\textcolor{MidnightBlue}{$ \widetilde{s}^0 $}};

            \draw [line width=0.9pt,dotted,color=BrickRed] (0,3)-- (6,3);
            \draw (-0.5,3.2) node[anchor=north west] {$ x $};
            \draw (5.8,-0.03) node[anchor=north west] {$ T $};
            \draw [decorate, decoration={brace, amplitude=5pt, mirror}] (4,-0.2) -- (5,-0.2) node[midway,below=5pt] {$\Dt$};
            \begin{scriptsize}
                \draw [color=BrickRed, line width=2.0pt] (-2.5pt,0) -- (2.5pt,0) (0,-2.5pt) -- (0,2.5pt);
                \fill [color=BrickRed] (6,3) circle (1.5pt);
                \fill [color=BrickRed] (3.3,3) circle (1.5pt);
                \fill [color=BrickRed] (2.3,2.2) circle (1.5pt); 
                
                \draw [line width=2.0pt, color=BrickRed] (0,0)-- ++(-2.5pt,0 pt) -- ++(5.0pt,0 pt) ++(-2.5pt,-2.5pt) -- ++(0 pt,5.0pt);
                \draw [line width=2.0pt,color=BrickRed] (0,2.2)-- ++(-2.5pt,0 pt) -- ++(5.0pt,0 pt) ++(-2.5pt,-2.5pt) -- ++(0 pt,5.0pt);
                \draw [line width=2.0pt,color=BrickRed] (0,3)-- ++(-2.5pt,0 pt) -- ++(5.0pt,0 pt) ++(-2.5pt,-2.5pt) -- ++(0 pt,5.0pt);
                \draw [line width=2.0pt, color=BrickRed] (2.3,0)-- ++(-2.5pt,0 pt) -- ++(5.0pt,0 pt) ++(-2.5pt,-2.5pt) -- ++(0 pt,5.0pt);
                \draw [line width=2.0pt,color=BrickRed] (3.3,0)-- ++(-2.5pt,0 pt) -- ++(5.0pt,0 pt) ++(-2.5pt,-2.5pt) -- ++(0 pt,5.0pt);
                \draw [line width=2.0pt,color=BrickRed] (6,0)-- ++(-2.5pt,0 pt) -- ++(5.0pt,0 pt) ++(-2.5pt,-2.5pt) -- ++(0 pt,5.0pt);
                \fill [color=MidnightBlue] (1,1.4) circle (1.5pt);
                \fill [color=MidnightBlue] (2,2.5) circle (1.5pt);
                \fill [color=MidnightBlue] (3,3) circle (1.5pt);
                \fill [color=MidnightBlue] (4,3) circle (1.5pt);
                \fill [color=MidnightBlue] (5,3) circle (1.5pt);

                \draw [line width=2.0pt, color=MidnightBlue] (-2.5pt,0.3) -- (2.5pt,0.3) (0,-2.5pt+8.5) -- (0,2.5pt+8.5);
                \draw [line width=2.0pt, color=MidnightBlue] (2,0)-- ++(-2.5pt,0 pt) -- ++(5.0pt,0 pt) ++(-2.5pt,-2.5pt) -- ++(0 pt,5.0pt);
                \draw [line width=2.0pt, color=MidnightBlue] (3,0)-- ++(-2.5pt,0 pt) -- ++(5.0pt,0 pt) ++(-2.5pt,-2.5pt) -- ++(0 pt,5.0pt);

                \draw [line width=2.0pt, color=MidnightBlue] (0,2.5)-- ++(-2.5pt,0 pt) -- ++(5.0pt,0 pt) ++(-2.5pt,-2.5pt) -- ++(0 pt,5.0pt);
            \end{scriptsize}
        \end{tikzpicture}
        \caption{Illustration a minimal path $\gamb$ under the action $\bar{A}$ defined in \eqref{eq:contAction} (solid red line) and its approximation by a path $\gamt$ defined on the time and velocity grids (dashed blue line).}\label{fig:choiceIntermediatePath}
    \end{figure}
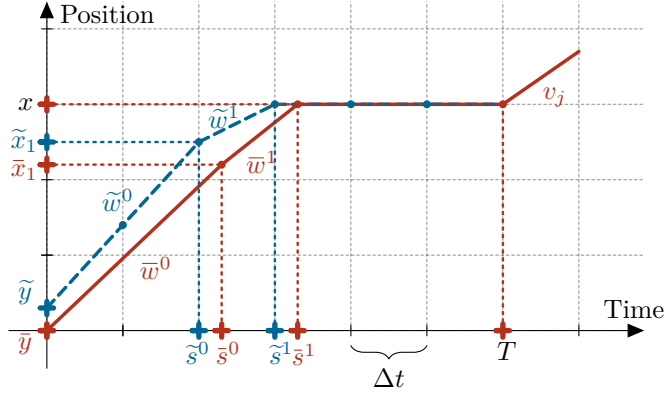

    Introducing $\gamt$ in \eqref{eq:convStep2start} yields
    \begin{align*}
        \varphi(T,x,v_j) = &\,\varphi_\tin(\widetilde{y},\wt^0) + \At_0^{T}[\gamt]\\
        & + \left(\varphi_\tin(\overbar{y},\wb^0) - \varphi_\tin(\widetilde{y}\,\wt^0)\right) + \left(\Ab_0^{T}[\gamb] - \Ab_0^{T}[\gamt]\right) + \left(\Ab_0^{T}[\gamt] - \At_0^{T}[\gamt]\right).
    \end{align*}
    Using the inclusion \eqref{eq:infInclusions}, one obtains
    \begin{align}\label{eq:InequalitiesConvergence2tmp}
        \varphi(T,x,v_j) &\geq \underset{\underset{\gamma(T)=x,\,w^{N_t}=v}{\gamma=(y,\textbf{w})\in\R\times\Sigma^{0,T}_{\Dt,\Dv}}}{\mathrm{inf}}\lbr  \varphi_\tin(y,w^0)+\At_0^T[\gamma]\rbr\notag\\
        &\qquad + \left(\varphi_\tin(y,\wb^0) - \varphi_\tin(\widetilde{y},\wt^0)\right) + \left(\Ab_0^{T}[\gamb] - \Ab_0^{T}[\gamt]\right) + \left(\Ab_0^{T}[\gamt] - \At_0^{T}[\gamt]\right)\notag\\
        &=\varphi_j^{N_t}(x) + \left(\varphi_\tin(\overbar{y},\wb^0) - \varphi_\tin(\widetilde{y},\wt^0)\right) + \left(\Ab_0^{T}[\gamb] - \Ab_0^{T}[\gamt]\right) + \left(\Ab_0^{T}[\gamt] - \At_0^{T}[\gamt]\right).
    \end{align}
    We now estimate the three differences in \eqref{eq:InequalitiesConvergence2tmp}. Firstly, the $L$-Lipschitz continuity of $\varphi_\tin-v^2/2$ and inequalities \eqref{eq:decalDt} yield
    \begin{equation*}
        \begin{aligned}
            |\varphi_\tin(\overbar{y},\wb^0) - \varphi_\tin(\widetilde{y},\wt^0)| \details{&\leq \left|\varphi_\tin(\overbar{y},\wb^0) - \frac{(\wb^0)^2}{2} - \left(\varphi_\tin(\widetilde{y},\wt^0) - \frac{(\wt^0)^2}{2}\right)\right| + \left| \frac{(\wb^0)^2}{2} - \frac{(\wt^0)^2}{2}\right|\\
            &\leq L ( |\overbar{y} - \tilde{y} | + | \wt^0 - \wt^0| ) + \frac{1}{2}| \wt^0 - \wt^0| | \wb^0 + \wt^0| \\}
            &\leq CL (\Dt + \Dv ) + \frac{1}{2}\Dv | \wb^0 + \wt^0|.
        \end{aligned}
    \end{equation*} 
    As both $\wb^0$ and $\wt^0$ are less than the target velocity $v_j$ (Lemma~\ref{lem:ActionReducDisc} and Remark~\ref{rem:ActionReduc}), there exists a constant $C$, that depends on $j$, such that
    \begin{equation}\label{eq:InequalitiesConvergence2LipInit}
        \left|\varphi_\tin(\overbar{y},\wb^0) - \varphi_\tin(\widetilde{y},\wt^0)\right|\leq C(\Dt+\Dv).
    \end{equation}
    Secondly, we deal with the difference of actions $\Ab_0^{T}[\gamb] - \Ab_0^{T}[\gamt]$. By definition, the paths $\gamb$ and $\gamt$ admit two jumps with the final one being shared. The difference reads, 
    \begin{equation*}
        \left|\Ab_0^{T}[\gamb] - \Ab_0^{T}[\gamt]\right| \leq \left|\overbar{s}^0+\overbar{s}^1- \widetilde{s}^0-\widetilde{s}^1\right| + \left|\frac{\left(\wb^1\right)^2}{2}- \frac{\left(\wt^1\right)^2}{2}\right| + \left|\mathds{1}_{\lbr v_j\neq \frac{x}{T},\,\wb^0\neq v_j\rbr }-\mathds{1}_{\lbr v_j\neq \frac{x}{T},\,\tilde{w}^0\neq v_j\rbr }\right|\frac{v_j^2}{2},
    \end{equation*}
    where the indicator functions take the value 1 if the condition is satisfied and 0 otherwise. Using \eqref{eq:decalDt}, the differences in time and velocity are of order $\Dt$ and $\Dv$, respectively. Regarding the difference of indicator functions encoding the last jump, we separate two cases. First, if $v_j\neq\frac{x}{T}$, the difference is zero because a final jump must be paid to reach the target velocity both in $\gamt$ and $\gamb$, and they have the same cost. Secondly, if $v_j=\frac{x}{T}$, the cases $\wb^0\neq v_j$ and $\wt^0\neq v_j$ again both require final a jump. Lastly, assume that $\wb^0=v_j\in\Dv\Z$, then $\wt^0=v_j$. Consequently, the indicator functions both equal zero at the same time and the last jump never introduces an error. It yields
    \begin{equation}\label{eq:InequalitiesConvergence2Cost}
        \left|\Ab_0^{T}[\gamb] - \At_0^{T}[\gamt]\right|\leq 2\Dt + (\wb^1+\wt^1)\frac{\Dv}{2}\leq 2\Dt + 2v_j\Dv.
    \end{equation}
    Thirdly, the difference $\Ab_0^{T}[\gamt] - \At_0^{T}[\gamt]$ is bounded by $2\Dt$ thanks to Lemma~\ref{lem:discrepancyDt}. Combining this bound with \eqref{eq:InequalitiesConvergence2Cost} and \eqref{eq:InequalitiesConvergence2LipInit}, there exists a constant $C$, depending on the final time $T$ and on the target velocity $v_j$, such that
    \begin{equation}\label{eq:InequalitiesConvergence2}
            \varphi(T,x,v_j)\geq\varphi_j^{N_t}(x) - C(\Dt+\Dv).
    \end{equation}

    Finally, by combining \eqref{eq:InequalitiesConvergence1} and \eqref{eq:InequalitiesConvergence2}, and by equivalence of \eqref{eq:repFormDiscSemi} and \eqref{eq:VariDiscSemiTV}, we conclude the proof of Lemma~\ref{lem:ConvViscDtDv}.
\end{proof}

\subsection{Fully discrete convergence}\label{subsec:ConvFull}
In this section, we prove the convergence of the solution to \eqref{eq:VariDiscFull} towards the one of \eqref{eq:VariDiscSemiTV} when $\Dx\to0$ with fixed $\Dt$ and $\Dv$.
\begin{lemma}\label{lem:ConvViscDx}
    Let $(\mu_i^n, \varphi_{ij}^n)_{i\in\Z,\,j\in\Z,\,n\in\ldb1,N_t\rdb}$ be given by \eqref{eq:VariDiscFull}. Suppose that $\varphi_\tin$ satisfies Assumption~\ref{Ass:LinfBounds} and that $\varphi_\tin-v^2/2$ is $L$-Lipschitz in $v$. Then, in the limit $\Dx\to0$, the solution to \eqref{eq:VariDiscFull} converges towards the solution to the semi-discrete scheme \eqref{eq:VariDiscSemiTV}: for all $(i,j)\in\Z^2$, and all $n\in\ldb 1, N_t\rdb$,
    \begin{align}
        \left|\mu^n(x_i) - \mu_{i}^n \right|\leq 2LT\frac{\Dx}{\Dt}, \label{eq:errorConvViscDxMu}\\
        \left|\varphi_j^n(x_i) - \varphi_{ij}^n \right|\leq 2LT\frac{\Dx}{\Dt}. \label{eq:errorConvViscDxPhi}
    \end{align}
\end{lemma}
Before beginning the proof, it is worth mentioning that the transport in $x$ is exactly resolved, meaning that no interpolations are needed, if $\Dt\Dv=\Dx$. In this case, the results of Lemma~\ref{lem:ConvViscDx} are true for \eqref{eq:VariDiscFull}. However, condition \eqref{eq:stabJumpCond} then implies a huge computational cost, as $\Dx\leq\Dv^3/2$.

\begin{proof}
    Proceeding by induction, we start by estimating the propagation of interpolation errors over a single time step. Assume that for some $n\geq0$, and for all $(i,j)\in\Z^2$,
    \begin{align}
        e_{\mu,i}^n&\coloneq\mu^n(x_i) - \mu_{i}^n,\quad \left|e_{\mu,i}^n\right|\leq 2Ln\Dx, \label{eq:errorConvViscDxMuInduction}\\
        e_{\varphi,ij}^n&\coloneq\varphi_j^n(x_i) - \varphi_{ij}^n,\quad \left|e_{\varphi,ij}^n\right|\leq 2Ln\Dx, \label{eq:errorConvViscDxPhiInduction}
    \end{align}
    where $(\mu_{i}^n,\varphi_{ij}^n)$ is obtained from \eqref{eq:VariDiscFull} and $(\mu^n(x_i),\varphi_{j}^n(x_i))$ from \eqref{eq:VariDiscSemiTV}. Remark that the interpolation \eqref{eq:interpolation} applied to a function $h=h(x)$ that is $L$-Lipschitz yields an error of size $2L\Dx$. 
    \details{After some straightforward computations one obtains the following error estimate
    \begin{equation*}
        |h(x)-\overbar{h}(x)| \leq 2L\Dx.
    \end{equation*}}
    Applying this result to $\overbar{\varphi}_{i_j,j}^{n}$, given by \eqref{eq:interpolation}, and Corollary~\ref{cor:LipBoundLim} on the propagation of Lipschitz bounds, one obtains:
    \begin{equation}\label{eq:errorInterpPhi}
        | \varphi_j^{n}(x_i-\Dt v_j)-\overbar{\varphi}_{i_j,j}^{n}|\leq 2L\Dx,
    \end{equation}
    where $L$ is the Lipschitz constant of $\varphi_\tin$. In particular, introducing $\varphi_j^n(x_{i_j})$ in the interpolation \eqref{eq:interpolation}, one has
    \begin{equation}\label{eq:errorInterpPhiPropag}
        \overbar{\varphi}_{i_j,j}^{n} = \alpha_j \varphi_j^n(x_{i_j}) + (1-\alpha_j)\varphi_j^n(x_{i_j-1}) + \alpha_j e_{\varphi,i_j,j}^n+ (1-\alpha_j)e_{\varphi,i_j-1,j}^n
    \end{equation}

    Let us now show the propagation of \eqref{eq:errorConvViscDxMuInduction}. Introducing $\varphi_{j}^n(x_i-\Dt v_j)$ and \eqref{eq:errorInterpPhiPropag} in the update relation on $\mu_i^{n+1}$ in \eqref{eq:VariDiscFull} one obtains
    \begin{align*}
        \mu_i^{n+1} &= \underset{j\in\Z}{\min} \Big\{ \left(\alpha_j\varphi_{j}^{n}(x_{i_j}) + (1-\alpha_j)\varphi_{j}^{n}(x_{i_j-1})-\varphi_j^{n}(x_i-\Dt v_j)\right)\\
        &\quad +\alpha_j e_{\varphi,i_j,j}^n+ (1-\alpha_j)e_{\varphi,i_j-1,j}^n +\varphi_j^{n}(x_i-\Dt v_j)\Big\}.
    \end{align*}
    Using the interpolation error estimate \eqref{eq:errorInterpPhi}, it yields
    \begin{equation*}
        \mu^{n+1}(x_i) - 2L\Dx - \underset{(i,j)\in\Z^2}{\mathrm{sup}}\lbr |e_{\varphi,ij}^n| \rbr\leq\mu_i^{n+1}\leq\mu^{n+1}(x_i) + 2L\Dx + \underset{(i,j)\in\Z^2}{\mathrm{sup}}\lbr |e_{\varphi,ij}^n|\rbr,
    \end{equation*}
    from which one obtains \eqref{eq:errorConvViscDxMuInduction} at step $n+1$ using the induction hypothesis \eqref{eq:errorConvViscDxPhiInduction} on $e_{\varphi,ij}^n$. The same ideas are then applied to the equation on $\varphi_{ij}^{n+1}$ in \eqref{eq:VariDiscFull}. Let 
    \begin{equation*}
        \begin{aligned}
            A&\coloneqq \varphi^n_j(x_i-\Dt v_j) + (\bar{\varphi}^n_{i_j,j} - \varphi^n_j(x_i-\Dt v_j)) + \Dt,\\
            B&\coloneqq \frac{v_j^2}{2}+ \mu^{n+1}(x_i)+\left(\mu_i^{n+1}-\mu^{n+1}(x_i)\right),
        \end{aligned}
    \end{equation*}
    such that $\varphi_{ij}^n=\min\lbr A,B\rbr$. Applying estimate \eqref{eq:errorInterpPhi} to $A$ and \eqref{eq:errorConvViscDxMuInduction} at step $n+1$ to $B$, one has in particular
    \begin{align}
        A&\leq \details{\varphi_j^{n}(x_i-\Dt v_j) + \Dt + 2L\Dx \leq} \varphi_j^{n}(x_i-\Dt v_j) + \Dt + 2L\Dx ,\label{eq:estimA_Upper}\\
        A&\geq \details{\varphi_j^{n}(x_i-\Dt v_j) + \Dt - 2L\Dx \geq} \varphi_j^{n}(x_i-\Dt v_j) + \Dt - 2L\Dx ,\label{eq:estimA_Lower}\\
        B&\leq \frac{v_j^2}{2}+ \mu^{n+1}(x_i)+ 2L(n+1)\Dx ,\label{eq:estimB_Upper}\\
        B&\geq \frac{v_j^2}{2}+ \mu^{n+1}(x_i)- 2L(n+1)\Dx .\label{eq:estimB_Lower}
    \end{align}
    Therefore, combining \eqref{eq:estimA_Upper} and \eqref{eq:estimB_Upper}, one obtains the upper bound
    \begin{equation}\label{eq:upperBoundPhi}
        \varphi^{n+1}_{ij} \leq \varphi_j^{n+1}(x_i) + 2L(n+1)\Dx.
    \end{equation}
    Similarly, estimates \eqref{eq:estimA_Lower} and \eqref{eq:estimB_Lower} yield the lower bound
    \begin{equation}\label{eq:lowerBoundPhi}
        \varphi^{n+1}_{ij} \geq \varphi_j^{n+1}(x_i) - 2L(n+1)\Dx.
    \end{equation}
    Finally, \eqref{eq:errorConvViscDxPhiInduction} is obtained at step $n+1$ combining \eqref{eq:upperBoundPhi} and \eqref{eq:lowerBoundPhi}.
\end{proof}

Using Lemmas~\ref{lem:ConvViscDtDv} and \ref{lem:ConvViscDx}, one can conclude the proof of Theorem~\ref{thm:ConvVisc}.

\begin{Remark}
    \textbf{On the generalization to higher dimensions}. The generalization of Theorems~\ref{thm:AP} and~\ref{thm:ConvVisc} to higher dimensions in velocity is straightforward, as all velocity-specific tools naturally generalize to this setting. However, when considering the position variable, the monotonicity of the interpolation \eqref{eq:interpolation} played a crucial role in the analysis of both \eqref{eq:schemeMuPhiFull} and \eqref{eq:VariDiscFull}. Therefore, the multidimensional interpolation must be carefully selected to ensure this property is preserved. On suitably structured meshes, a linear interpolation in each coordinate direction should be sufficient. Additionally, since the unknown is expected to exhibit only Lipschitz regularity, this imposes an extra constraint on the interpolation method. Finally, the commutation of the interpolation \eqref{eq:interpolation} with constants and translations allowed the exact preservation of Lipschitz bounds. Without these last two properties, this preservation may deteriorate over time.
\end{Remark}

\section{Numerical results}\label{sec:NumRes}
In this section, we present numerical experiments validating the properties of schemes \eqref{eq:schemeMuPhiFull} and \eqref{eq:VariDiscFull}, as established in Theorems~\ref{thm:AP} and~\ref{thm:ConvVisc}. While these results were proven in an unbounded phase space, practical considerations require to restrict the computational domain.

For the spatial discretization, for the sake of simplicity, we impose periodic boundary conditions on the domain $[-x_\star, x_\star]$, with a uniform grid spacing defined as
\begin{equation*} 
    \Dx = \frac{2 x_\star}{N_x}, \quad x_i = -x_\star + \left(i - \frac{1}{2}\right) \Dx, \quad i \in \ldb1, N_x\rdb. 
\end{equation*}
A similar discretization is applied in velocity but no boundary condition needs to be defined as there is no transport in this variable:
\begin{equation*} 
    \Dv = \frac{2 v_\star}{N_v}, \quad v_j = -v_\star + \left(j - \frac{1}{2}\right) \Dv, \quad j \in \ldb1, N_v\rdb. 
\end{equation*}
Unless stated otherwise, we set the discretization parameters as follows:
\begin{equation*} 
    N_x = 64, \quad N_v = 61, \quad x_\star = 10, \quad v_\star = 10. 
\end{equation*}
Note that to ensure that $v=0$ belongs to the grid, we choose an odd number of points in velocity. The time step is chosen to satisfy the stability condition \eqref{eq:stabJumpCond}: $\Dt = 0.9\Dv/2$. In the following, we consider the initial data 
\begin{equation}\label{eq:InitDataNum}
    \varphi_\tin(t,x,v)= \min\lbr 3(v-3)^2+5, 5(v+7)^2+2\rbr+ 0.9\cos\left(\frac{4\pi xv}{x_\star v_\star}\right),
\end{equation}
which enjoys Assumption~\ref{Ass:LinfBounds}, and is such that $\varphi_\tin-v^2/2$ is $L$-Lipschitz in the $v$ variable. To illustrate the asymptotic-preserving property of \eqref{eq:schemeMuPhiFull}, we compare it with a standard explicit finite difference scheme on \eqref{eq:kineticScaled}. Denoting $f_{ij}^n$ an approximation of $f(t^n, x_i, v_j)$, the scheme is given by  
\begin{equation}\label{eq:naiveScheme}
    \frac{f_{ij}^{n+1} - f_{ij}^n}{\Dt} + (v \dx f^n)_{ij} = \frac{1}{\ep} (\rho_i^n \M^\ep_j - f_{ij}^n),
\end{equation}
where  
\begin{equation*}
    (v \dx f^n)_{ij} = v_j^+ \frac{f_{ij}^n - f_{i-1,j}^n}{\Dx} + v_j^- \frac{f_{i+1,j}^n - f_{ij}^n}{\Dx}
\end{equation*}  
is a classical upwind scheme for the transport term, $\rho_i^n = \lla  f_i^n \rra_\Dv$, and $\M^\ep_j = c \M^\ep(v_j)$ with $c = \lla  \M^\ep \rra_\Dv^{-1}$ being a normalization constant. To compare the results of scheme \eqref{eq:naiveScheme} to the ones of \eqref{eq:schemeMuPhiFull}, we introduce the Hopf-Cole transform of the solution to \eqref{eq:naiveScheme}: for all $(n,i,j)\in \ldb0,N_t\rdb\times\ldb1, N_x\rdb\times\ldb1, N_v\rdb $,
\begin{equation*}
    \psi_{ij}^{n} = -\ep \log(f_{ij}^n).
\end{equation*}  
Since the scheme \eqref{eq:naiveScheme} is explicit, its stability is guaranteed only under a CFL condition of the form $\Dt \leq \ep C \Dx$, which becomes excessively restrictive as $\ep \to 0$. Moreover, due to the concentration phenomena in velocity as $\ep \to 0$, this scheme is expected to poorly resolve this variable in that regime. Another challenge arises from the use of the logarithm in the Hopf-Cole transform: for small values of $\ep$, the distribution tails approach the machine epsilon, leading to significant numerical errors. Consequently, the scheme \eqref{eq:naiveScheme} is expected to approach accurately the solution to \eqref{eq:kineticScaled} only in the regime $\ep\sim1$.

\begin{Remark}
    With the discretization parameters presented above, the exact preservation of the equilibrium $v^2/2$ is observed. More precisely, setting $\varphi_\tin=v^2/2$, the quantities $\max_{1\leq n\leq N_t}\lbr \varphi^\ep - v^2/2 \rbr$ and $\max_{1\leq n\leq N_t}\lbr \varphi - v^2/2 \rbr$ are exactly equal to 0.
\end{Remark}

\subsection{AP property}
We start by evaluating the AP property of \eqref{eq:schemeMuPhiFull}. This is done through a comparison with both the limiting scheme \eqref{eq:VariDiscFull} and the explicit (non-AP) scheme \eqref{eq:naiveScheme}.

Figures~\ref{fig:comparisonDynamicX} and~\ref{fig:comparisonDynamicV} illustrate the evolution of $\varphi^\ep$, $\psi$ and $\varphi$ under various values of the scaling parameter $\ep$. In Figure~\ref{fig:comparisonDynamicX}, $\varphi^\ep$, $\psi$ and $\varphi$ are shown as a function of $x$ with fixed velocity $v = 0$. We observe that the explicit scheme (top row) fails to capture the correct limiting behavior as $\ep \to 0$, whereas the AP scheme \eqref{eq:schemeMuPhiFull} (bottom row) is able to align with the asymptotic solution. In addition, it is worth noticing that, as $\ep$ decreases, sharp spatial peaks emerge, and a Lipschitz-type regularity in $x$ becomes apparent in the limiting profile. This behavior is typical of Hamilton-Jacobi equation.

Figure~\ref{fig:comparisonDynamicV} presents the same simulation but now with fixed spatial coordinate $x = 0$, displaying $\varphi^\ep$, $\psi$ and $\varphi$ as a function of $v$. Once again the AP scheme approximates the limit well, while the explicit scheme exhibits notable discrepancies. In particular, the explicit method performs poorly for large velocities. This can be attributed to the procedure it follows—first computing the distribution function and then taking the logarithm to recover $\psi$. At high velocities, this approach suffers from numerical instabilities due to values near machine precision, leading to inaccuracies in the logarithm.

Finally, the convergence of the AP scheme as $\varepsilon \to 0$ is quantified in Figure~\ref{fig:APconvergenceEpsilon}. The pointwise $L^\infty$-error between $\varphi^\ep$ and $\varphi$ converges as $\ep\to0$, as predicted in Theorem~\ref{thm:AP}. While Theorem~\ref{thm:AP} does not provide an explicit rate of convergence with respect to $\ep$, a rate of $1$ can be observed for this test case.

\begin{figure}
    \centering
    \includegraphics[width=.95\linewidth]{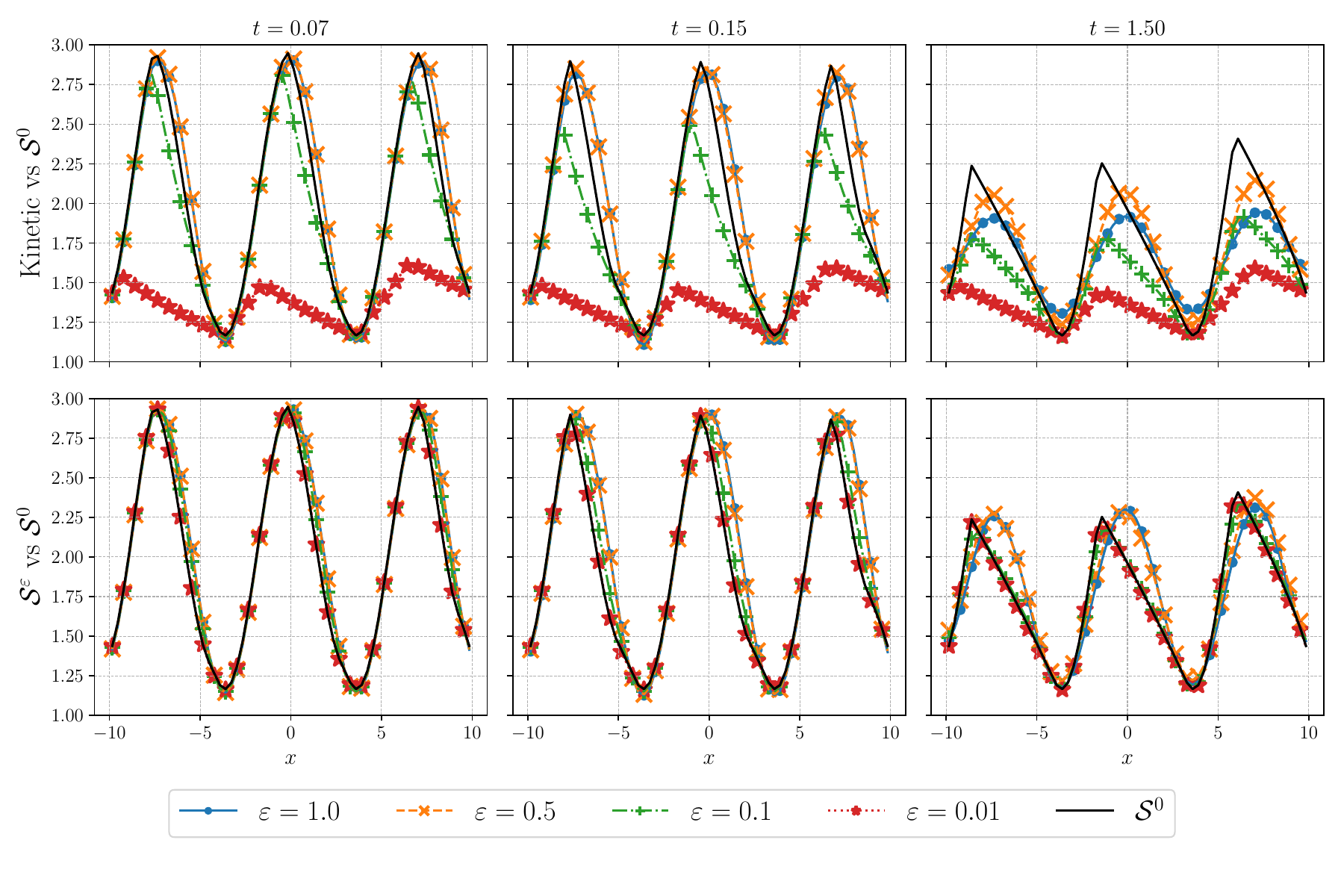}
    \caption{Evolution of the solution to \eqref{eq:naiveScheme} (Kinetic, top) and \eqref{eq:schemeMuPhiFull} (bottom) compared to the solution to \eqref{eq:VariDiscFull} for $v=0$ at $t=0.07$ (left), $0.15$ (center), and $1.5$ (right) for different values of $\ep$.}
    \label{fig:comparisonDynamicX}
\end{figure}

\begin{figure}
    \centering
    \includegraphics[width=.45\linewidth]{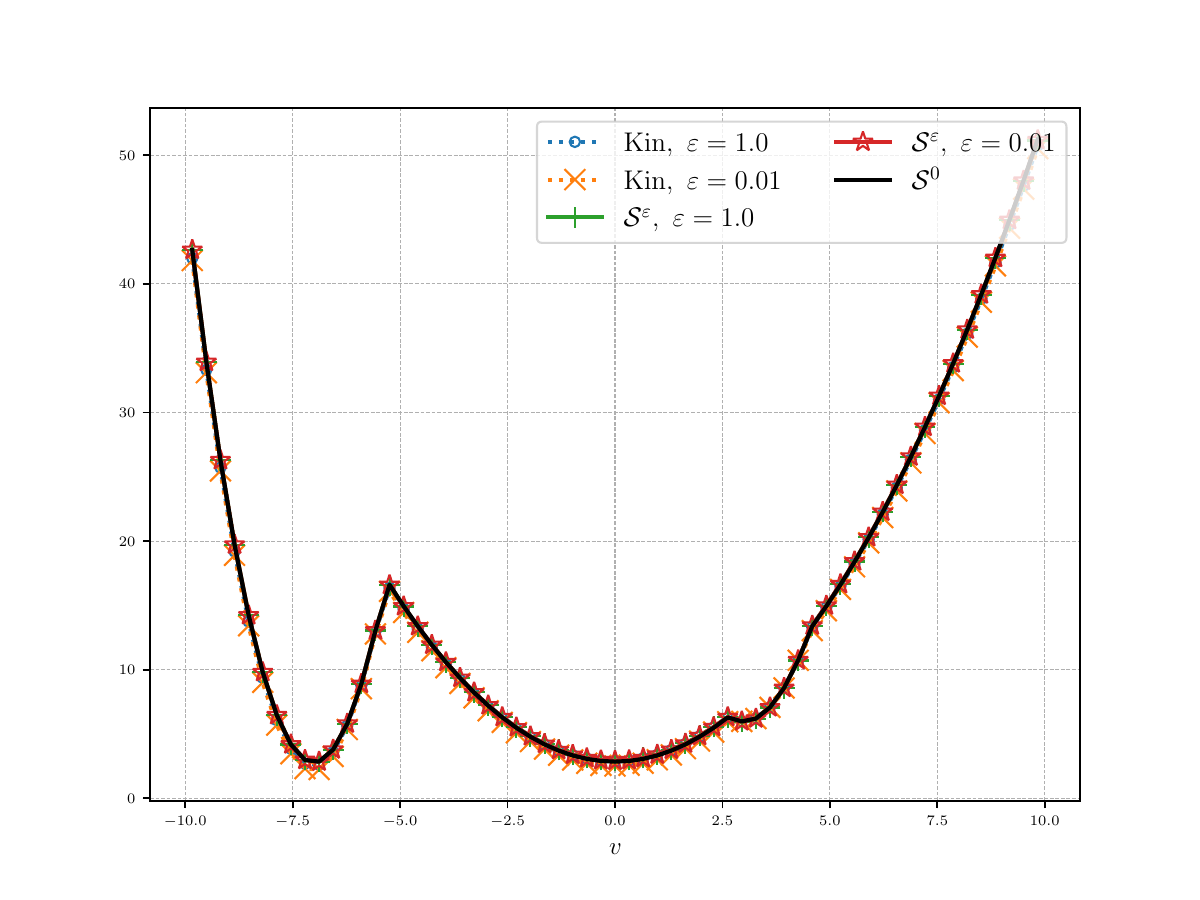}
    \includegraphics[width=.45\linewidth]{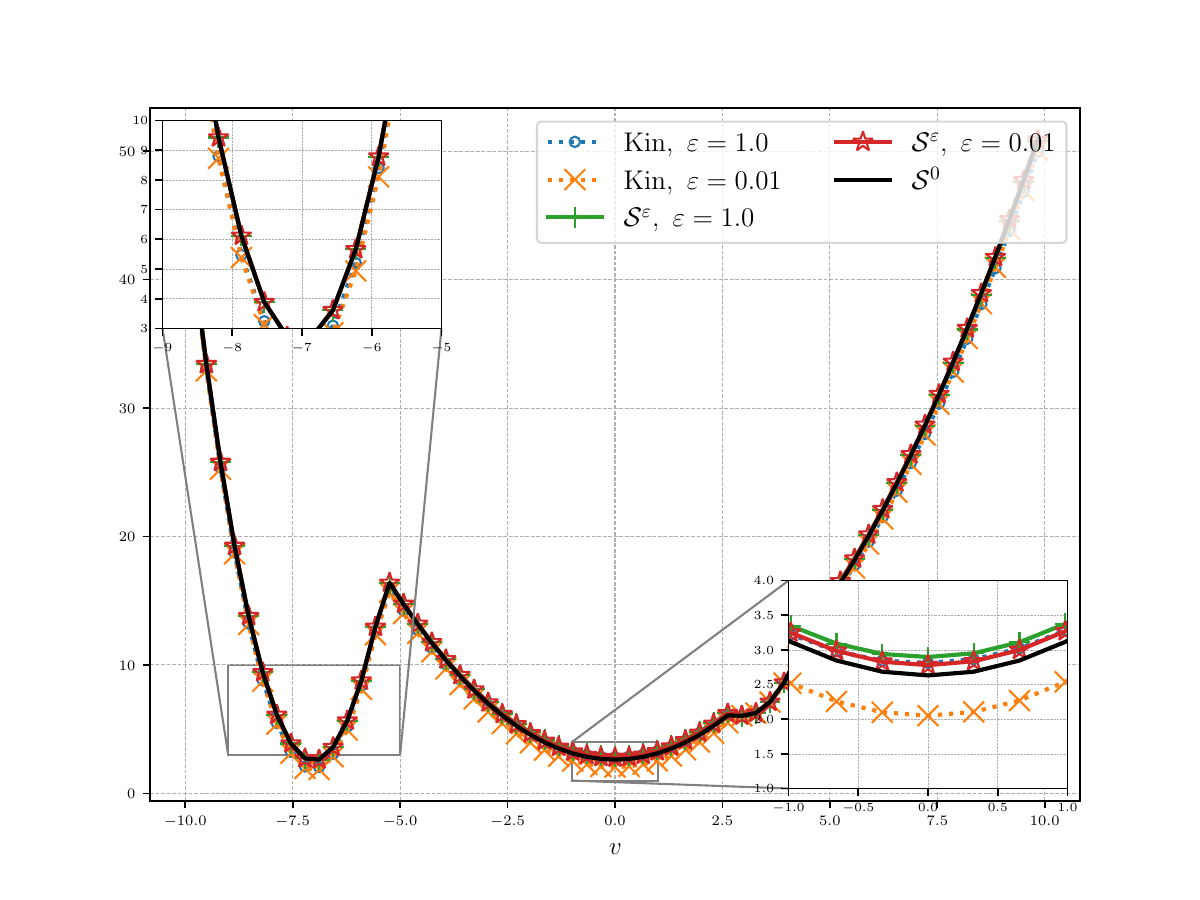}\\
    \includegraphics[width=.45\linewidth]{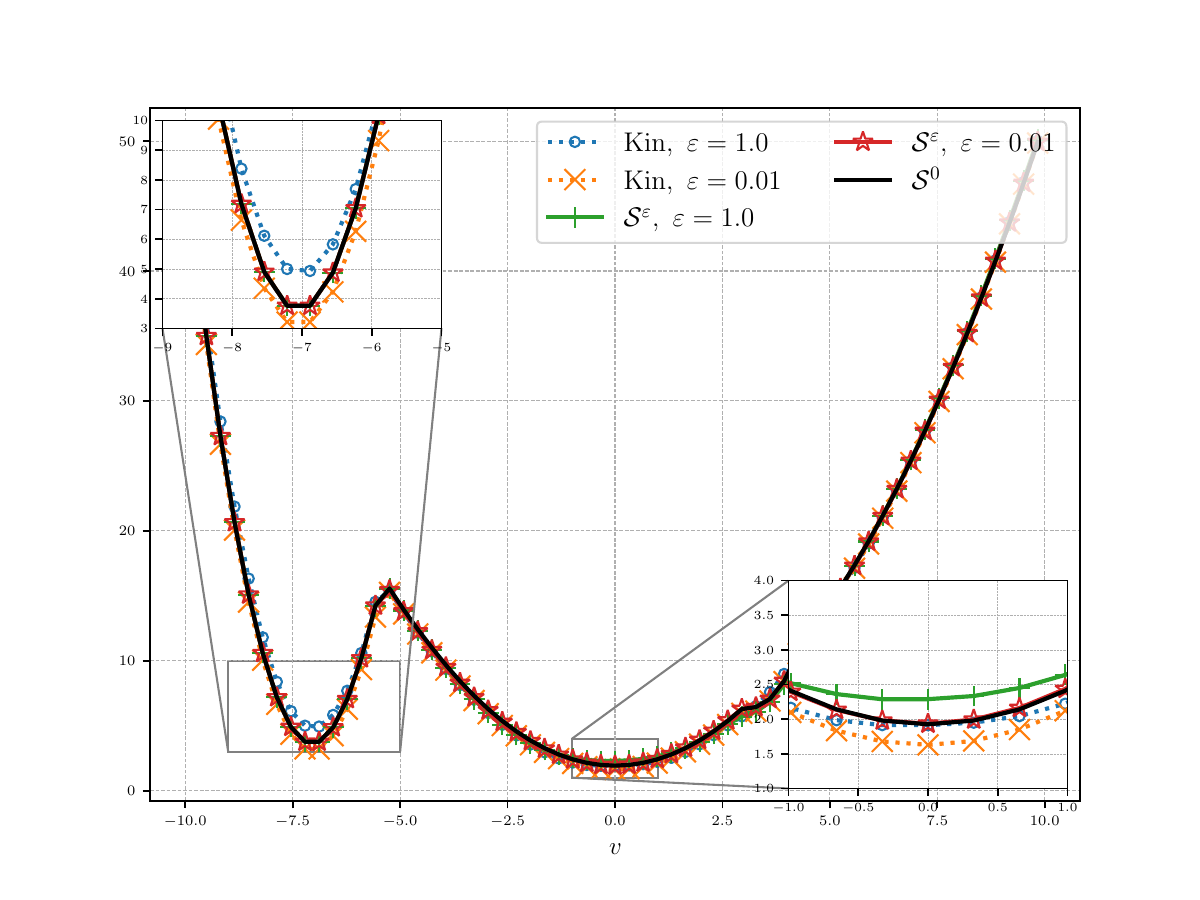}
    \includegraphics[width=.45\linewidth]{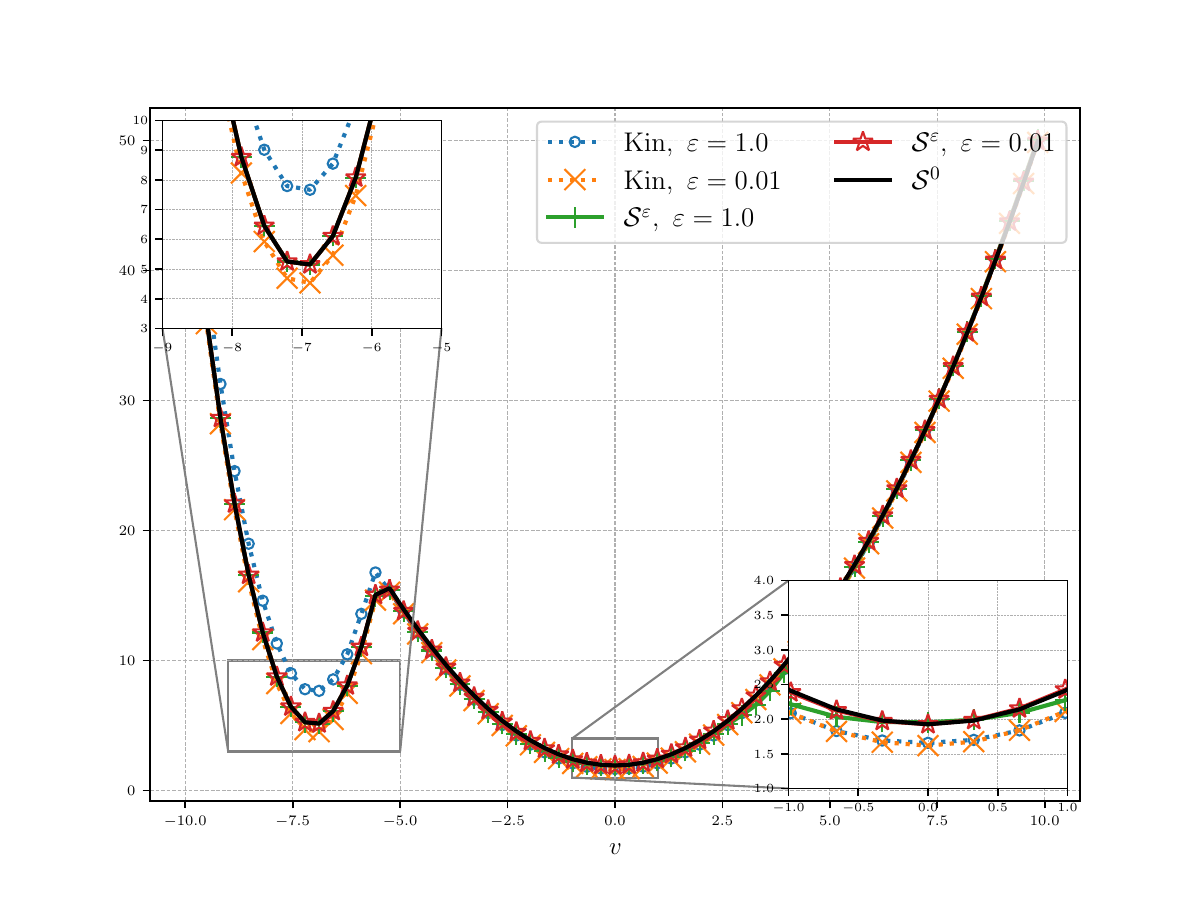}
    \caption{Comparison of the solutions to \eqref{eq:schemeMuPhiFull}, \eqref{eq:VariDiscFull} and \eqref{eq:naiveScheme}, for $x=0$, at $t=0.075$ (top left), $0.15$ (top right), $1.5$ (bottom left) and $3.0$ (bottom right) for different values of $\ep$.}
    \label{fig:comparisonDynamicV}
\end{figure}

\begin{figure}
    \centering
    \includegraphics[width=.5\linewidth]{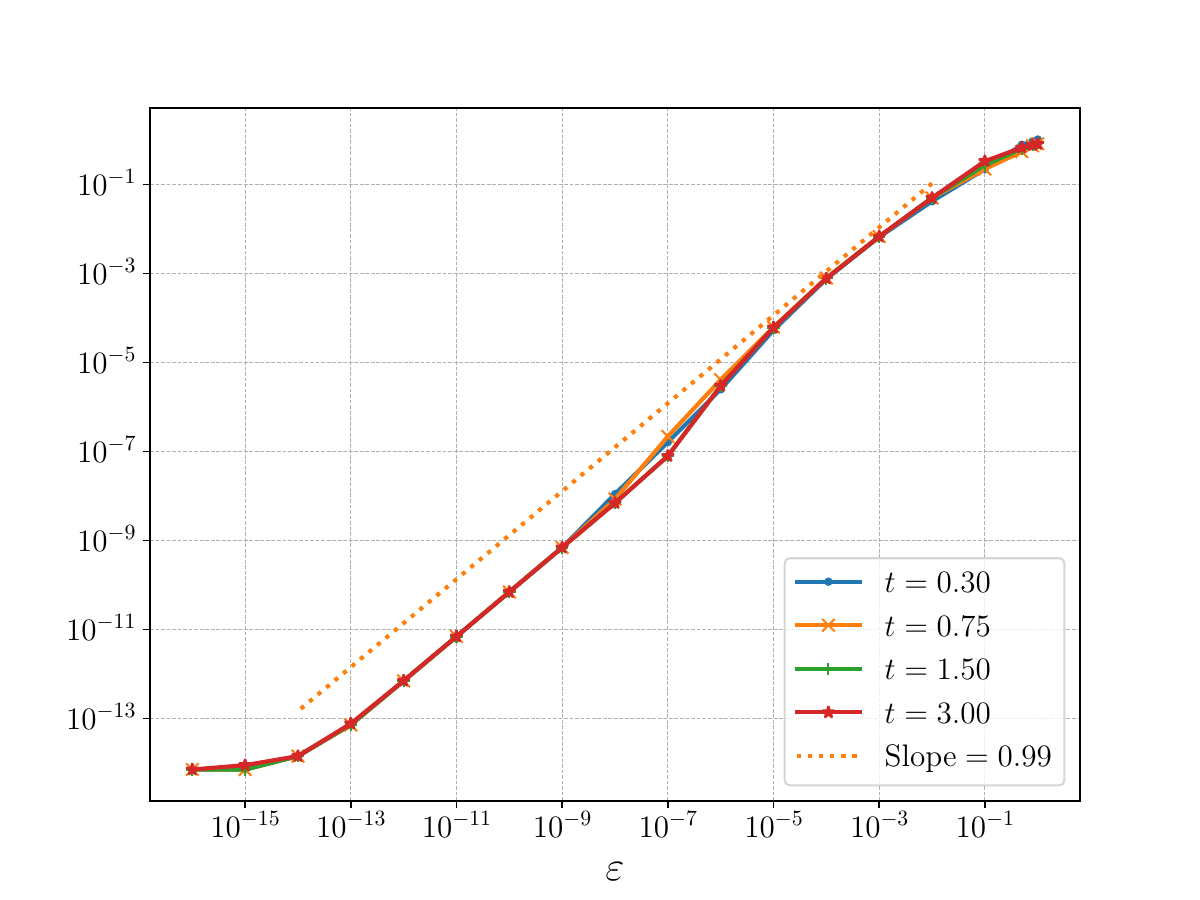}
    \caption{\textbf{Convergence in $\ep$.} $L^\infty$-error between \eqref{eq:schemeMuPhiFull} and \eqref{eq:VariDiscFull}, as a function of $\ep$ for different times}
    \label{fig:APconvergenceEpsilon}
\end{figure}

\subsection{Convergence}
We now illustrate the convergence behavior of schemes \eqref{eq:schemeMuPhiFull} and \eqref{eq:VariDiscFull} with respect to the discretization parameters $\Dt$, $\Dv$, and $\Dx$. The error associated to a given triplet of parameters $\Delta=(\Dt, \Dv, \Dx)$ is defined by
\begin{equation*}
E(\ep,\Delta) = \frac{|\varphi^\ep_{\mathrm{Ref}} - u^\ep_{\Delta}|_\infty}{|\varphi^\ep_{\mathrm{Ref}}|_\infty},
\end{equation*}
where $E(0,\Delta)$ corresponds to the error of the limiting scheme and $u^\ep_{\Delta}=\varphi^\ep$, $\varphi$ or $\psi$ computed with \eqref{eq:schemeMuPhiFull}, \eqref{eq:VariDiscFull} and \eqref{eq:naiveScheme} respectively. The reference solution $\varphi^\ep_{\mathrm{Ref}}$ is computed using a sufficiently fine discretization, while $u^\ep_\Delta$ is obtained using varying values of $\Delta$. All convergence tests use \eqref{eq:InitDataNum} as initial data. The convergence curves associated to the schemes \eqref{eq:schemeMuPhiFull}, \eqref{eq:VariDiscFull}, and the explicit scheme \eqref{eq:naiveScheme} are displayed in Figure~\ref{fig:convergenceOrder}.

To study the convergence with respect to the spatial step $\Dx$, we fix $N_v = 201$, $\Dt = 2.5 \times 10^{-5}$, and $T = 0.01$. The reference solution is computed using $N_x = 2^{15}$. The AP scheme \eqref{eq:schemeMuPhiFull} for $\ep=1$ and limit scheme \eqref{eq:VariDiscFull} both exhibit an observed convergence order of approximately $0.7$. The deterioration of the expected order $1$ of these two schemes can be explained by the Lipschitz (and not $\mathcal{C}^1$) regularity of the approximated solution. Nevertheless, the schemes do converge. For scheme \eqref{eq:VariDiscFull}, the results should be interpreted in light of the error estimate \eqref{eq:errorFull} from Theorem~\ref{thm:ConvVisc}, where the ratio $\Dx/\Dt$ appears. As a result, the influence of $\Dx$ also cannot be fully separated from that of $\Dt$.

The convergence with respect to the velocity step $\Dv$ is investigated by setting $N_x = 256$, $\Dt = 2.5 \times 10^{-5}$, and $T = 0.01$. A reference solution is computed using $N_v = 3^{10}$. In this case, the error curves for all schemes are superimposed, and an order of $1$ is observed, in agreement with the theoretical prediction from \eqref{eq:errorFull}.

To analyze the convergence with respect to the time step $\Dt$, we fix $N_v = 101$ and $T = 0.5$, and compute the reference solution using $N_t = 2^{11}$. As noted in the spatial convergence study, $\Dt$ also appears in the denominator in \eqref{eq:errorFull}, so letting $\Dt \to 0$ requires adapting $\Dx$ accordingly to avoid divergence of the ratio $\Dt/\Dx$. As discussed in Section~\ref{subsec:ConvFull}, if the relation $\Dt\Dv = \Dx$ is satisfied, the transport step defined in \eqref{eq:interpolation} becomes exact, and the term $\Dt/\Dx$ in \eqref{eq:errorFull} vanishes. We adopt this setting and determine the number of grid points from $N_x = \left\lfloor 2x_\star / \Dx \right\rfloor$. With this choice, the convergence in $\Dt$ can be properly assessed, and the predicted first-order behavior in time is observed. Note that the errors of \eqref{eq:schemeMuPhiFull} and \eqref{eq:VariDiscFull} actually coincide. In addition, the naive scheme \eqref{eq:naiveScheme} fails to achieve first-order accuracy in time for this test case, highlighting the need for an appropriate discretization of the evolution equation for $\varphi^\ep$.

\begin{figure}
    \centering
    \includegraphics[width=.45\linewidth]{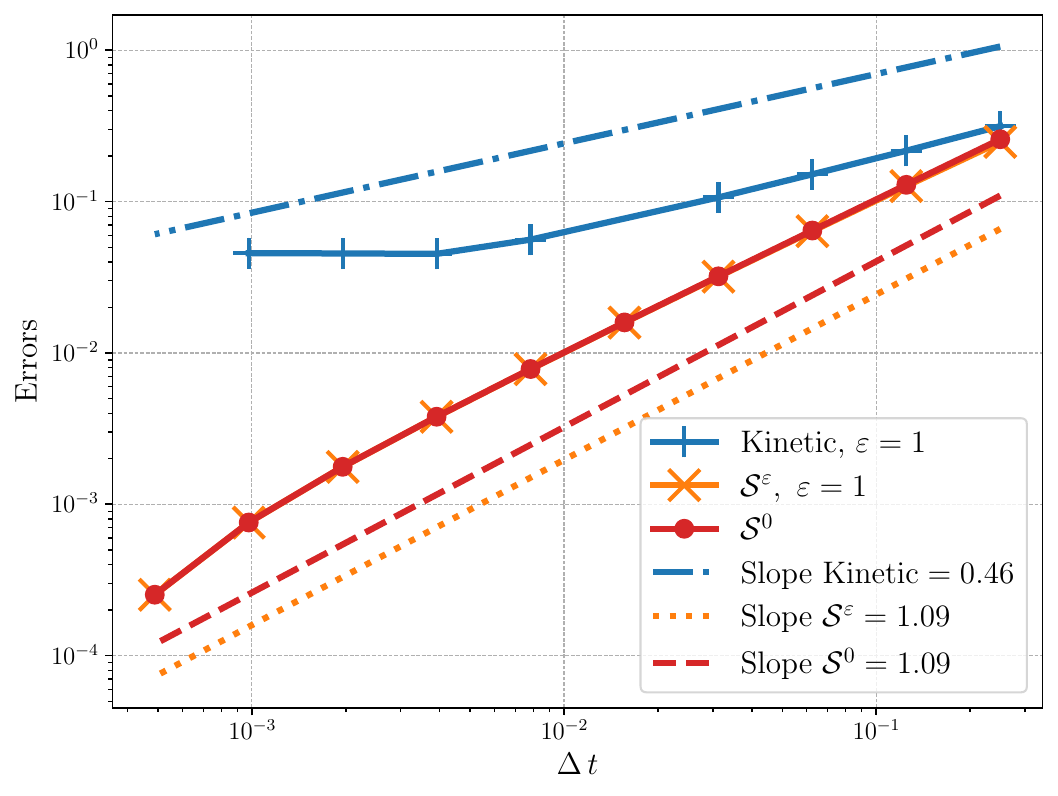}
    \includegraphics[width=.45\linewidth]{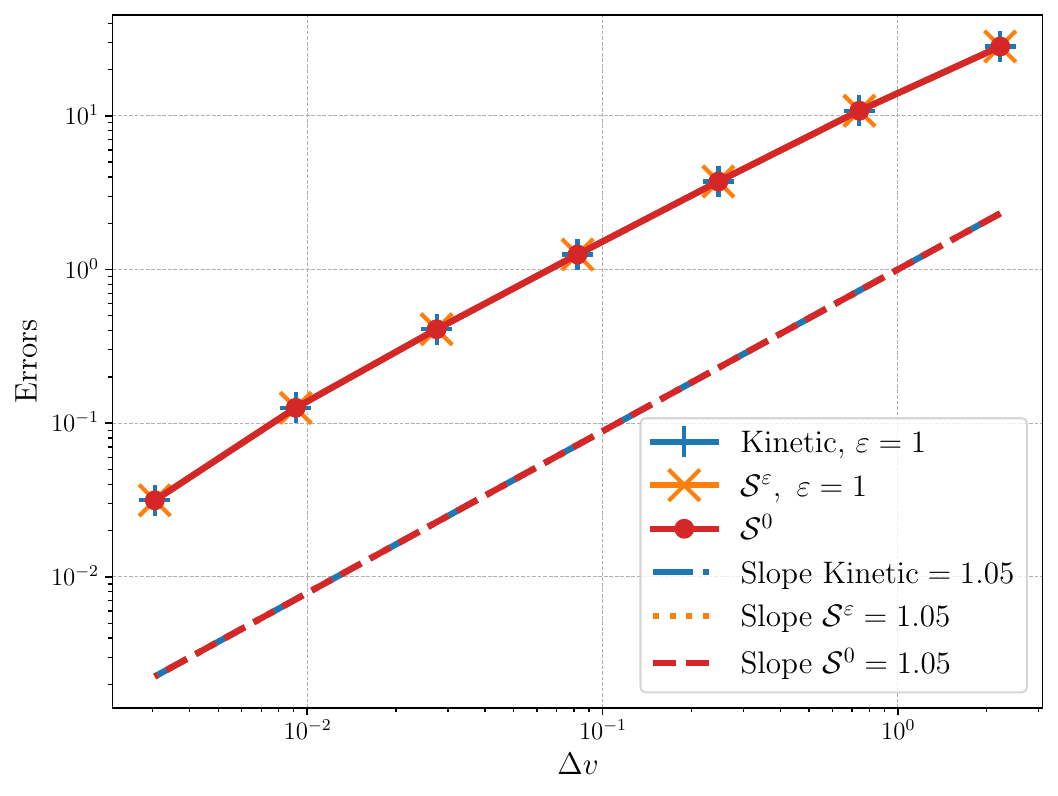}\\
    \includegraphics[width=.45\linewidth]{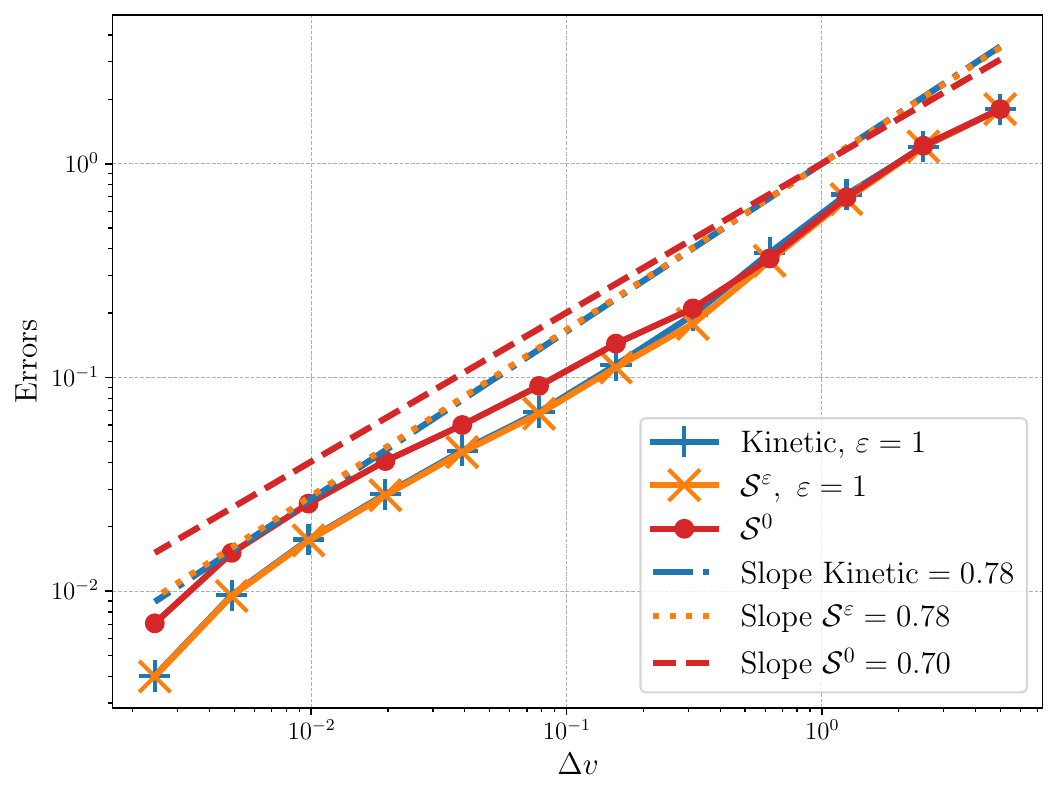}
    \caption{\textbf{Convergence study of the schemes \eqref{eq:schemeMuPhiFull}, \eqref{eq:VariDiscFull} and \eqref{eq:naiveScheme}.} Values of $E(1,\Delta)$ and $E(0,\Delta)$ as a function of $\Dt$ (top left), $E(1,\Delta)$ and $E(0,\Delta)$ as a function of $\Dv$ (top right) and $E(1,\Delta)$ and $E(0,\Delta)$ as a function of $\Dx$ (bottom).}
    \label{fig:convergenceOrder}
\end{figure}

\details{
\subsection{Equilibrium preserving property}
We now validate the equilibrium preserving properties proven in Lemmas~\ref{lem:WBeps} and~\ref{lem:WBLim}. The initial data is set to 
\begin{equation*}
    \varphi_\tin(x,v)=\frac{v^2}{2}.
\end{equation*}
Figure~\ref{fig:equilibriumPreserving} shows that in both cases, all curves coincide and the solution remains stationary. This is further confirmed by computing the $L^\infty$ error with respect to the initial data over time.

\begin{figure}
    \centering
    \includegraphics[width=.48\linewidth]{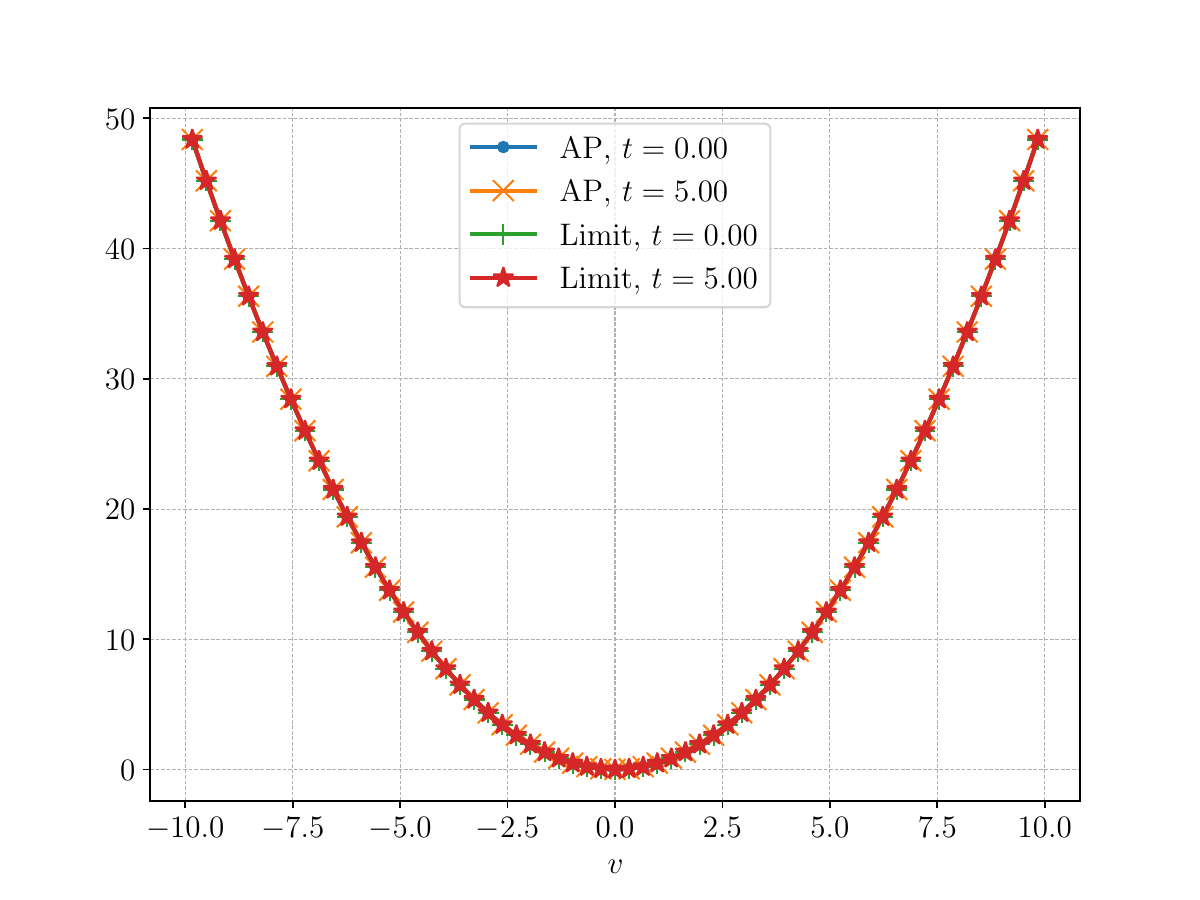}
    \includegraphics[width=.48\linewidth]{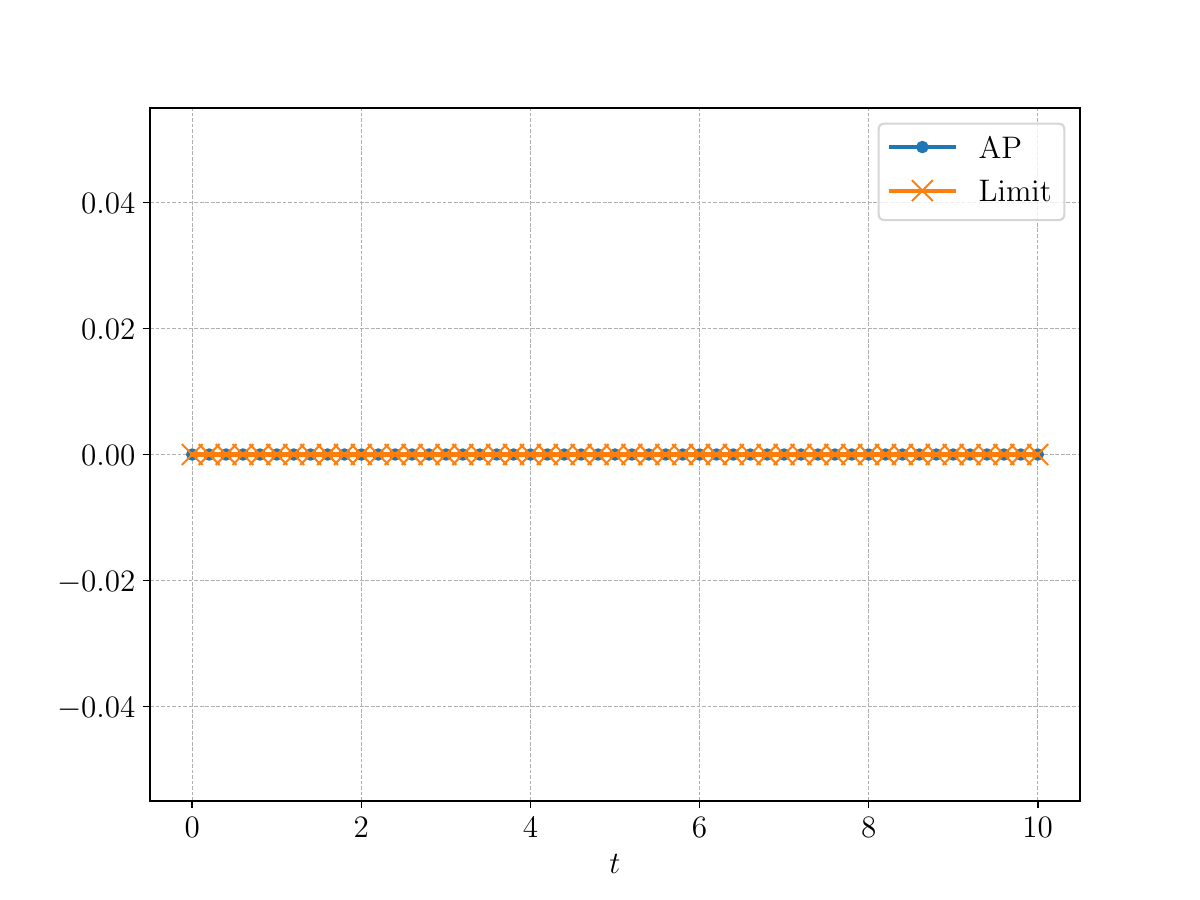}
    \caption{\textbf{Equilibrium preserving property.} Solution (left) and $L^\infty$-error with the initial data as a function of time (right).}
    \label{fig:equilibriumPreserving}
\end{figure}
}

\subsection{Decreasing of the minimum}
We now investigate the decay of the minimum of $\varphi$. Recall that the limit system \eqref{eq:Vari} satisfies the condition $\partial_t \min_v \lbr \varphi \rbr \leq 0$. As shown in Lemma~\ref{lem:DiscDTmuneg}, the discrete system \eqref{eq:VariDiscFull} also preserves a discrete analogue of this property. Figure~\ref{fig:decayMinPhi} displays the evolution of $\min_j \lbr \varphi \rbr$ for the limit scheme. For comparison, the evolution of $\min_j \lbr \varphi^\ep \rbr$ computed with \eqref{eq:schemeMuPhiFull} for $\ep=1$ is also displayed in Figure~\ref{fig:decayMinPhi}. For the limit scheme, we observe a monotonic decrease of the minimum, and the solution appears to reach a non-trivial stationary state. In contrast, the simulation with $\ep = 1$ shows that the minimum can increase, confirming that this monotonicity is indeed a behavior specific to the asymptotic regime.

\begin{figure}
    \centering
    \includegraphics[width=.45\linewidth]{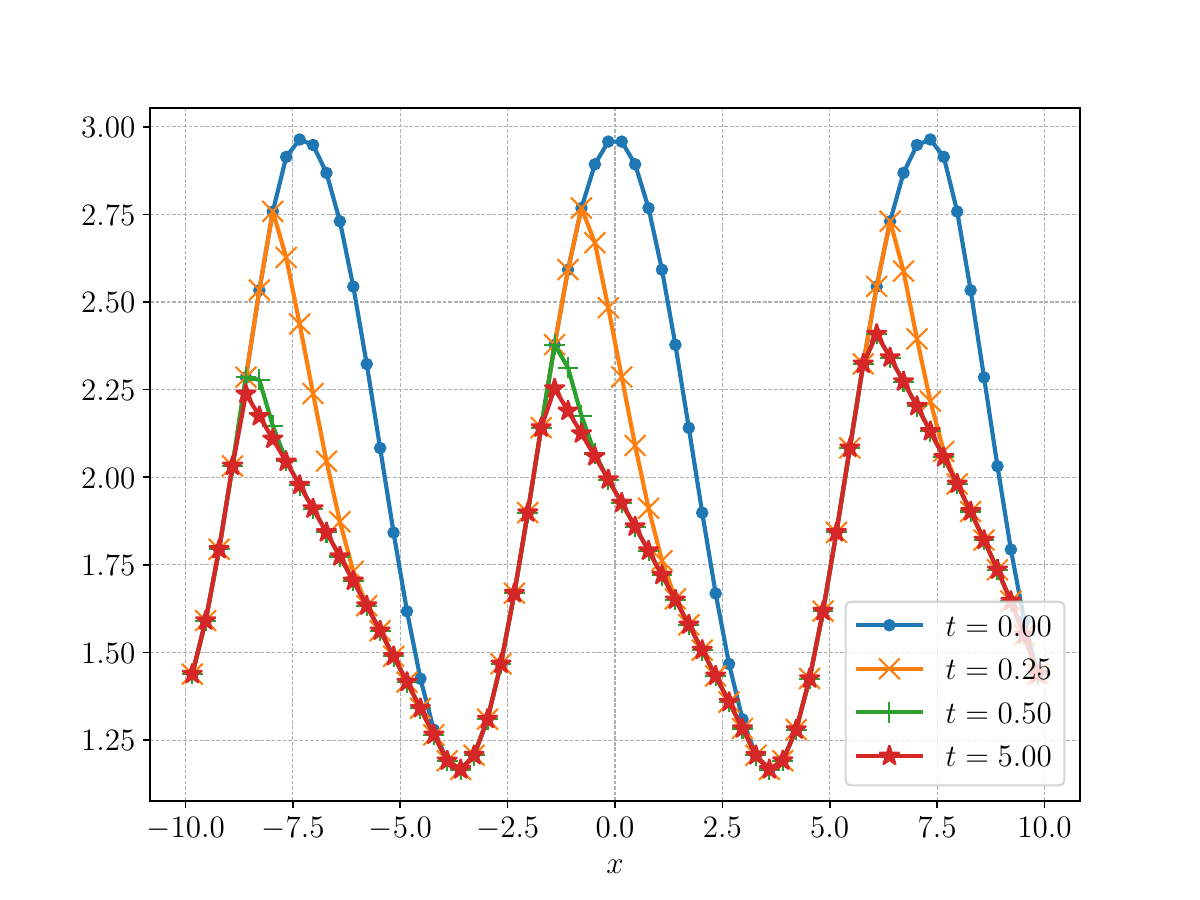}
    \includegraphics[width=.45\linewidth]{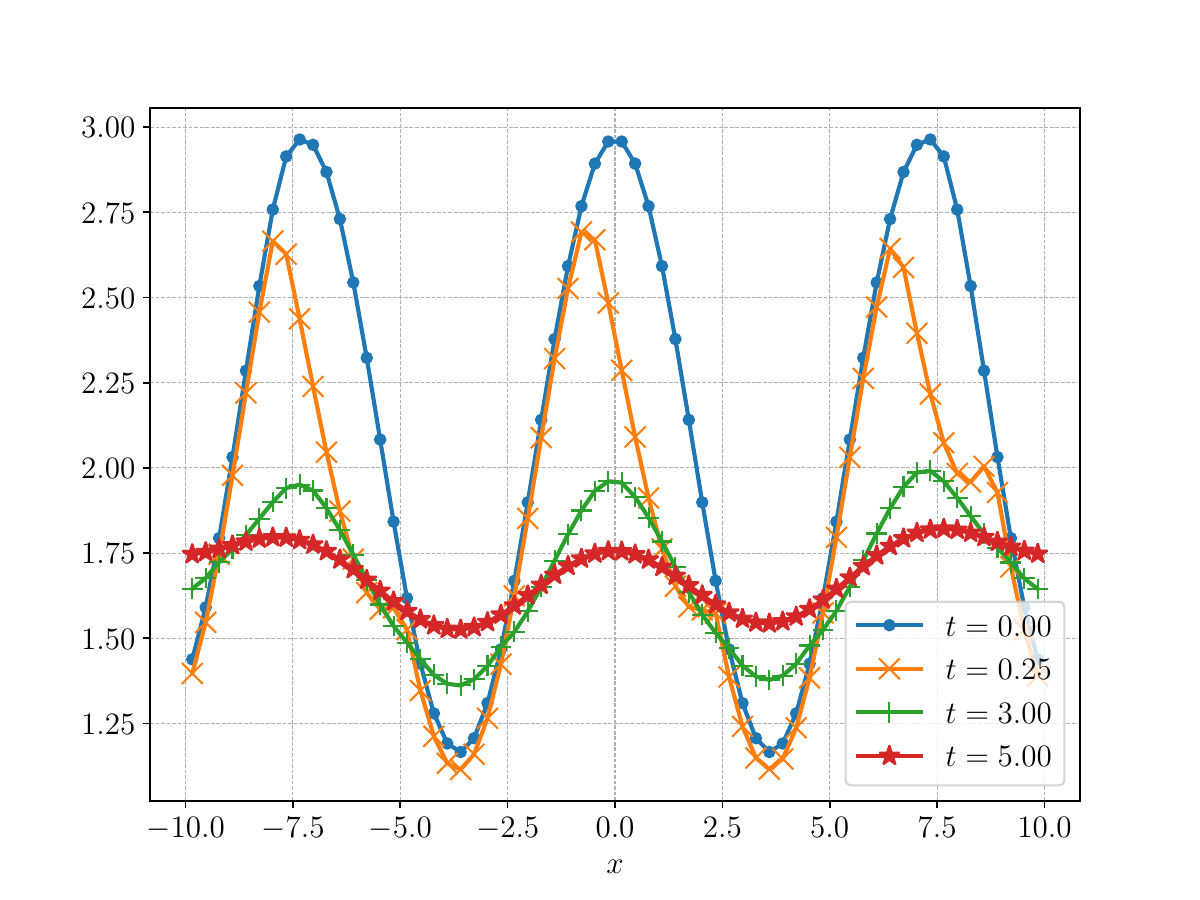}
    \caption{Evolution of $\min_{v} \lbr\varphi\rbr(t,x)$ (left) and $\min_{v} \lbr\varphi^\ep\rbr(t,x)$ for $\ep = 1$ (right), as functions of $x$ at different times.}
    \label{fig:decayMinPhi}
\end{figure}

\subsection{Persistence of the initial data}

In this final numerical experiment, we examine the long-time behavior of the solution. It is well known (see, e.g., \cite{DolbeaultMouhotSchmeiser2015}) that for $\ep = 1$, the solution $f^\ep$ to \eqref{eq:kineticScaled} converges, as $t \to \infty$, towards a stationary state given by 
\begin{equation*} 
    f_\infty = \rho_\infty \mathcal{M}(v), 
\end{equation*} 
where $\rho_\infty$ is a positive constant that depends on the initial mass and the length of the periodic spatial domain $\mathbb{T}$. This convergence is driven by the mixing effect resulting from the interplay between the transport operator $v \cdot \nabla_x$ and the relaxation term $\rho \mathcal{M} - f$. Notably, this mixing mechanism remains active even when $\ep > 0$, i.e., after the rescaling in \eqref{eq:kineticScaled}. However, as already hinted at in Figure~\ref{fig:LTBSchematic}, a striking observation is that in the limit $\ep \to 0$, this behavior no longer holds: the stationary profile in position is no longer uniform but retains features of the initial data. Mathematically, this reflects the non-commutation of the limits $\ep \to 0$ and $t \to \infty$.

To illustrate this phenomenon, we consider the initial data \eqref{eq:InitDataNum}. Figure~\ref{fig:Phi_xLTB} compares, at fixed velocity $v = 0$, the long-time behavior of the solution to \eqref{eq:VariDiscFull} with that of \eqref{eq:schemeMuPhiFull} for $\ep = 1$ and $\ep = 0.1$. As $\ep$ decreases, the mixing effect becomes progressively weaker. Another way to observe and quantify this weakening as $\ep \to 0$ is through the spatial amplitude of the solution. The left panel of Figure~\ref{fig:LTBamplitude} shows the evolution of
\begin{equation}
\max_{x\in\TT} u(t,x,0) - \min_{x\in\TT} u(t,x,0),
\end{equation}
computed from the solutions to \eqref{eq:schemeMuPhiFull}, \eqref{eq:VariDiscFull}, and \eqref{eq:naiveScheme}. The weakening of the mixing effect is evident: as $\ep$ decreases, it takes increasingly longer for the amplitude to decay to a given value. Furthermore, the asymptotic-preserving property of \eqref{eq:schemeMuPhiFull} is clearly visible, as it correctly captures the convergence toward the limiting behavior. In contrast, the classical scheme \eqref{eq:naiveScheme} does converge to a stationary profile with non-zero amplitude, but this profile does not correspond to that of the limiting scheme. On the right panel of Figure~\ref{fig:LTBamplitude}, only \eqref{eq:schemeMuPhiFull} and \eqref{eq:VariDiscFull} are compared. In particular, one can observe that at short times there appears to be a characteristic point of inflexion in the curve (between $t=1$ and $t=1.5$), where the dynamics intersect. This highlights the change in the system’s behavior induced by the scaling \eqref{eq:scaling}.

\begin{figure}
    \includegraphics[width=.45\linewidth]{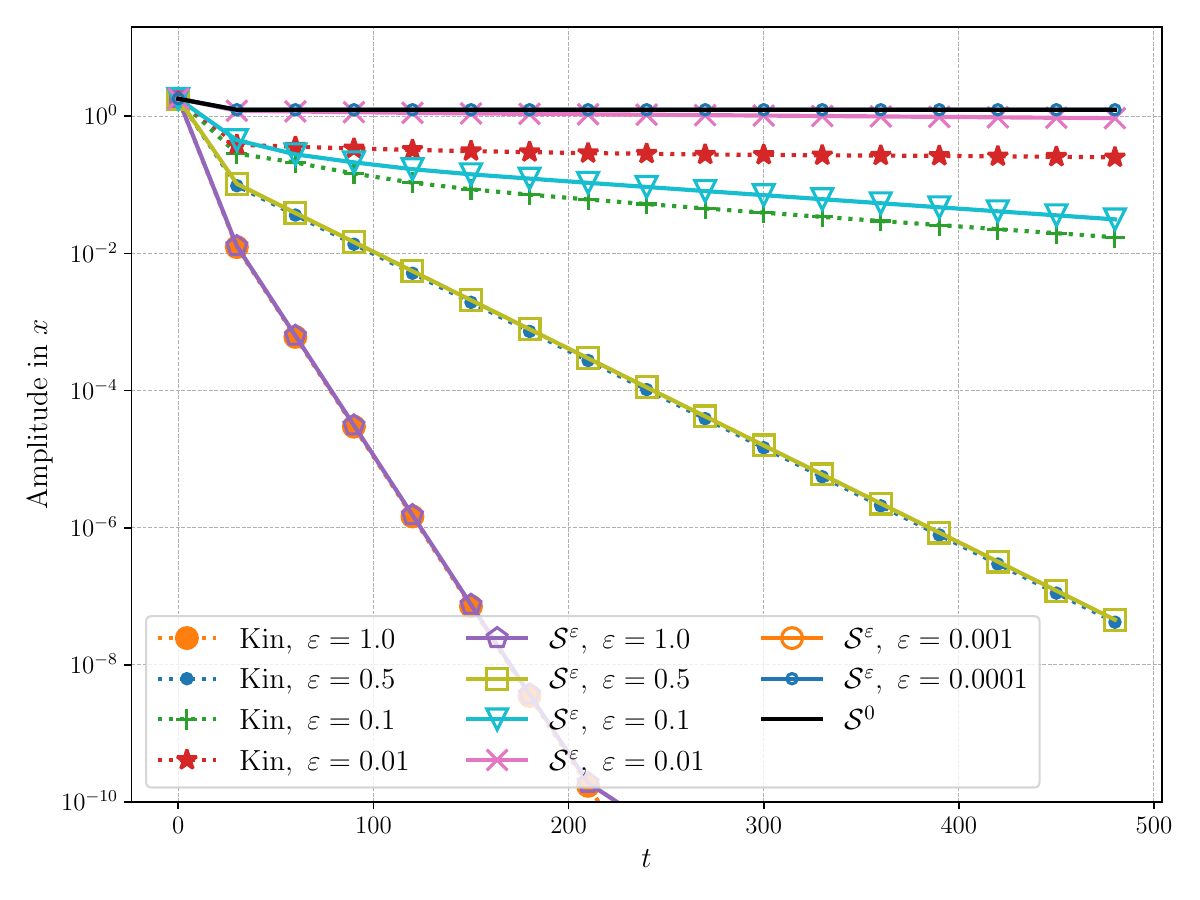}
    \includegraphics[width=.45\linewidth]{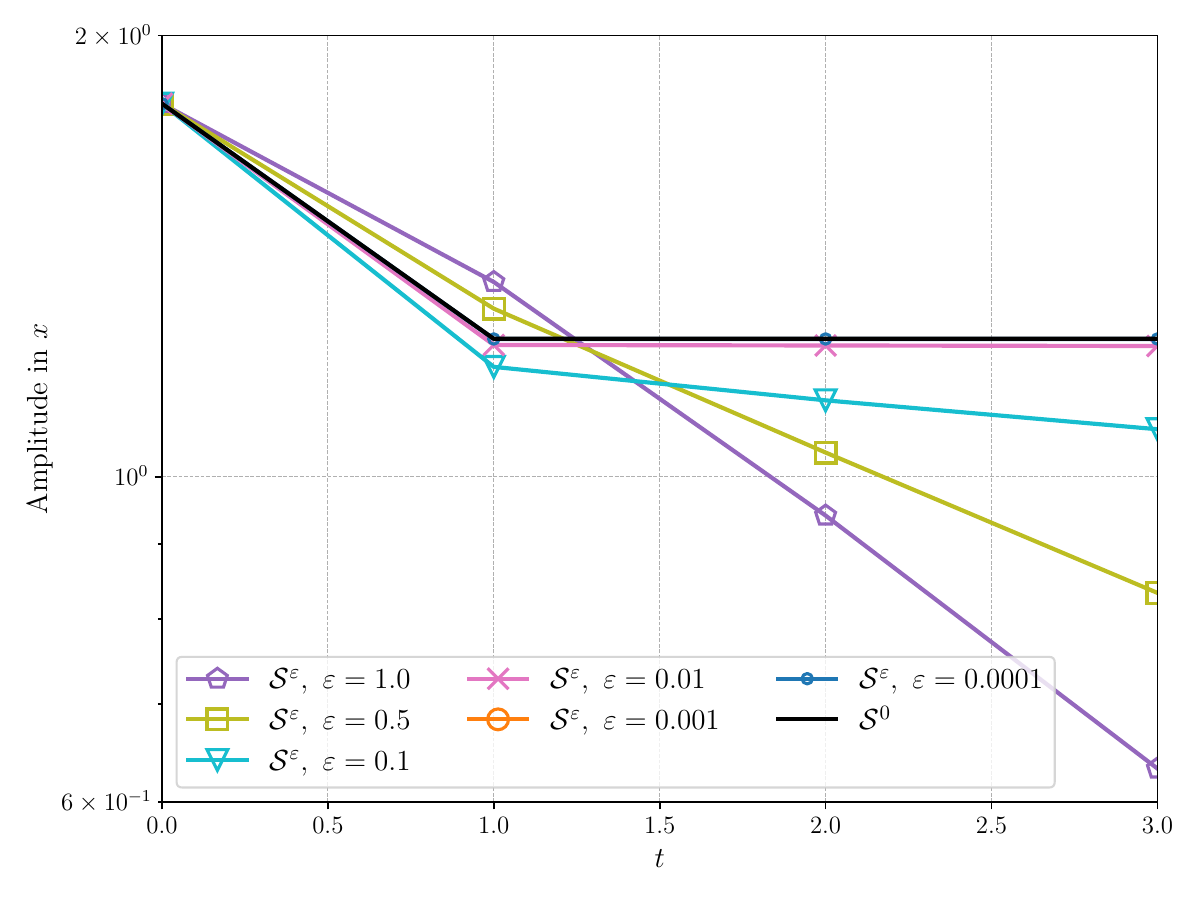}
    \caption{Long-time (left) and short-time (right) spatial amplitude evolution of the solution, evaluated at $v=0$, corresponding to \eqref{eq:schemeMuPhiFull}, \eqref{eq:VariDiscFull} and \eqref{eq:naiveScheme}, for various values of $\ep$.}
    \label{fig:LTBamplitude}
\end{figure}

An even more striking illustration of this property is obtained by looking at the fundamental solution of \eqref{eq:Vari}, starting—at the distributional level—from a Dirac mass in position and a Gaussian in velocity. Numerically, since computing the logarithm of a Dirac mass is not feasible for initializing $\varphi$, the initial data $\varphi_{\text{in}}$ is instead set to zero on a few grid cells centered around $x = 0$, and to $100$ elsewhere.

The asymptotic behavior as $t \to \infty$ was characterized in the one-dimensional case in \cite{BouinCalvezGrenierNadin2023}. The action between a point $x$ and the origin is given by 
\begin{equation}\label{eq:kernel} 
    \Ab_0^t[\gamb] = \lbr\begin{aligned} \frac{3}{2}|x|^{2/3}, & \text{if } |x| \leq t^{3/2}, \\ \frac{|x|^2}{2t^2} + t, & \text{if } |x| \geq t^{3/2}.\end{aligned}\right.
\end{equation}

The evolution of $\varphi$ is shown in Figure~\ref{fig:Phi_xLTB_Dirac}. As $t \to \infty$, the spatial profile converges to the cusp-shaped function 
\begin{equation*} 
    x \mapsto \frac{3}{2}|x|^{2/3},
\end{equation*} 
as predicted by the first case of \eqref{eq:kernel}.

Figure~\ref{fig:Phi_xLTB_Dirac} also presents the evolution from two Dirac masses. In this case, the solution converges to the superposition of two cusps, and a clear imprint of the initial data remains visible in the long-time regime. A similar phenomenon is illustrated in Figure~\ref{fig:Phi_xLTB}: for $\ep = 0$, the oscillatory structure of the initial condition \eqref{eq:InitDataNum} persists as $t \to \infty$, whereas for $\ep > 0$, the solution tends to relax and lose memory of the initial state due to mixing. It is also worth noting that, as $\ep$ decreases, relaxation towards a constant occurs more slowly—larger times are required to observe this behavior, as seen in the top left and top right panels of Figure~\ref{fig:Phi_xLTB}.  This persistent structure in the limiting regime is further illustrated in Figure~\ref{fig:EvolutionXVlim}, which shows the long-time behavior of the solution to \eqref{eq:VariDiscFull} in the $(x,v)$-phase space: the profile converges to the parabolic shape $v^2/2$ in velocity, while remaining non-trivial and structured in position.

\begin{figure}
    \centering
    \includegraphics[width=.45\linewidth]{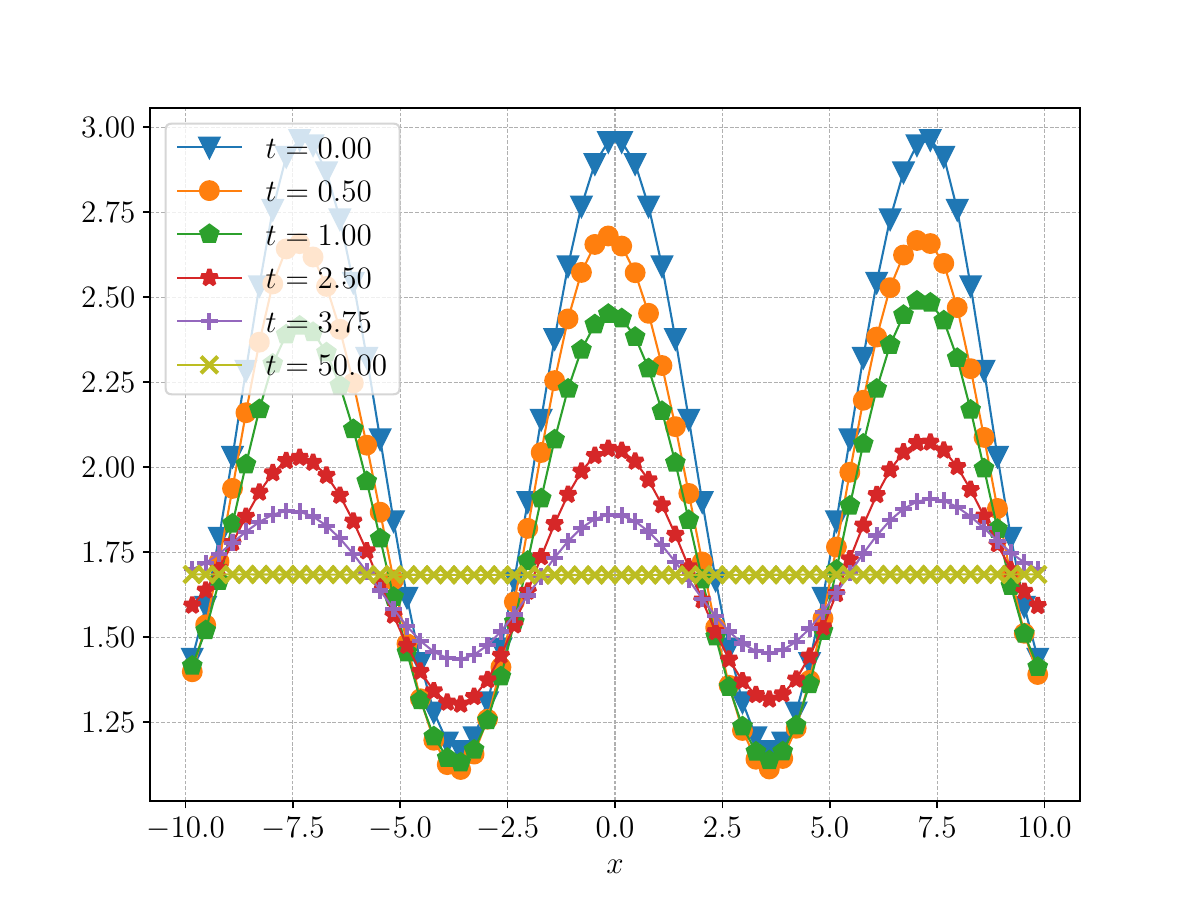}
    \includegraphics[width=.45\linewidth]{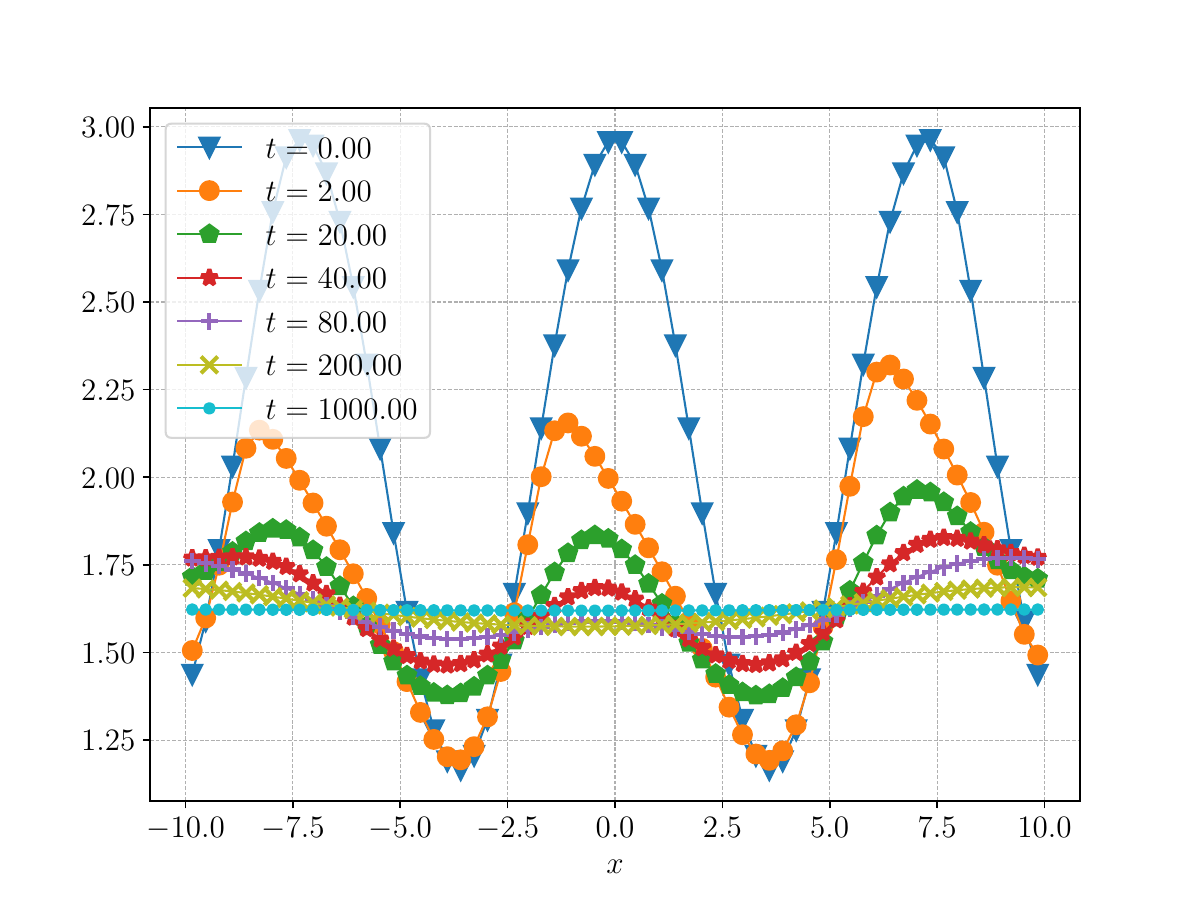}\\
    \includegraphics[width=.45\linewidth]{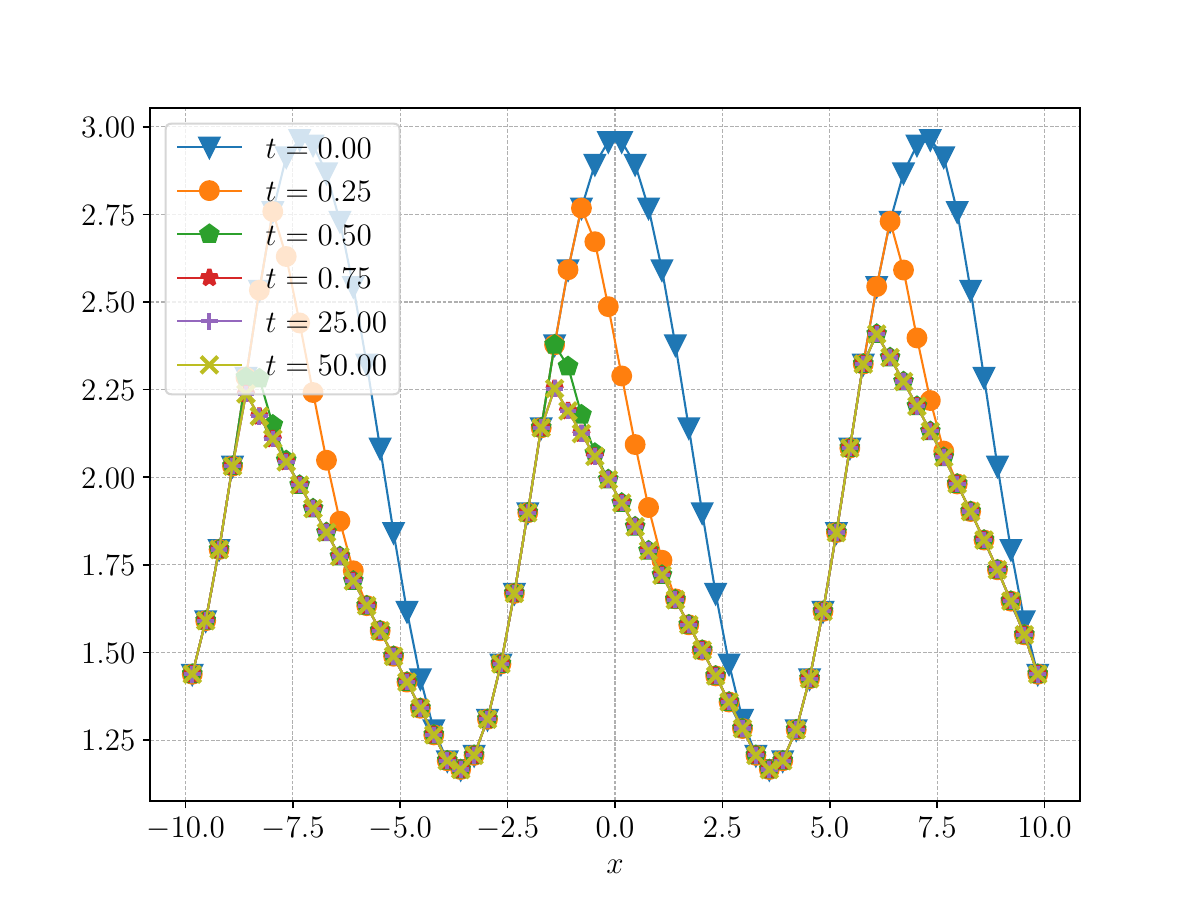}
    \caption{Evolution of $\varphi^\ep(t,x,0)$ for $\ep = 1$ (top left), $\varphi^\ep(t,x,0)$ for $\ep = 0.1$ (top right), and $\varphi(t,x,0)$ (bottom), as functions of $x$ at different times.}
    \label{fig:Phi_xLTB}
\end{figure}

\begin{figure}
    \centering
    \includegraphics[width=.45\linewidth]{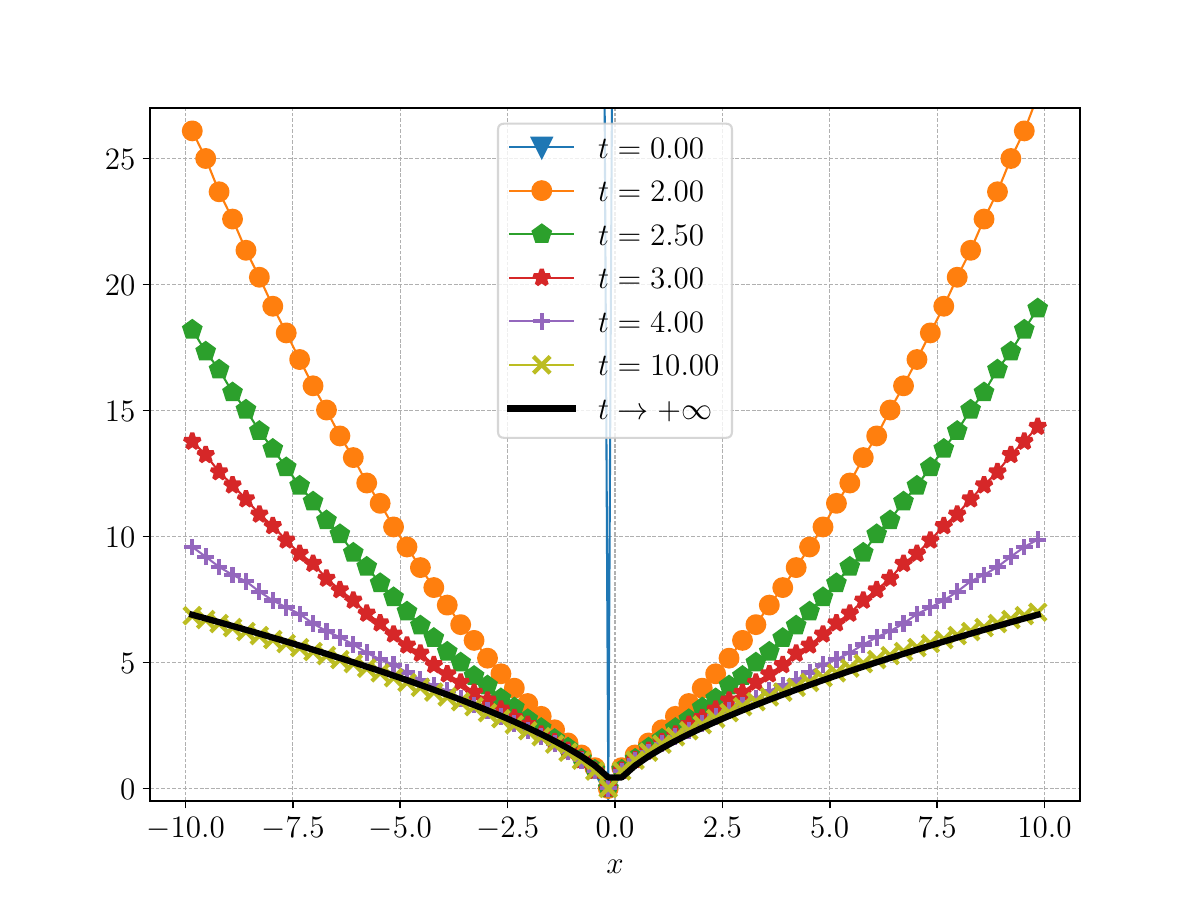}
    \includegraphics[width=.45\linewidth]{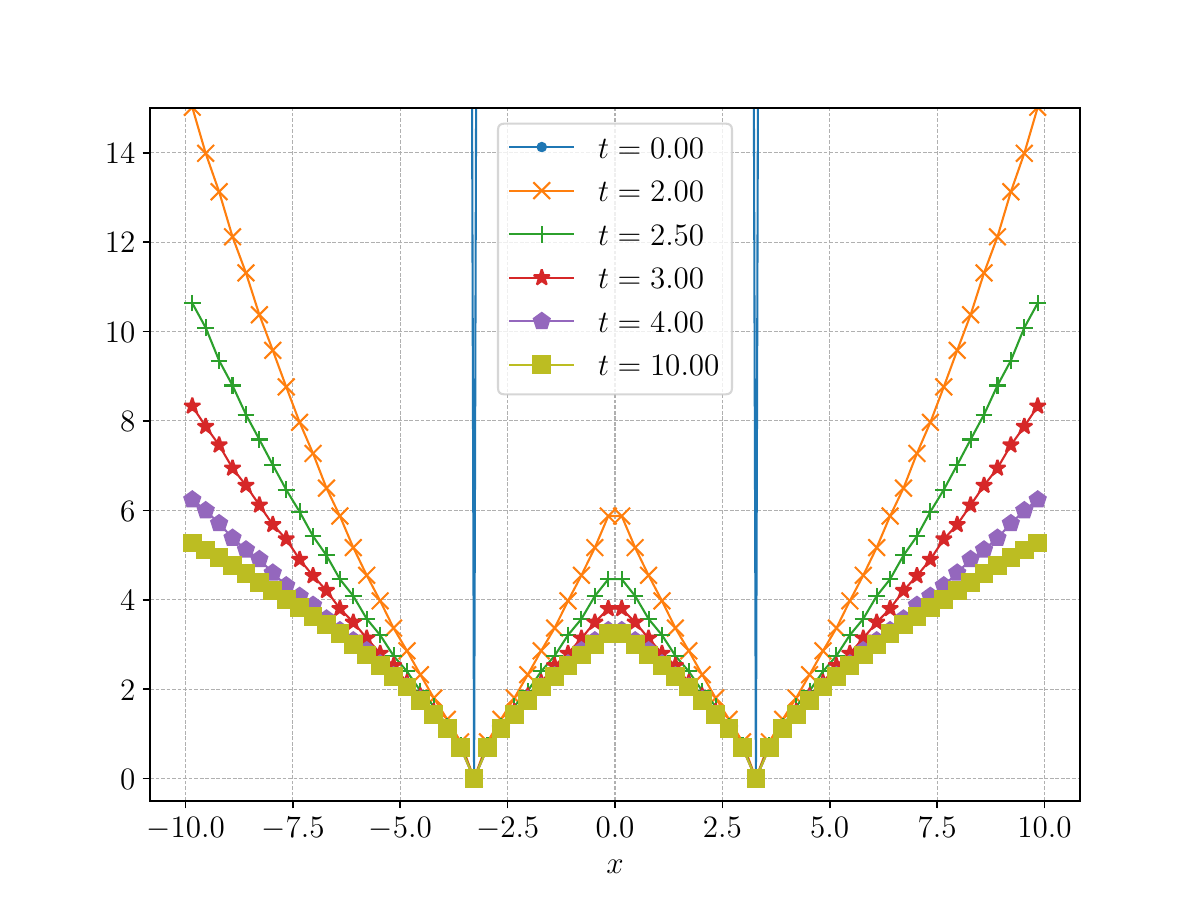}\\
    \caption{Evolution of the logarithm of one (left) and two (right) Dirac masses, as functions of $x$ at different times.}
    \label{fig:Phi_xLTB_Dirac}
\end{figure}

\begin{figure}
    \centering
    \includegraphics[width=.45\linewidth]{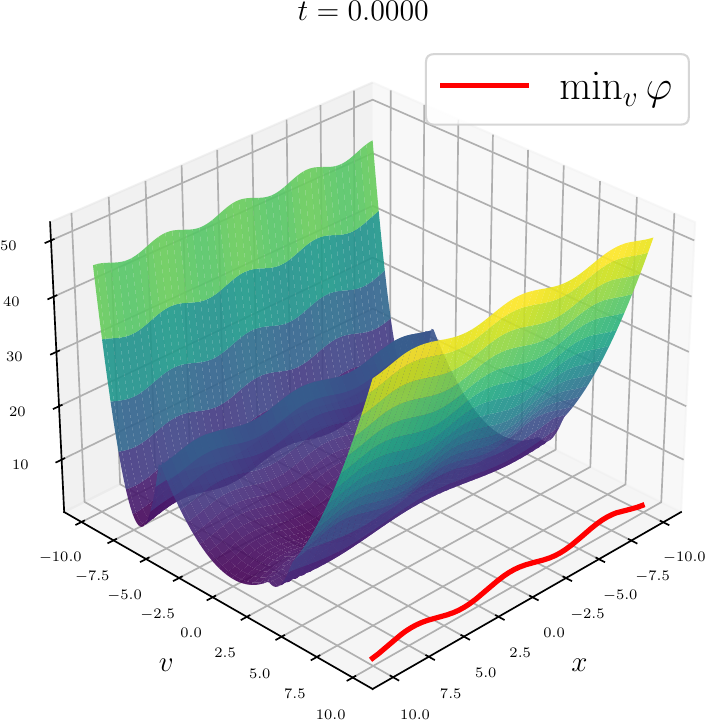}
    \includegraphics[width=.45\linewidth]{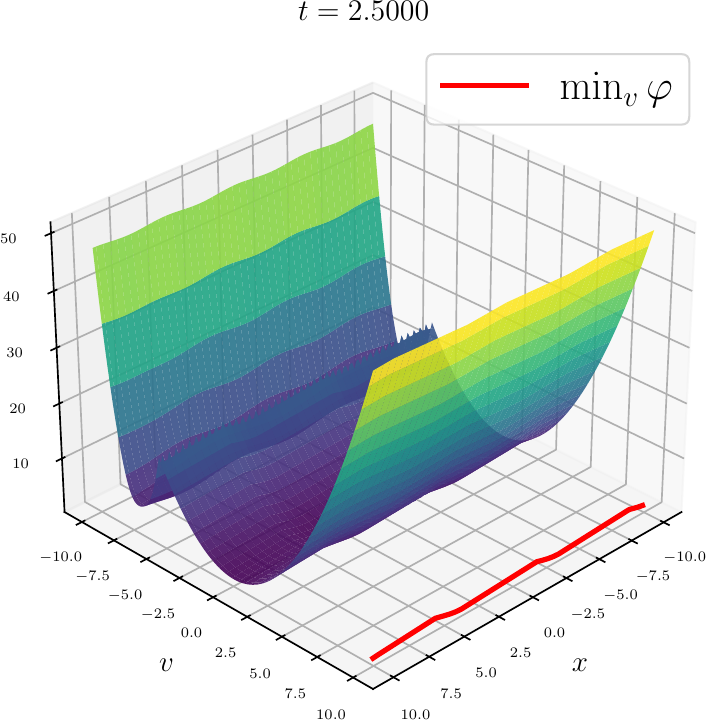}\\\vspace{.5cm}
    \includegraphics[width=.45\linewidth]{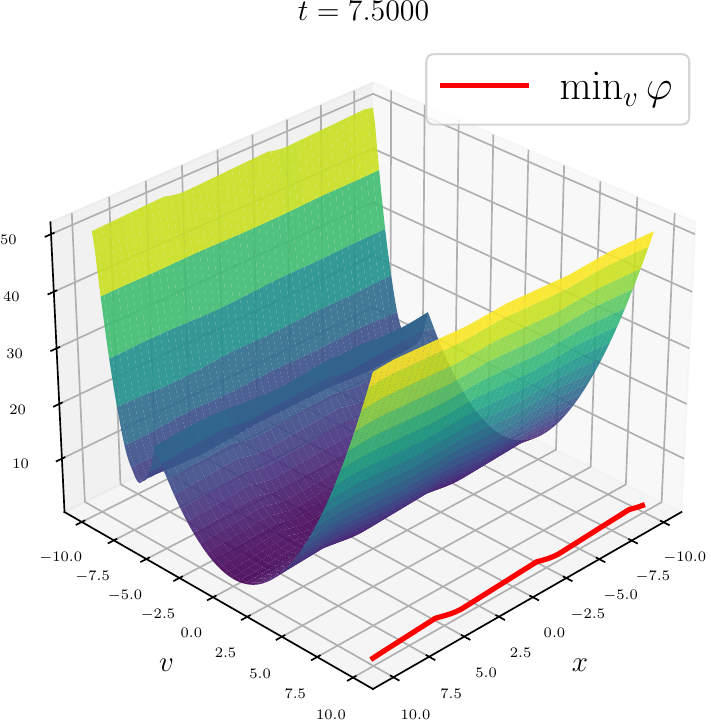}
    \includegraphics[width=.45\linewidth]{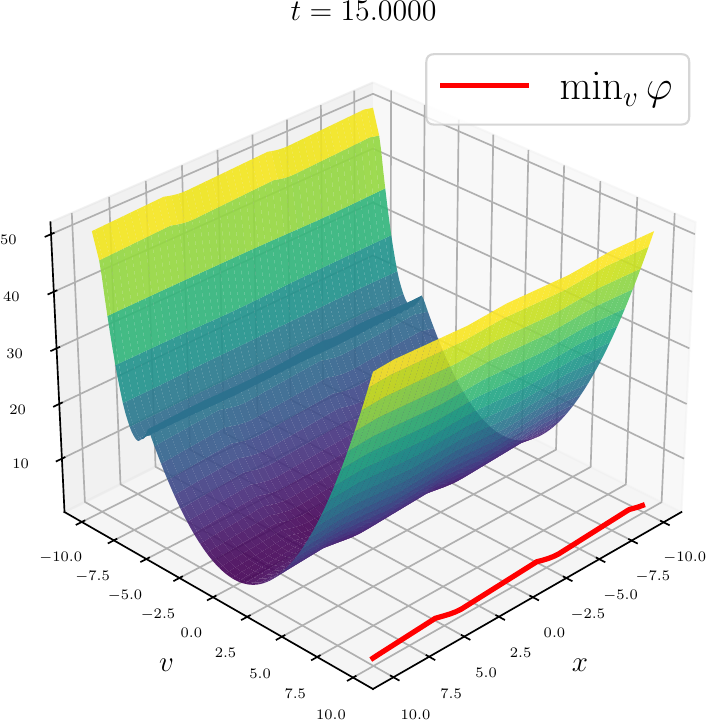}\\\vspace{.5cm}
    \includegraphics[width=.45\linewidth]{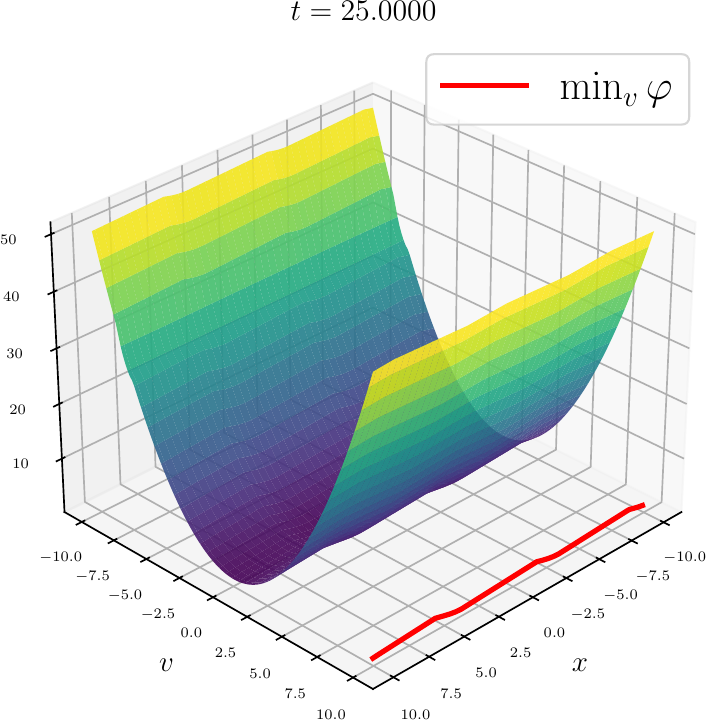}
    \includegraphics[width=.45\linewidth]{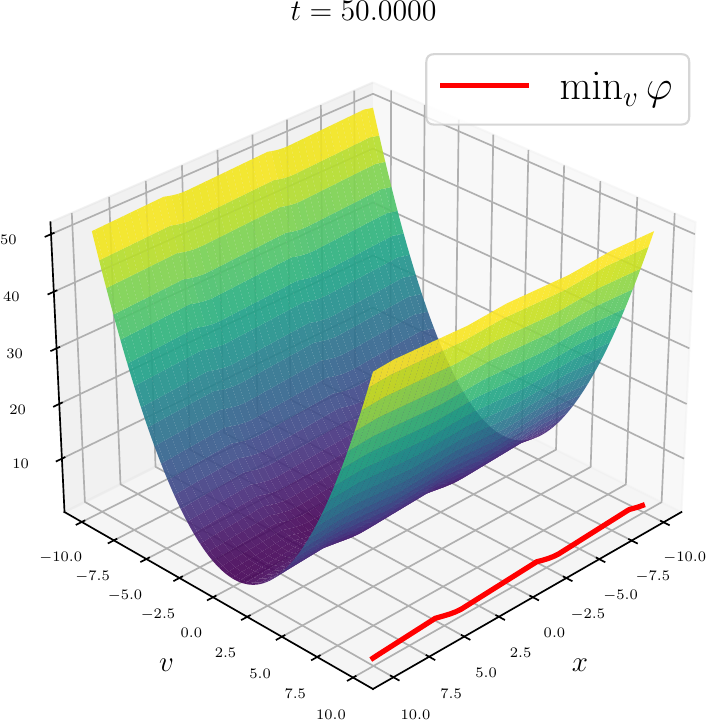}\\
    \caption{Evolution of the solution to \eqref{eq:VariDiscFull} in the $(x,v)$-phase-space.}
    \label{fig:EvolutionXVlim}
\end{figure}

\section*{Acknowledgment}
The authors would like to warmly thank Vincent Calvez for the fruitful discussions and insights on this project.

The authors have received funding from the European Research Council (ERC) under the European Union’s Horizon 2020 research and innovation program (grant agreement No 865711). 

\bibliographystyle{acm}
\bibliography{BiblioHJ}

\end{document}